\documentclass[12pt, a4paper]{amsart}
\usepackage{amsmath,amsfonts,amssymb,amsthm}
\usepackage[shortalphabetic,abbrev,nobysame]{amsrefs}
\usepackage{mathrsfs}
\usepackage[all]{xy}
 \usepackage{relsize}
\usepackage{hyperref}
\usepackage{accents}
\usepackage{xcolor}
\usepackage{soul}\usepackage{calligra}

\allowdisplaybreaks

\DeclareMathAlphabet{\mathcalligra}{T1}{calligra}{m}{n}

\DeclareMathOperator{\Res}{Res}
\DeclareMathOperator{\Hom}{Hom}
\DeclareMathOperator{\Sing}{Sing}

\DeclareMathOperator{\Aut}{Aut}

\DeclareMathOperator{\End}{End}

\begin{document}

\newtheorem{thm}{Theorem}[section]
\newtheorem{prop}[thm]{Proposition}
\newtheorem{coro}[thm]{Corollary}
\newtheorem{conj}[thm]{Conjecture}
\newtheorem{example}[thm]{Example}
\newtheorem{lem}[thm]{Lemma}
\newtheorem{rem}[thm]{Remark}
\newtheorem{hy}[thm]{Hypothesis}
\newtheorem*{acks}{Acknowledgements}
\theoremstyle{definition}
\newtheorem{de}[thm]{Definition}
\newtheorem{ex}[thm]{Example}

\newtheorem{convention}[thm]{Convention}

\newtheorem{bfproof}[thm]{{\bf Proof}}
\xymatrixcolsep{5pc}

\newcommand{\C}{{\mathbb{C}}}
\newcommand{\Z}{{\mathbb{Z}}}
\newcommand{\N}{{\mathbb{N}}}
\newcommand{\Q}{{\mathbb{Q}}}
\newcommand{\te}[1]{\mbox{#1}}
\newcommand{\set}[2]{{
    \left.\left\{
        {#1}
    \,\right|\,
        {#2}
    \right\}
}}
\newcommand{\sett}[2]{{
    \left\{
        {#1}
    \,\left|\,
        {#2}
    \right\}\right.
}}

\newcommand{\choice}[2]{{
\left[
\begin{array}{c}
{#1}\\{#2}
\end{array}
\right]
}}
\def \<{{\langle}}
\def \>{{\rangle}}

\def\({\left(}

\def\){\right)}

\def \:{\mathopen{\overset{\circ}{
    \mathsmaller{\mathsmaller{\circ}}}
    }}
\def \;{\mathclose{\overset{\circ}{\mathsmaller{\mathsmaller{\circ}}}}}

\newcommand{\overit}[2]{{
    \mathop{{#1}}\limits^{{#2}}
}}
\newcommand{\belowit}[2]{{
    \mathop{{#1}}\limits_{{#2}}
}}

\newcommand{\wt}[1]{\widetilde{#1}}

\newcommand{\wh}[1]{\widehat{#1}}

\newcommand{\no}[1]{{
    \mathopen{\overset{\circ}{
    \mathsmaller{\mathsmaller{\circ}}}
    }{#1}\mathclose{\overset{\circ}{\mathsmaller{\mathsmaller{\circ}}}}
}}

\newlength{\dhatheight}
\newcommand{\dwidehat}[1]{%
    \settoheight{\dhatheight}{\ensuremath{\widehat{#1}}}%
    \addtolength{\dhatheight}{-0.45ex}%
    \widehat{\vphantom{\rule{1pt}{\dhatheight}}%
    \smash{\widehat{#1}}}}
\newcommand{\dhat}[1]{%
    \settoheight{\dhatheight}{\ensuremath{\hat{#1}}}%
    \addtolength{\dhatheight}{-0.35ex}%
    \hat{\vphantom{\rule{1pt}{\dhatheight}}%
    \smash{\hat{#1}}}}

\newcommand{\dwh}[1]{\dwidehat{#1}}

\newcommand{\ck}[1]{\check{#1}}

\newcommand{\dis}{\displaystyle}

\newcommand{\pd}[1]{{\partial {#1}}}
\newcommand{\hpd}[1]{{\partial_{(#1)}}}
\newcommand{\pdiff}[2]{\frac{\partial^{#2}}{\partial #1^{#2}}}


\newcommand{\g}{{\frak g}}
\newcommand{\fg}{\g}
\newcommand{\ff}{{\frak f}}
\newcommand{\f}{\ff}
\newcommand{\gc}{{\bar{\g'}}}
\newcommand{\h}{{\frak h}}
\newcommand{\cent}{{\frak c}}
\newcommand{\notc}{{\not c}}
\newcommand{\Loop}{{\mathcal L}}
\newcommand{\G}{{\mathcal G}}
\newcommand{\D}{\mathcal D}
\newcommand{\T}{\mathcal T}
\newcommand{\Free}{\mathcal F}
\newcommand{\Cfk}{\mathcal C}
\newcommand{\nil}{\mathfrak n}
\newcommand{\al}{\alpha}
\newcommand{\alck}{\al^\vee}
\newcommand{\be}{\beta}
\newcommand{\beck}{\be^\vee}
\newcommand{\ssl}{{\mathfrak{sl}}}
\newcommand{\id}{\te{id}}
\newcommand{\rtu}{{\xi}}
\newcommand{\period}{{N}}
\newcommand{\half}{{\frac{1}{2}}}
\newcommand{\quar}{{\frac{1}{4}}}
\newcommand{\oct}{{\frac{1}{8}}}
\newcommand{\hex}{{\frac{1}{16}}}
\newcommand{\reciprocal}[1]{{\frac{1}{#1}}}
\newcommand{\inverse}{^{-1}}
\newcommand{\inv}{\inverse}
\newcommand{\SumInZm}[2]{\sum\limits_{{#1}\in\Z_{#2}}}
\newcommand{\uce}{{\mathfrak{uce}}}
\newcommand{\Rcat}{\mathcal R}
\newcommand{\cS}{{\mathcal{S}}}


\newcommand{\orb}[1]{|\mathcal{O}({#1})|}
\newcommand{\up}{_{(p)}}
\newcommand{\uq}{_{(q)}}
\newcommand{\upq}{_{(p+q)}}
\newcommand{\uz}{_{(0)}}
\newcommand{\uk}{_{(k)}}
\newcommand{\nsum}{\SumInZm{n}{\period}}
\newcommand{\ksum}{\SumInZm{k}{\period}}
\newcommand{\overN}{\reciprocal{\period}}
\newcommand{\df}{\delta\left( \frac{\xi^{k}w}{z} \right)}
\newcommand{\dfl}{\delta\left( \frac{\xi^{\ell}w}{z} \right)}
\newcommand{\ddf}{\left(D\delta\right)\left( \frac{\xi^{k}w}{z} \right)}

\newcommand{\ldfn}[1]{{\left( \frac{1+\xi^{#1}w/z}{1-{\xi^{#1}w}/{z}} \right)}}
\newcommand{\rdfn}[1]{{\left( \frac{{\xi^{#1}w}/{z}+1}{{\xi^{#1}w}/{z}-1} \right)}}
\newcommand{\ldf}{{\ldfn{k}}}
\newcommand{\rdf}{{\rdfn{k}}}
\newcommand{\ldfl}{{\ldfn{\ell}}}
\newcommand{\rdfl}{{\rdfn{\ell}}}

\newcommand{\kprod}{{\prod\limits_{k\in\Z_N}}}
\newcommand{\lprod}{{\prod\limits_{\ell\in\Z_N}}}
\newcommand{\E}{{\mathcal{E}}}
\newcommand{\F}{{\mathcal{F}}}

\newcommand{\Etopo}{{\mathcal{E}_{\te{topo}}}}

\newcommand{\Ye}{{\mathcal{Y}_\E}}

\newcommand{\rh}{{{\bf h}}}
\newcommand{\rp}{{{\bf p}}}
\newcommand{\rrho}{{{\pmb \varrho}}}
\newcommand{\ral}{{{\pmb \al}}}

\newcommand{\comp}{{\mathfrak{comp}}}
\newcommand{\ctimes}{{\widehat{\boxtimes}}}
\newcommand{\ptimes}{{\widehat{\otimes}}}
\newcommand{\ptimeslt}{{
{}_{\te{t}}\ptimes
}}
\newcommand{\ptimesrt}{{\ot_{\te{t}} }}
\newcommand{\ttp}[1]{{
    {}_{{#1}}\ptimes
}}
\newcommand{\bigptimes}{{\widehat{\bigotimes}}}
\newcommand{\bigptimeslt}{{
{}_{\te{t}}\bigptimes
}}
\newcommand{\bigptimesrt}{{\bigptimes_{\te{t}} }}
\newcommand{\bigttp}[1]{{
    {}_{{#1}}\bigptimes
}}

\newcommand{\ot}{\otimes}
\newcommand{\Ot}{\bigotimes}

\newcommand{\affva}[1]{V_{\wh\g}\(#1,0\)}
\newcommand{\saffva}[1]{L_{\wh\g}\(#1,0\)}
\newcommand{\saffmod}[1]{L_{\wh\g}\(#1\)}

\newcommand{\otcopies}[2]{\belowit{\underbrace{{#1}\ot \cdots \ot {#1}}}{{#2}\te{-times}}}

\newcommand{\twotcopies}[3]{\belowit{\underbrace{{#1}\wh\ot_{#2} \cdots \wh\ot_{#2} {#1}}}{{#3}\te{-times}}}


\newcommand{\tar}{{\mathcal{DY}}_0\(\mathfrak{gl}_{\ell+1}\)}
\newcommand{\U}{{\mathcal{U}}}
\newcommand{\htar}{\mathcal{DY}_\hbar\(A\)}
\newcommand{\hhtar}{\widetilde{\mathcal{DY}}_\hbar\(A\)}
\newcommand{\htarz}{\mathcal{DY}_0\(\mathfrak{gl}_{\ell+1}\)}
\newcommand{\hhtarz}{\widetilde{\mathcal{DY}}_0\(A\)}
\newcommand{\qhei}{\U_\hbar\left(\hat{\h}\right)}
\newcommand{\n}{{\mathfrak{n}}}
\newcommand{\vac}{{{\bf 1}}}
\newcommand{\vtar}{{{
    \mathcal{V}_{\hbar,\tau}\left(\ell,0\right)
}}}

\newcommand{\qtar}{
    \U_q\(\wh\g_\mu\)}
\newcommand{\rk}{{\bf k}}

\newcommand{\hctvs}[1]{Hausdorff complete linear topological vector space}
\newcommand{\hcta}[1]{Hausdorff complete linear topological algebra}
\newcommand{\ons}[1]{open neighborhood system}
\newcommand{\B}{\mathcal{B}}
\newcommand{\rx}{{\bf x}}
\newcommand{\re}{{\bf e}}
\newcommand{\rphi}{{\boldsymbol{ \phi}}}

\newcommand{\der}{\mathcal D}

\newcommand{\prodlim}{\mathop{\prod_{\longrightarrow}}\limits}

\newcommand{\lp}[1]{\mathcal L(f)}

\newcommand{\cha}{\check a}
\newcommand{\chh}{\check \h}


\makeatletter
\renewcommand{\BibLabel}{%
    \Hy@raisedlink{\hyper@anchorstart{cite.\CurrentBib}\hyper@anchorend}%
    [\thebib]%
}
\@addtoreset{equation}{section}
\def\theequation{\thesection.\arabic{equation}}
\makeatother \makeatletter

\newcommand{\qtarl}{\U^l_q(\wh\g_\mu)}

\title[Quantum lattice VA]{Twisted quantum affine algebras and equivariant $\phi$-coordinated modules for quantum vertex algebras}

\author{Naihuan Jing$^1$}
\address{Department of Mathematics, North Carolina State University, Raleigh, NC 27695,
USA}
\email{jing@math.ncsu.edu}
\thanks{$^1$Partially supported by NSF of China (No.12171303) and Simons Foundation (No.523868).}

\author{Fei Kong}
\address{Key Laboratory of Computing and Stochastic Mathematics (Ministry of Education), School of Mathematics and Statistics, Hunan Normal University, Changsha, China 410081} \email{kongmath@hunnu.edu.cn}

\author{Haisheng Li}
 \address{Department of Mathematical Sciences, Rutgers University, Camden, NJ 08102}
 \email{hli@camden.rutgers.edu}

 \author{Shaobin Tan$^2$}
 \address{Department of Mathematics, Xiamen University,
 Xiamen, China 361005} \email{tans@xmu.edu.cn}
 \thanks{$^2$Partially supported by NSF of China (Nos.12131018 and 12161141001).}

\begin{abstract}
This paper is about establishing a natural connection of quantum affine algebras with quantum vertex algebras.
Among the main results,  we establish $\hbar$-adic versions of the smash product construction of
 quantum vertex algebras and their $\phi$-coordinated quasi modules, which were obtained before in a sequel,
we construct a family of $\hbar$-adic quantum vertex algebras $V_L[[\hbar]]^{\eta}$
as deformations of the lattice vertex algebras $V_L$, and establish
a natural connection between twisted quantum affine algebras of type $A, D, E$
and equivariant $\phi$-coordinated quasi modules for the $\hbar$-adic quantum vertex algebras $V_L[[\hbar]]^{\eta}$
with certain specialized $\eta$.
\end{abstract}

\maketitle

\section{Introduction}
This paper is to study $\hbar$-adic quantum vertex algebras and their $\phi$-coordinated quasi modules,
where the main goal is to establish a natural association of $\hbar$-adic quantum vertex algebras to quantum affine algebras.

In the general field of vertex algebras, one conceptual problem is to develop suitable quantum vertex algebra theories
so that quantum vertex algebras can be naturally associated to certain quantum algebras such as quantum affine algebras.
As one of the fundamental works, Etingof and Kazhdan  (see \cite{EK-qva}) developed
a theory of quantum vertex operator algebras in the sense of formal deformation, where
the underlying space of a quantum vertex operator algebra is a topologically free $\C[[\hbar]]$-module.
As for the vertex algebra-like structure, while the (weak) associativity stays in place,
the usual locality, namely weak commutativity, is replaced with what was called $\mathcal{S}$-locality.

Partly motivated by the work of Etingof and Kazhdan, the third named author of this current paper
conducted a series of studies. During these studies,  various (new) notions were introduced, including nonlocal vertex algebra,
weak quantum vertex algebra, quasi module, $\phi$-coordinated quasi module, and $\hbar$-adic counterparts of these.
While vertex algebras are analogues of commutative and associative algebras,
nonlocal vertex algebras (see \cite{bk}, \cite{li-g1}, \cite{Li-nonlocal}) are analogues of noncommutative associative algebras.
Weak quantum vertex algebras are nonlocal vertex algebras which satisfy Etingof and Kazhdan's $\mathcal{S}$-locality.
In a general picture, the notion of weak quantum vertex algebra (see \cite{Li-nonlocal})
generalizes the notion of vertex super-algebra,
while the classical limits $V/\hbar V$ of $\hbar$-adic (weak) quantum vertex algebras $V$ (see \cite{Li-h-adic})
are (weak) quantum vertex algebras. In this framework, the notion of $\hbar$-adic quantum vertex algebra
slightly generalizes Etingof and Kazhdan's notion of quantum vertex operator algebra where
the classical limit is assumed to be a vertex algebra.

On the representation side, the notion of quasi module was introduced in \cite{Li-new-construction}
in order to associate vertex algebras to Lie algebras of a certain type.
Quasi modules for vertex algebras still enjoy (weak) associativity, whereas locality is replaced with ``quasi locality.''
While this notion generalizes that of module naturally, it has a close connection with twisted module
 (see \cite{Li-quasi-twisted}). Later, mainly in order to associate quantum vertex algebras
 to algebras like quantum affine algebras, a theory of what were called $\phi$-coordinated quasi modules
 for general nonlocal vertex algebras was developed in \cite{Li-phi-coor, Li-G-phi},
where $\phi$ is an associate of the $1$-dimensional additive formal group as we mention next.

It was known that the $1$-dimensional additive formal group (law) $F_{a}(x,y):=x+y\in \C[[x,y]]$
plays (implicitly) an important role  in vertex algebra theory.
(This can be seen from the associativity or the physics operator product expansion.)
 A notion of what was called associate of $F_{a}(x,y)$ was introduced in \cite{Li-phi-coor}, where
an associate is a formal series $\phi(x,z)\in \C((x))[[z]]$ such that
$$\phi(x,0)=x,\    \   \  \phi( \phi(x,y),z)=\phi(x,y+z).$$
It was proved therein that for any $p(x)\in \C((x))$, the formal series $\phi(x,z):=e^{zp(x)\frac{d}{dx}}x$
 is an associate, and every associate is of this form with $p(x)$ uniquely determined.
The essence of \cite{Li-phi-coor} is that for each associate $\phi$,
a theory  of what were called $\phi$-coordinated quasi modules for a general nonlocal vertex algebra was developed,
where a module is simply a $\phi$-coordinated module with $\phi(x,z)$ taken to be
the formal group itself---$F_{a}(x,z)$.

In another direction, in seeking for effective tools to construct desired quantum vertex algebras,
we came to certain classical constructions.
As one of such attempts, a vertex-algebra analogue of the smash product construction
in Hopf algebra theory was obtained in \cite{Li-smash}.
Specifically, a notion of nonlocal vertex bialgebra  and
a notion of (left) $H$-module nonlocal vertex algebra for a nonlocal vertex bialgebra $H$ were introduced, and then
the smash product $V\sharp H$ of an $H$-module nonlocal vertex algebra $V$ with $H$ was constructed.
This construction was used therein to give a variant construction of lattice vertex algebras and their modules.

The study on smash product construction was continued in \cite{JKLT-Defom-va},
where a notion of right $H$-comodule nonlocal vertex algebra
for a nonlocal vertex bialgebra $H$ was introduced.
Let $V$ be an $H$-module nonlocal vertex algebra.
Assume that $V$ is also a right $H$-comodule nonlocal vertex algebra such that
the right $H$-comodule structure map $\rho: V\rightarrow V\ot H$ is ``compatible'' with the $H$-module structure on $V$. Then
a new nonlocal vertex algebra ${\mathfrak{D}}_{\rho}(V)$ with $V$ as the underlying space was obtained.
It was also proved that under a certain condition, ${\mathfrak{D}}_{\rho}(V)$ is a quantum vertex algebra.
On the other hand, it was proved that for any $\phi$-coordinated $V$-module $W$ with a ``compatible'' $\phi$-coordinated $H$-module
structure on $W$, there exists a $\phi$-coordinated ${\mathfrak{D}}_{\rho}(V)$-module structure on $W$.
Furthermore, these general results were applied to lattice vertex algebras $V_L$,
to obtain a family of quantum vertex algebras $V_L^{\eta}$ and their $\phi$-coordinated quasi modules.

In the current paper, we continue the study of \cite{JKLT-Defom-va}, to establish $\hbar$-adic versions
of the basic constructions and results, while one of the main goals is to associate $\hbar$-adic quantum vertex algebras
to twisted (and untwisted) quantum affine algebras.
Among the main results, we construct a family of $\hbar$-adic quantum vertex algebras $V_L[[\hbar]]^{\eta}$
as deformations of the lattice vertex algebras $V_L$ and establish
a natural connection between twisted quantum affine algebras of type $A, D, E$
and equivariant $\phi$-coordinated quasi modules for $V_L[[\hbar]]^{\eta}$
with suitably defined $\eta$.

We mention that there are several closely related studies using different approaches.
In the pioneer work \cite{EK-qva}, Etingof and Kazhdan constructed quantum vertex operator algebras
 as formal deformations of type $A$ universal affine vertex algebras
by using the $R$-matrix type relations in \cite{RS-RTT}.
Later, Butorac, Jing and Ko\v{z}i\'{c}  (see \cite{BJK-qva-BCD}) extended Etingof-Kazhdan's construction to
types $B$, $C$ and $D$ rational $R$-matrices, obtaining $\hbar$-adic quantum vertex algebras.
It was proved therein that modules for these $\hbar$-adic quantum vertex algebras
are in one-to-one correspondence with restricted modules for the corresponding Yangian doubles.
Recently, based on the $R$-matrix presentation of quantum affine algebras of classical types
(see \cite{DF-qaff-RTT-Dr}, \cite{JLM-qaff-RTT-Dr-BD}, \cite{JLM-qaff-RTT-Dr-C}),
Ko\v{z}i\'{c} in \cite{Kozic-qva-tri-A, K-qva-phi-mod-BCD} constructed $\hbar$-adic quantum vertex algebras
by using trigonometric $R$-matrices and established a one-to-one correspondence
between $\phi$-coordinated modules for the $\hbar$-adic quantum vertex algebras
and restricted modules for the quantum affine algebras.

Now, we continue the introduction to provide some more detailed information.
First, we give more information about the basic concepts which are somewhat less well known.
By definition, a nonlocal vertex bialgebra is a nonlocal vertex algebra $V$
equipped with a classical coalgebra structure on $V$ such that
the comultiplication $\Delta$ and the counit $\epsilon$ are homomorphisms of nonlocal vertex algebras.
For a nonlocal vertex bialgebra $H$, a (left) $H$-module nonlocal vertex algebra is a nonlocal vertex algebra $V$
equipped with a module structure for $H$ viewed as a nonlocal vertex algebra, satisfying certain compatibility conditions.
On the other hand,  a right $H$-comodule nonlocal vertex algebra
is a nonlocal vertex algebra $V$ equipped with a comodue structure
$\rho: V\rightarrow V\otimes H$ for $H$ viewed as a coalgebra such that
$\rho$ is a homomorphism of nonlocal vertex algebras.

Next, we mention lattice vertex algebras which play a major role in this paper.
Let $L$ be a non-degenerate even lattice and let $\varepsilon(\cdot,\cdot): L\times L\rightarrow \C^{\times}$
be a $2$-cocycle satisfying a certain normalizing condition.
Denote by $\C_{\varepsilon}[L]$ the $\varepsilon$-twisted group algebra of $L$, which
has a designated basis $\{ e_{\alpha}\ |\ \alpha\in L\}$ with
$$e_{\alpha}\cdot e_{\beta}=\varepsilon(\alpha,\beta)e_{\alpha+\beta}\quad \text{ for }\alpha,\beta\in L.$$
Set $\h=\C\otimes_{\Z}L$ and then set $\widehat{\h}^{-}=\h\otimes t^{-1}\C[t^{-1}]$.
The underlying vector space of the lattice vertex algebra $V_L$ is $S(\widehat{\h}^{-})\otimes \C_{\varepsilon}[L]$.
A well known fact is that when $L$ is taken to be the root lattice of
a finite-dimensional simple Lie algebra $\mathfrak{g}$ of type $A$, $D$, or $E$,
$V_L$ affords a (level one) basic module structure for the untwisted affine Kac-Moody algebra $\widehat{\mathfrak{g}}$.

On the other hand, consider $B_L:=S(\widehat{\h}^{-})\otimes \C[L]$, which is a commutative and associative algebra,
where $\C[L]$ denotes the group algebra of $L$. Note that $B_L$ is naturally a Hopf algebra, in particular, a bialgebra.
The algebra $B_L$ admits a derivation $\partial$ uniquely determined by
$$\partial a(-n)=na(-n-1),\quad \partial (e^{\alpha})=\alpha(-1)e^{\alpha} \quad \text{ for }a\in \h,\ n\in \Z_{+},\ \alpha\in L,$$
where $a(-n)=a\otimes t^{-n}$.
Then $B_L$ becomes a commutative vertex algebra through Borcherds' construction.
This makes $B_L$ a commutative and cocommutative vertex bialgebra, which was exploited in \cite{Li-smash} and \cite{JKLT-Defom-va}.

Let $\pi: V_L\rightarrow B_L$ be the natural vector space isomorphism (with $\pi(e_{\alpha})=e^{\alpha}$ for $\alpha\in L$).
Set $\rho=(\pi^{-1}\otimes 1)\Delta\circ \pi$, a linear map from $V_L$ to $V_L\otimes  B_L$,
where $\Delta: B_L\rightarrow B_L\otimes B_L$ is the comultiplication of $B_L$.
It was proved in \cite{JKLT-Defom-va} that $V_L$ with  $\rho: V_L\rightarrow V_L\otimes  B_L$ is a right $B_L$-comodule vertex algebra.
On the other hand,
for any linear map $\eta(\cdot,x): \h\rightarrow \h\otimes x\C[[x]]$,  a $B_L$-module structure
$Y_M^{\eta}(\cdot,x)$ on $V_L$ was obtained, which makes $V_L$ a $B_L$-module vertex algebra.
Then we obtained a quantum vertex algebra ${\mathfrak{D}}_{\rho}^{\eta}(V_L)$.
Assuming that $\eta$ satisfies certain conditions, from a $\phi$-coordinated quasi $V_L$-module $W$,
 we obtained a $\phi$-coordinated quasi ${\mathfrak{D}}_{\rho}^{\eta}(V_L)$-module structure on $W$.

In this current paper,  we formulate $\hbar$-adic versions of the aforementioned results
and then make use of them to study twisted quantum affine algebras.
Note that $V_L[[\hbar]]$ is an $\hbar$-adic vertex algebra while $B_L[[\hbar]]$
is a commutative $\hbar$-adic vertex bialgebra.
Assume $\eta(\cdot,x):\h \rightarrow \h\otimes \C((x))[[\hbar]]$ is a linear map such that $\eta_0(a,x)\in \h\ot x\C[[x]]$ for $a\in \h$,
where $\eta(a,x)=\sum_{n\ge 0}\eta_n(a,x)\hbar^n$ for $a\in \h$.
Then we obtain an $\hbar$-adic quantum vertex algebra $V_L[[\hbar]]^{\eta}$ with $V_L[[\hbar]]$ as its underlying space.
For our consideration on twisted quantum affine algebras, we assume that $L$ is the root lattice of a simple Lie algebra
$\mathfrak{g}$ of type $A, D$, or $E$, and let $\mu$ be a symmetry of the Dynkin diagram of $\mathfrak{g}$.
The twisted quantum affine algebra $\qtar$ associated to $\mu$ was introduced by Drinfeld (see \cite{Dr-new}, cf. \cite{J-inv}).
By using $\mu$ and some structure data of $\qtar$, we then define a particular linear map
  $\eta(\cdot,x):\h \rightarrow \h\otimes \C((x))[[\hbar]]$.
Finally, we prove that every ``equivariant'' $\phi$-coordinated quasi $V_L[[\hbar]]^{\eta}$-module
is naturally a module for the twisted quantum affine algebra $\qtar$.

This paper is organized as follows: Section 2 is preliminary; We recall some necessary notions and basic results on
nonlocal vertex algebras, (weak) quantum vertex algebras, and $\phi$-coordinated quasi modules.
In Section 3, we study $\hbar$-adic (weak) quantum vertex algebras and their $\phi$-coordinated quasi modules.
In Section 4, we study smash product $\hbar$-adic nonlocal vertex algebras and $\phi$-coordinated quasi modules.
In Section 5, we first recall the lattice vertex algebra $V_L$ and vertex bialgebra $B_L$,
and then study $\hbar$-adic quantum vertex algebras $V[[\hbar]]^{\eta}$.
In Section 6, we recall the twisted quantum affine algebra $\qtar$ and give a variant presentation, and then state the main result
which states that every equivariant $\phi$-coordinated quasi module for $V[[\hbar]]^{\eta}$
with specially defined $\eta$ is naturally a  $\qtar$-module.
Section 7 is devoted to the proof of the main result.

In this paper, we use the formal variable notations and conventions as established in \cite{FLM2} and \cite{fhl}.
In addition, we use $\Z_+$ and $\N$ for the sets of positive integers and nonnegative integers, respectively.
For a commutative ring $R$,  $R(x), R(x_1,x_2)$ denote the rings of rational functions, and $R^{\times}$ denotes the set of invertible elements.
For a variable $z$, we denote by $\partial_{z}$ the partial differential operator $\frac{\partial}{\partial z}$.

\section{Nonlocal vertex algebras and their $\phi$-coordinated quasi modules}

This section is preliminary, where we recall the basic notions and results on nonlocal vertex algebras
and their $\phi$-coordinated quasi modules mostly from  \cite{Li-nonlocal, Li-phi-coor}.
For this paper, the scalar field is assumed to be the complex number field $\C$ unless it is stated otherwise.

\subsection{Nonlocal vertex algebras and quantum vertex algebras}
We start with the notion of nonlocal vertex algebra.

\begin{de}\label{def-nlva}
A {\em nonlocal vertex algebra} is a vector space $V$ equipped with a linear map
\begin{eqnarray*}
Y(\cdot,x): &&  V\rightarrow (\te{End} V)[[x,x^{-1}]]\nonumber\\
&&v\mapsto Y(v,x)=\sum_{n\in \Z}v_nx^{-n-1}\  \  (\mbox{where }v_n\in \te{End }V)
\end{eqnarray*}
and a distinguished vector ${\bf 1}\in V$, satisfying the conditions that
\begin{align}
Y(u,x)v\in V((x))\   \   \mbox{ for }u,v\in V,
\end{align}
\begin{align}
Y({\bf 1},x)v=v,\   \   \   \   Y(v,x){\bf 1}\in V[[x]]\   \mbox{ and }\  \lim_{x\rightarrow 0}Y(v,x){\bf 1}=v\   \   \mbox{ for }v\in V,
\end{align}
and that for any $u,v\in V$, there exists a nonnegative integer $k$ such that
\begin{align}\label{nlva-compatibilty}
(x_1-x_2)^{k}Y(u,x_1)Y(v,x_2)\in \Hom(V,V((x_1,x_2)))
\end{align}
and
\begin{align}\label{nlva-associativity}
\left((x_1-x_2)^{k}Y(u,x_1)Y(v,x_2)\right)|_{x_1=x_2+x_0}=
x_0^{k}Y(Y(u,x_0)v,x_2).
\end{align}
\end{de}

From  \cite{JKLT-G-phi-mod} (cf. \cite{DLMi}, \cite{LTW1}) we have:

\begin{lem}\label{two-definitions}
Let $V$ be a nonlocal vertex algebra. Then the following
 {\em weak associativity} holds: For any $u,v,w\in V$, there exists $l\in \N$ such that
 \begin{align}\label{weak-assoc}
 (x_0+x_2)^{l}Y(u,x_0+x_2)Y(v,x_2)w= (x_0+x_2)^{l}Y(Y(u,x_0)v,x_2)w.
\end{align}
\end{lem}

\begin{rem}\label{nlva-definitions}
{\em Note that a notion of nonlocal vertex algebra
was defined in  \cite{Li-nonlocal} in terms of weak associativity,
which is the same as the notion of axiomatic $G_1$-vertex algebra introduced in \cite{li-g1} and
also the same as the notion of field algebra introduced in \cite{bk}.
In view of Lemma \ref{two-definitions}, the notion of nonlocal vertex algebra in the sense of
Definition \ref{def-nlva} is theoretically stronger than the other notions. }
\end{rem}

Let $V$ be a nonlocal vertex algebra. Define a linear operator $\der$ by $\der (v)= v_{-2}\vac$ for $v\in V$.
Then
\begin{align}
[\der, Y(v,x)]=Y(\der (v),x)=\frac{d}{dx}Y(v,x)\   \   \te{for }v\in V.
\end{align}

The following notion singles out a family of nonlocal vertex algebras (see \cite{Li-nonlocal}):

\begin{de}\label{weak-qva}
A {\em weak quantum vertex algebra} is a nonlocal vertex algebra $V$
satisfying the following {\em ${\mathcal{S}}$-locality:}  For any $u,v\in V$, there exist
$$u^{(i)},v^{(i)}\in V,\ f_i(x)\in \C((x)) \  \   (i=1,\dots,r)$$
and a nonnegative integer $k$ such that
\begin{eqnarray}
(x-z)^{k}Y(u,x)Y(v,z)=(x-z)^{k}\sum_{i=1}^{r}f_{i}(z-x)Y(v^{(i)},z)Y(u^{(i)},x).
\end{eqnarray}
\end{de}

The following was proved in \cite{Li-nonlocal}:

\begin{prop}\label{S-Jacobi-def}
A weak quantum vertex algebra can be defined equivalently by replacing the conditions (\ref{nlva-compatibilty})
and (\ref{nlva-associativity}) in Definition \ref{def-nlva} (for  a nonlocal vertex algebra) with the condition
that for any $u,v\in V$, there exist
$$u^{(i)},v^{(i)}\in V,\ f_i(x)\in \C((x)) \  \   (i=1,\dots,r)$$
 such that the following {\em ${\mathcal{S}}$-Jacobi identity} holds:
 \begin{eqnarray} \label{eq:S-Jacobi-id}
 &&x_0^{-1}\delta\left(\frac{x_1-x_2}{x_0}\right)Y(u,x_1)Y(v,x_2)\\
 &&\  \   \   \   -x_0^{-1}\delta\left(\frac{x_2-x_1}{-x_0}\right)\sum_{i=1}^{r}f_i(x_2-x_1)Y(v^{(i)},x_2)Y(u^{(i)},x_1)\nonumber\\
 &&= x_2^{-1}\delta\left(\frac{x_1-x_0}{x_2}\right)Y(Y(u,x_0)v,x_2).\nonumber
 \end{eqnarray}
\end{prop}

A {\em rational quantum Yang-Baxter operator} on a vector space $U$ is a linear map
 $$\cS(x):  U\ot U\rightarrow U\ot U\ot \C((x))$$
 such that
 \begin{eqnarray}
   {\mathcal{S}}^{12}(x)\cS^{13}(x+z)\cS^{23}(z)=\cS^{23}(z)\cS^{13}(x+z)\cS^{12}(x)
 \end{eqnarray}
 (the {\em quantum Yang-Baxter equation}).
Furthermore, $\cS(x)$ is said to be {\em unitary} if
\begin{align}
  \cS^{21}(x)\cS^{12}(-x)=1,
\end{align}
where $\cS^{21}(x)=\sigma \cS(x)\sigma$ with $\sigma$ the flip operator on $U\ot U$.

For a nonlocal vertex algebra $V$, follow \cite{EK-qva} to define a linear map
\begin{align}\label{def-Y(x)}
Y(x):\   V\ot V\rightarrow V((x))
\end{align}
by $Y(x)(u\ot v)=Y(u,x)v$ for $u,v\in V$.

The following notion is essentially due to \cite{EK-qva} (cf.  \cite{Li-nonlocal}):

\begin{de}
 A {\em quantum vertex algebra} is a weak quantum vertex algebra $V$ equipped with
a unitary rational quantum Yang-Baxter operator $\cS(x)$ on $V$, satisfying the following conditions for $u,v\in V$:
\begin{align}
 & \left[\mathcal D\otimes 1,{\mathcal{S}}(x)\right]=-\frac{d}{dx}{\mathcal{S}}(x),\\
 & Y(u,x)v=e^{x\D}Y(-x)\cS(-x)(v\ot u)\quad  (\text{the {\em $\cS$-skew symmetry}}),\label{S-skew-symmetry}\\
&  \cS(x)\(Y(z)\otimes 1\)=\(Y(z)\otimes 1\)\cS^{23}(x)\cS^{13}(x+z)\quad (\text{the {\em hexagon identity}}).
\end{align}
\end{de}

It is known (see \cite{EK-qva}) that the $\cS$-skew symmetry is equivalent to the {\em $\cS$-locality:}
 For any $u,v\in V$, there exists $k\in \N$ such that
\begin{align*}
  (x-z)^kY(x)\(1\otimes Y(z)\)(u\otimes v\otimes w)
=(x-z)^kY(z)\(1\otimes Y(x)\)\(\cS(z-x)(v\otimes u)\otimes w\)
\end{align*}
for all $w\in V$.


A nonlocal vertex algebra $V$ is said to be {\em non-degenerate} (see \cite{EK-qva})
  if for every positive integer $n$, the linear map
 $$Z_n: \C((x_1))\cdots ((x_n))\otimes V^{\otimes n}\rightarrow V((x_1))\cdots ((x_n))),$$
 defined by
 $$Z_n(f\ot v^{(1)}\otimes \cdots \otimes v^{(n)})=f Y(v^{(1)},x_1)\cdots Y(v^{(n)},x_n){\bf 1}$$
for $f\in \C((x_1))\cdots ((x_n)),\  v^{(1)},\dots, v^{(n)}\in V$, is injective.

 We have (see \cite{EK-qva}, \cite{Li-nonlocal}):

 \begin{prop}\label{nondeg-wqva}
 Let $V$ be a weak quantum vertex algebra which is non-degenerate. Then
 there exists a linear map $\cS(x): V\ot V\rightarrow V\ot V\ot \C((x))$,
 which  is uniquely determined by the $\cS$-locality or equivalently by the $\cS$-symmetry.
 Furthermore, $V$ together with $\cS(x)$ is a quantum vertex algebra.
\end{prop}

The following notion plays an important role in this paper (see \cite{Li-new-construction}, \cite{JKLT-G-phi-mod}):

\begin{de}\label{def-chi-G-va}
Let $G$ be a group with a linear character $\chi: G\rightarrow \C^{\times}$.
A {\em $(G,\chi)$-module nonlocal vertex algebra} is a nonlocal vertex algebra $V$ equipped with
a group representation $R:G\rightarrow \te{GL} (V)$ of $G$ on $V$  such that
$R(g)\vac=\vac$ and
\begin{align}
  R(g)Y(v,x)R(g)\inverse =Y(R(g)v,\chi(g)x)\quad\te{for }g\in G,\  v\in V.
\end{align}
We may denote a $(G,\chi)$-module nonlocal vertex algebra by $(V,R)$.
\end{de}

\begin{rem}
{\em Note that a $(G,\chi)$-module nonlocal vertex algebra with $\chi=1$ (the trivial character) is simply
a nonlocal vertex algebra on which  $G$ acts as an automorphism group.
In this case, $V$ is a $\C[G]$-module nonlocal vertex algebra in the sense of \cite{Li-smash} (see Section  3).}
\end{rem}

Let $(U,R_U)$ and $(V,R_V)$ be $(G,\chi)$-module nonlocal vertex algebras.
A {\em $(G,\chi)$-module nonlocal vertex algebra morphism from $U$ to $V$} is a nonlocal vertex algebra morphism
which is also a $G$-module morphism.

 \begin{rem}\label{rnva-on-R}
 {\em Suppose $\mathcal{R}$ is a commutative and associative algebra over $\C$ with $1$.
The notions of  nonlocal vertex algebra and (weak) quantum vertex algebra over $\mathcal{R}$
 are defined in the obvious way.}
 \end{rem}

\subsection{$\phi$-coordinated quasi modules for nonlocal vertex algebras}

We here recall the basic results on (equivariant)  $\phi$-coordinated quasi modules
for nonlocal vertex algebras.

We begin with the {\em one-dimensional additive formal group (law)} $F_a(x,y)$ over $\C$.
 By definition, $F_a(x,y)=x+y\in \C[[x,y]]$, which satisfies
$$F_a(0,y)=y,\  \ F_a(x,0)=x,\   \    F_a(F_a(x,y),z)=F_a(x,F_a(y,z)).$$
An {\em associate}  of $F_a(x,y)$ (see \cite{Li-phi-coor}) is a formal series
$\phi(x,z)\in\C((x))[[z]]$ such that
\begin{align*}
  \phi(x,0)=x,\ \  \  \   \phi(\phi(x,y),z)=\phi(x,y+z).
\end{align*}
It was proved therein that for any $p(x)\in \C((x))$,
\begin{align*}
  \phi(x,z):=e^{zp(x)\partial_x}x
  =\sum_{n\ge0}\frac{z^n}{n!}\(p(x)\partial_x\)^nx
\end{align*}
is an associate of $F_a(x,y)$, and every associate is of this form with $p(x)$ uniquely determined.
 From \cite{FLM2}, \cite{fhl},  for $r\in \Z$ we have
\begin{eqnarray}\label{expression-phi-r}
\phi_r(x,z):=e^{zx^{r+1}\partial_x}x=\begin{cases}x\(1-rzx^{r}\)^{-\frac{1}{r}}&\  \  \mbox{ if }r\ne 0\\
xe^z& \  \  \mbox{ if }r= 0.
\end{cases}
\end{eqnarray}
In particular, we have
$$\phi_{-1}(x,z)=e^{z\partial_x}x=x+z=F_{a}(x,z)\ \text{ and }\  \phi_0(x,z)=e^{zx\partial_x}x=xe^{z}.$$

Assume $\phi(x,z)=e^{zp(x)\partial_x}x$ with $p(x)\ne 0$, or equivalently, $\phi(x,z)\ne x$.
For any $f(x_1,x_2)\in \C((x_1,x_2))$,  $f(\phi(x,z),x)$ exists in $\C((x))[[z]]$, and
$$f(\phi(x,z),x)\ne 0\quad \te{ whenever }f(x_1,x_2)\ne 0.$$
Let $\C_\ast((x_1,x_2))$ denote  the fraction field of $\C((x_1,x_2))$:
\begin{align}
\C_\ast((x_1,x_2))=\{ f/g\ |\  f,g\in \C((x_1,x_2))\  \te{ with }g\ne 0\}.
\end{align}
Note that $\C((x))((z))$ is (isomorphic to) the fraction field of $\C((x))[[z]]$.
Then  for any $h(x_1,x_2)\in \C_{*}((x_1,x_2))$,
$$h(\phi(x,z),z)\mbox{ exists in }\C((x))((z)).$$
On the other hand, for any $f(x_1,x_2)\in\C((x_1,x_2))$,  $f(\phi(x,y),\phi(x,z))$ exists in $\C((x))[[y,z]],$
 and  the correspondence $$f(x_1,x_2)\mapsto f(\phi(x,y),\phi(x,z))$$
is a ring embedding of $\C((x_1,x_2))$ into $\C((x))[[y,z]]$.
Furthermore, this ring embedding gives rise to a ring embedding of $\C_{*}((x_1,x_2))$ into $\C((x))_{*}((y,z))$,
the fraction field of $\C((x))((y,z))$.

\begin{de}
Let $\phi(x,z)=e^{zp(x)\partial_x}x$ with $p(x)\in \C((x))^{\times}$.
Define $\C_{\phi}((x_1,x_2))$ to consist of $F(x_1,x_2)\in \C_{*}((x_1,x_2))$ such that
\begin{align}\label{p(x)-differential}
p(x_1)\partial_{x_1}F(x_1,x_2)=-p(x_2)\partial_{x_2}F(x_1,x_2).
\end{align}
\end{de}

The following two lemmas were obtained in \cite{JKLT-Defom-va}:

\begin{lem}\label{F-characterization}
Let $\phi(x,z)=e^{zp(x)\partial_x}x$ with $p(x)\in \C((x))^{\times}$ and let $F(x_1,x_2)\in \C_{*}((x_1,x_2))$. Then
(\ref{p(x)-differential}) holds  if and only if
\begin{align}
F(\phi(x_1,z),x_2)=F(x_1,\phi(x_2,-z)).
\end{align}
Furthermore, assuming either one of the two equivalent conditions, we have
\begin{align}
&F(\phi(x,z),x)=f(z),\\
&F(\phi(x,z_1),\phi(x,z_2))=f(z_1-z_2)\label{eq:F-characterization-result2}
\end{align}
for some uniquely determined $f(z)\in \C((z))$.
\end{lem}

\begin{lem}\label{diff-algebra-embeding}
The set $\C_{\phi}((x_1,x_2))$ is a subalgebra of $\C_{*}((x_1,x_2))$ with $p(x_1)\partial_{x_1}$ $(=-p(x_2)\partial_{x_2})$ as a derivation
and  the map $\pi_{\phi}: \C_{\phi}((x_1,x_2))\rightarrow \C((z))$
defined by
\begin{align}
\pi_{\phi}(F(x_1,x_2))=F(\phi(x,z),x)\quad \text{  for }F(x_1,x_2)\in \C_{\phi}((x_1,x_2))
\end{align}
 is an embedding of differential algebras, where $\C((z))$
is viewed as a differential algebra with derivation $\frac{d}{dz}$.
\end{lem}

\begin{rem}\label{rem-p(x)=1}
{\em Consider the example $\phi(x,z)=xe^z$ with $p(x)=1$.
For any (rational function) $f(x)\in \C(x)$, we have $f(x_1/x_2), f(x_2/x_1)\in \C_{\phi}((x_1,x_2))$ with
\begin{align*}
\pi_{\phi} f(x_1/x_2)=f(e^z),\quad \pi_{\phi} f(x_2/x_1)=f(e^{-z}),
\end{align*}
where $f(e^{\pm z})$ are understood as elements of $\C((z))$. Set
\begin{align}
\C_{e}((z))=\{ f(e^z)\ |\ f(x)\in \C(x)\},
\end{align}
which is a differential subfield of $\C((z))$.}
\end{rem}

\begin{rem}
{\em Let $U$ be a vector space with a linear map
$$\wh S(x_1,x_2):\ U\ot U\rightarrow U\ot U\ot \C_{\phi}((x_1,x_2)).$$
Set  $S(x)=(1\ot 1\ot \pi_\phi) \wh S(x_1,x_2)$. Note that $S(z_1-z_2)= \wh S(\phi(x,z_1),\phi(x,z_2))$.
Then $S(x)$ is a rational quantum Yang-Baxter operator if and only if
\begin{align}
  \wh S^{12}(x_1,x_2)\wh S^{13}(x_1,x_3)\wh S^{23}(x_2,x_3)
  =\wh S^{23}(x_2,x_3)\wh S^{13}(x_1,x_3)\wh S^{12}(x_1,x_2).
\end{align}}
\end{rem}

\begin{de}
Let $V$ be a nonlocal vertex algebra.
A {\em $\phi$-coordinated quasi $V$-module} is a vector space $W$ equipped with a linear map
\begin{eqnarray*}
Y_{W}(\cdot,x): \  V\rightarrow (\te{End} W)[[x,x^{-1}]];\quad
v\mapsto Y_{W}(v,x),
\end{eqnarray*}
satisfying the conditions that
$$Y_{W}(v,x)w\in W((x))\   \   \    \te{for }v\in V,\ w\in W,$$
$$Y_{W}({\bf 1},x)=1_{W} \   \  (\te{the identity operator on }W),$$
and that for any $u,v\in V$, there exists nonzero $f(x_1,x_2)\in \C((x_1,x_2))$ such that
\begin{align}
&f(x_1,x_2)Y_{W}(u,x_1)Y_{W}(v,x_2)\in \Hom(W,W((x_1,x_2))),\\
&\(f(x_1,x_2)Y_{W}(u,x_1)Y_{W}(v,x_2)\)|_{x_1=\phi(x_2,z)}=f(\phi(x_2,z),x_2)Y_{W}(Y(u,z)v,x_2).
\end{align}
\end{de}

Let $W$ be a vector space (over $\C$). Set
$$\E(W)=\Hom(W,W((x))).$$
More generally, for any positive integer $r$, set
\begin{align}
\E^{(r)}(W)=\Hom(W,W((x_1,x_2,\dots,x_r))).
\end{align}
An (ordered) sequence $\(a_1(x),\dots,a_r(x)\)$ in
$\E(W)$ is said to be {\em quasi-compatible} (see  \cite{Li-nonlocal}) if
there exists a nonzero series $f(x,y)$ from $\C((x,y))$ such that
\begin{align*}
  \(\prod_{1\le i<j\le r}f(x_i,x_j)\)a_1(x_1)\cdots a_r(x_r)
  \in\E^{(r)}(W).
\end{align*}
Furthermore, a subset $U$ of $\E(W)$ is said to be {\em quasi-compatible} if every finite sequence in $U$ is quasi-compatible.

Let $(\alpha(x),\beta(x))$ be a quasi-compatible ordered pair in $\E(W)$. Define
$$\alpha(x)_{n}^{\phi}\beta(x)\in \E(W) \   \  \te{ for }n\in \Z$$
in terms of generating function
\begin{align}
Y_{\E}^{\phi}(\alpha(x),z)\beta(x)=\sum_{n\in \Z}\alpha(x)_{n}^{\phi}\beta(x) z^{-n-1}
\end{align}
by
\begin{align}
Y_{\E}^{\phi}(\alpha(x),z)\beta(x)=\iota_{x,z}(1/f(\phi(x,z),x))\(f(x_1,x)\alpha(x_1)\beta(x)\)|_{x_1=\phi(x,z)},
\end{align}
where $f(x_1,x_2)$ is any nonzero series from $\C((x_1,x_2))$ such that
\begin{align}
f(x_1,x_2)\alpha(x_1)\beta(x_2)\in \Hom(W,W((x_1,x_2)))\  \  (=\E^{(2)}(W)).
\end{align}
A quasi-compatible subspace $U$ of $\E(W)$ is said to be {\em $Y_\E^\phi$-closed} if
\begin{align}
a(x)_{n}^{\phi}b(x)\in U\quad \text{for all }a(x),b(x)\in U,\ n\in \Z.
\end{align}

The following is a result of \cite{Li-phi-coor}:

\begin{thm}\label{thm:abs-construct-non-h-adic}
Let $W$ be a vector space and  let $V$ be a $Y_\E^\phi$-closed quasi-compatible subspace of $\E(W)$
with $1_{W}\in V$.
Then $(V,Y_\E^\phi,1_{W})$ carries the structure of a nonlocal vertex algebra and
$W$ is a $\phi$-coordinated quasi $V$-module with $Y_W(\al(x),z)=\al(z)$ for $\al(x)\in V$.
On the other hand, for every quasi-compatible subset $U$ of $\E(W)$,
there exists a unique minimal $Y_\E^\phi$-closed quasi-compatible subspace $\<U\>_\phi$
which contains $1_W$ and $U$.
Furthermore, $(\<U\>_\phi,Y_\E^\phi,1_W)$ carries the structure of a nonlocal vertex
algebra and $W$ is a $\phi$-coordinated quasi $\<U\>_\phi$-module.
\end{thm}



For a subgroup $\Gamma$ of $\C^{\times}$, set
\begin{align}
\C_{\Gamma}[x]=\langle x-\alpha\ |\ \alpha\in \Gamma\rangle,
\end{align}
the multiplicative monoid generated in $\C[x]$ by $x-\alpha$ for $\alpha\in \Gamma$, and set
\begin{align}
\C_{\Gamma}[x_1,x_2]=\langle x_1-\alpha x_2\ |\ \alpha\in \Gamma\rangle.
\end{align}

The following notion was introduced in \cite{JKLT-G-phi-mod} (cf. \cite{Li-G-phi}):

\begin{de}\label{de:G-equiv-phi-mod-1}
Let $V$ be a $(G,\chi)$-module nonlocal vertex algebra and let $\chi_{\phi}$ be a linear character of $G$
such that
\begin{eqnarray}\label{p(x)-compatibility}
\phi(x,\chi(g)z)=\chi_{\phi}(g)\phi(\chi_{\phi}(g)^{-1}x,z)\   \   \mbox{ for }g\in G.
\end{eqnarray}
 A {\em $(G,\chi_{\phi})$-equivariant $\phi$-coordinated quasi $V$-module} is a $\phi$-coordinated quasi $V$-module
$(W,Y_W)$ satisfying the conditions that
\begin{eqnarray}\label{phi-module-equiv-1}
Y_{W}(R(g)v,x)=Y_{W}(v,\chi_{\phi}(g)^{-1}x)\   \   \   \   \mbox{ for }g\in G,\  v\in V
\end{eqnarray}
and that for $u,v\in V$, there exists $q(x_1,x_2)\in \C_{\chi_{\phi}(G)}[x_1,x_2]$ such that
\begin{eqnarray}
q(x_1,x_2)Y_W(u,x_1)Y_W(v,x_2)\in \Hom (W,W((x_1,x_2))).
\end{eqnarray}
\end{de}

\begin{rem}
{\em Note that for $\phi(x,z)=e^{zx^{r+1}\partial_x}x$ with $r\in \Z$,
(\ref{p(x)-compatibility}) is equivalent to $\chi=\chi_{\phi}^{-r}$. In particular, for $r=-1$ this amounts to $\chi=\chi_{\phi}$, while
for $r=0$, this amounts to $\chi=1$ (the trivial character).}
\end{rem}



Let $W$ be a vector space, $\Gamma$ a subgroup of $\C^\times$.
Define a group homomorphism
$$R:\Gamma\rightarrow \te{GL}(\E(W)),\  \  g\mapsto R_g$$
by
\begin{align}
R_g(a(x))=a(g^{-1} x)\quad  \te{ for }a(x)\in \E(W).
\end{align}
Let $\chi_{\phi}$ be the natural embedding of $\Gamma$ into $\C^{\times}$, i.e.,
\begin{eqnarray}
\chi_{\phi}(g)=g\   \   \    \mbox{ for }g\in \Gamma.
\end{eqnarray}

\begin{de}
A subset $U$ of $\E(W)$ is said to be {\em $\Gamma$-quasi compatible} if
 for any $a_1(x),a_2(x),\dots,a_k(x)\in U$, there exists
$ f(x_1,x_2)\in \C_{\Gamma}[x_1,x_2]$ such that
\begin{align*}
  \left(\prod_{1\le i<j\le k}f(x_i,x_j)\right)a_1(x_1)a_2(x_2)\cdots a_k(x_k)\in\Hom(W,W((x_1,\dots,x_k))).
\end{align*}
\end{de}

We have (see  \cite{JKLT-G-phi-mod}):

\begin{thm}\label{coro:G-va-abs-construct-non-h-adic}
Let $W$ be a vector space and let $\Gamma$ be a subgroup of $\C^\times$.
Suppose that $U$ is a $\Gamma$-quasi compatible $\Gamma$-submodule of $\E(W)$.
Then the nonlocal vertex algebra $\<U\>_\phi$ together with the group homomorphism $R$
 is a $(\Gamma,\chi)$-module nonlocal vertex algebra
and $W$ is a $(\Gamma,\chi_{\phi})$-equivariant $\phi$-coordinated quasi $\<U\>_\phi$-module with
$Y_W(a(x),z)=a(z)$ for $a(x)\in \<U\>_\phi$.
\end{thm}



\section{$\hbar$-adic (weak) quantum vertex algebras}

In this section, we recall and refine various results on $\hbar$-adic nonlocal vertex algebras,
$\hbar$-adic quantum vertex algebras, and their $\phi$-coordinated quasi modules.
Throughout this paper, $\hbar$ is a formal variable.

\subsection{$\hbar$-adic nonlocal vertex algebras and $\hbar$-adic quantum vertex algebras}

We first review some basic facts about $\C[[\hbar]]$-modules. A $\C[[\hbar]]$-module $V$ is said to be
{\em torsion-free} if $\hbar^{n}v=0$  implies $v=0$ for $n\in \N,\ v\in V$, and it is said to be
{\em separated} if $\cap_{n\ge 0}\hbar^{n}V=0$. A {\em Cauchy sequence}  in $V$ is a sequence $\{ v_n\}_{n\ge 0}$
satisfying the condition that  for every $n\in \N$, there exists $k\in \N$ such that $v_p-v_q\in \hbar^nV$ for all $p,q\ge k$.
A sequence $\{ v_n\}_{n\ge 0}$ is said to
have a limit $v$ in $V$, written as $\lim_{n\rightarrow \infty}v_n=v$,  if for every $n\in \N$, there exists $k\in \N$
such that $v_p-v\in \hbar^nV$ for all $p\ge k$.
Furthermore, $V$ is said to be ($\hbar$-adically) {\em complete} if every Cauchy sequence in $V$ has a limit in $V$.

A $\C[[\hbar]]$-module $V$ is said to be {\em topologically free}
 if there exists a $\C$-subspace $V_0$ of $V$ such that $V=V_0[[\hbar]]$.
 For two topologically free $\C[[\hbar]]$-modules $V=V_0[[\hbar]]$ and $W=W_0[[\hbar]]$,
let $V\ptimes W$ denote the completed tensor product, which is isomorphic to $\(V_0\otimes W_0\)[[\hbar]]$.

The following result can be found in \cite{Kassel-topologically-free}:

\begin{lem}\label{free-top}
A $\C[[\hbar]]$-module is topologically free  if and only if it is  torsion-free, separated, and $\hbar$-adically complete.
\end{lem}

As every $\C[[\hbar]]$-submodule of a topologically free $\C[[\hbar]]$-module is
torsion-free and separated, we immediately have:

\begin{lem}\label{free-top-submodule}
Every $\hbar$-adically closed $\C[[\hbar]]$-submodule of a topologically free $\C[[\hbar]]$-module is topologically free.
\end{lem}

The following is a very useful fact which can be found in  \cite{Kassel-topologically-free}: 

\begin{lem}\label{basic-facts}
Let $W_1,W_2$ be topologically free $\C[[\hbar]]$-modules and let
$\psi:W_1\rightarrow W_2$ be a continuous $\C[[\hbar]]$-module map.
 If the derived $\C$-linear map $\bar{\psi}: W_1/\hbar W_1\rightarrow W_2/\hbar W_2$ is one-to-one (resp. onto),
 then $\psi$ is one-to-one (resp. onto).
\end{lem}

Let $U$ be a $\C[[\hbar]]$-submodule of a topologically free $\C[[\hbar]]$-module $V$.
Denote by $\overline{U}$ the closure of $U$ (with respect to the $\hbar$-adic topology of $V$).
Follow \cite{Li-h-adic} to set
\begin{align}
  [U]=\set{v\in V}{\hbar^n v\in U\te{ for some }n\in \N}.
\end{align}
It is straightforward to see that
\begin{align}
  \left[[U]\right]=[U],\   \   \   \   \overline{\overline{U}}=\overline{U}.
  \end{align}

By using a similar argument  in the proof of \cite[Proposition 3.7]{Li-h-adic}, we have:

\begin{lem}\label{lem:topo-op}
Let $V$ be a topologically free $\C[[\hbar]]$-module.
Suppose that $U$ is a $\C[[\hbar]]$-submodule  such that
$[U]=U$. Then $[\overline{U}]=\overline{U}$ and $\overline{U}$ is topologically free.
\end{lem}


Recall that for a positive integer $r$ and for a vector space $U$ over $\C$,
\begin{align*}
  \E^{(r)}(U)=\Hom\(U,U((x_1,x_2,\dots,x_r))\).
\end{align*}
Set $\E(U)=\E^{(1)}(U)=\Hom(U,U((x)))$.

Now, let $W=W_0[[\hbar]]$ be a topologically free $\C[[\hbar]]$-module.
For each positive integer $n$, the canonical quotient map from $W$ to $W/\hbar^nW$ induces a linear map
\begin{align}
\pi_n:\  \(\te{End}_{\C[[\hbar]]}W\)[[x_1^{\pm 1},\dots,x_r^{\pm 1}]]
\longrightarrow
\(\te{End }(W/\hbar^nW)\)[[x_1^{\pm 1},\dots,x_r^{\pm 1}]].\  \
\end{align}
For any positive integer $r$, set
\begin{align*}
 \E_\hbar^{(r)}(W)=\set{f\in\(\te{End}_{\C[[\hbar]]}W\)
   [[x_1^{\pm 1},\dots,x_r^{\pm 1}]]}
 {\pi_n(f)\in\E^{(r)}(W/\hbar^nW)\te{ for }n\ge 1}.
\end{align*}
Denote $\E^{(1)}_{\hbar}(W)$ by  $\E_\hbar(W)$ alternatively.
The following can be found in \cite{Li-h-adic}:

\begin{lem}\label{lem:E-h}
Let $W=W_0[[\hbar]]$ be a topologically free $\C[[\hbar]]$-module and let $r$ be a positive integer.  Then

(1) $\E_\hbar^{(r)}(W)=\E^{(r)}(W_0)[[\hbar]]$.

(2) For every positive integer $n$, $\pi_n(\E_\hbar^{(r)}(W))=\E^{(r)}(W/\hbar^nW)$.

(3) $\left[\E_\hbar^{(r)}(W)\right]=\E_\hbar^{(r)}(W)$
and $\overline{\E_\hbar^{(r)}(W)}=\E_\hbar^{(r)}(W)$.

(4) The inverse system
\begin{align*}
  \xymatrix{
    0&W/\hbar W\ar[l]&W/\hbar^2 W\ar[l]&\cdots\ar[l]
  }
\end{align*}
induces an inverse system
\begin{align*}
  \xymatrix{
    0&\E^{(r)}(W/\hbar W)\ar[l]_{\theta_0}&\E^{(r)}(W/\hbar^2W)\ar[l]_{\theta_1}&\cdots\ar[l]_{\theta_2}.
  }
\end{align*}

(5) The map $\E^{(r)}_\hbar(W)\rightarrow \varprojlim\limits_{n\ge 1}\E^{(r)}(W/\hbar^nW)$ given by
$f\mapsto \(\pi_n(f)\)_{n\ge 1}$ is a $\C[[\hbar]]$-module isomorphism.
\end{lem}

The following notion was introduced in \cite{Li-h-adic}:

\begin{de}\label{de:h-adic-nonva}
An {\em $\hbar$-adic nonlocal vertex algebra} is a topologically free
$\C[[\hbar]]$-module $V$, equipped with a $\C[[\hbar]]$-module
map
\begin{align*}
  Y(\cdot,x):\ V&\longrightarrow \E_\hbar(V),\\
  v&\mapsto Y(v,x)=\sum_{n\in\Z}v_nx^{-n-1},
\end{align*}
and equipped with a distinguished vector $\vac\in V$,  satisfying the conditions that
\begin{align*}
  &Y(\vac,x)=1,\\
  &Y(v,x)\vac\in V[[x]]\,\,\te{and}\,\,
  \lim\limits_{x\rightarrow 0}Y(v,x)\vac=v\  \
  \te{for }v\in V,
\end{align*}
and that for $u,v,w\in V$ and for every $n\in\Z_{+}$, there exists $\ell\in\N$
such that
\begin{align}
  (z+y)^\ell Y(u,z+y)Y(v,y)w\equiv (z+y)^\ell Y(Y(u,z)v,y)w
\end{align}
modulo $\hbar^n V[[z^{\pm 1},y^{\pm 1}]]$ (the {\em $\hbar$-adic weak associativity}).
\end{de}

For $n\in \Z_{+}$, set
\begin{align}
\mathcal{R}_n=\C[[\hbar]]/\hbar^n\C[[\hbar]]\simeq \C[\hbar]/\hbar^n\C[\hbar].
\end{align}
We have (loc. cit):

\begin{prop}\label{h-adic-quotient}
Let $V$ be a topologically free $\C[[\hbar]]$-module with a $\C[[\hbar]]$-module map
$Y(\cdot,x): V\rightarrow ({\rm End} V)[[x,x^{-1}]]$ and a vector ${\bf 1}\in V$.
Then $V$ is an $\hbar$-adic nonlocal vertex algebra if and only if for every positive integer $n$,
$V/\hbar^nV$ is a nonlocal vertex algebra over $\mathcal{R}_n$.
\end{prop}

Let $V$ be an $\hbar$-adic nonlocal vertex algebra. Just as for a nonlocal vertex algebra,
define a $\C[[\hbar]]$-linear operator $\mathcal D$ on $V$ by
$\mathcal D(v)=v_{-2}\vac$ for $v\in V$, and we have
\begin{align}
[\mathcal{D},Y(v,x)]=Y(\mathcal D v,x)=\frac{d}{dx}Y(v,x)\quad \text{ for }v\in V.
\end{align}
On the other hand, following \cite{EK-qva}, denote by $Y(x)$
the $\C[[\hbar]]$-linear map from $V\ptimes V$ to $V[[x,x^{-1}]]$, associated to
the vertex operator map $Y(\cdot,x)$:
\begin{align}
  Y(x):\  V\ptimes V\longrightarrow V[[x,x^{-1}]]; \  \
  (u,v)\mapsto Y(u,x)v.
\end{align}

\begin{de}\label{de:h-adic-wqva}
An {\em $\hbar$-adic weak quantum vertex algebra} is an $\hbar$-adic nonlocal vertex algebra
$V$ which satisfies the {\em $\hbar$-adic $\mathcal{S}$-locality}:
For $u,v\in V$, there exists
\begin{align*}
  F(u,v,x)\in V\ptimes V\ptimes \C((x))[[\hbar]],
\end{align*}
satisfying the condition that for any $n\in\Z_{+}$, there exists $k\in\N$ such that
\begin{align}
  (x-y)^k Y(u,x)Y(v,y)w\equiv
  (x-y)^k Y(y)(1\otimes Y(x))\(F(u,v,y-x)\otimes w\)
\end{align}
modulo $\hbar^n V[[x^{\pm 1},y^{\pm 1}]]$ for all $w\in V$.
 If the $\hbar$-adic $\mathcal{S}$-locality holds with $F(u,v,x)=v\otimes u$ for all $u,v\in V$,
 we call $V$ an {\em $\hbar$-adic vertex algebra}.
\end{de}

Using a routine argument in vertex algebra theory, we have (see \cite{Li-h-adic}):

\begin{prop}\label{rem:h-S-Jacobi}
An $\hbar$-adic weak quantum vertex algebra can be defined equivalently by replacing
 the $\hbar$-adic weak associativity in Definition \ref {de:h-adic-nonva} with
the property that for $u,v\in V$, there exists $F(u,v,x)\in V\ptimes V\ptimes \C((x))[[\hbar]]$ such that
\begin{align}\label{S-Jacobi}
&z\inverse\delta\(\frac{x-y}{z}\)
Y(u,x)Y(v,y)w\\
&\quad \   \   -z\inverse\delta\(\frac{y-x}{-z}\)
Y(y)\(1\otimes Y(x)\)\(F(u,v,-z)\otimes w\)\nonumber\\
=\ &x\inverse\delta\(\frac{y+z}{x}\)
Y\(Y(u,z)v,y\)w\nonumber
\end{align}
(the {\em $\mathcal{S}$-Jacobi identity}) for all $w\in V$.
\end{prop}

The following is a slight generalization of Etingof-Kazhdan's notion of quantum vertex operator algebra
(see \cite{EK-qva}):

\begin{de}
An {\em $\hbar$-adic quantum vertex algebra} is an $\hbar$-adic nonlocal vertex algebra $V$
equipped with a unitary rational quantum Yang-Baxter operator
\begin{align}
  {\mathcal{S}}(x):V\ptimes V\longrightarrow V\ptimes V\ptimes \C((x))[[\hbar]],
\end{align}
which satisfies the {\em shift condition:}
\begin{align}
  \left[\mathcal D\otimes 1,\mathcal{S}(x)\right]=-\frac{d}{dx}\mathcal{S}(x),
\end{align}
the {\em $\hbar$-adic $\mathcal{S}$-locality:} For any $u,v\in V$ and $n\in \Z_{+}$,
there exists $k\in \N$ such that
\begin{align*}
 & (x-z)^kY(x)\(1\otimes Y(z)\)(u\otimes v\otimes w)\\
\equiv \  &  (x-z)^kY(z)\(1\otimes Y(x)\)\(\mathcal{S}(z-x)(v\otimes u)\otimes w\)
\end{align*}
 modulo $\hbar^nV[[x^{\pm 1},z^{\pm 1}]]$ for all $w\in V$, and the hexagon identity holds:
\begin{align}
  \mathcal{S}(x)\(Y(z)\otimes 1\)=\(Y(z)\otimes 1\)\mathcal{S}^{23}(x)\mathcal{S}^{13}(x+z).
\end{align}
\end{de}

\begin{rem}
{\em Note that a quantum vertex operator algebra in the sense of Etingof-Kazhdan
is the same as an $\hbar$-adic quantum vertex algebra $V$ such that $V/\hbar V$ is a vertex algebra.
In contrast, for an $\hbar$-adic quantum vertex algebra $V$, $V/\hbar V$ is a quantum vertex algebra over $\C$. }
\end{rem}

By definition, any $\hbar$-adic quantum vertex algebra is an $\hbar$-adic weak quantum vertex algebra.
On the other hand, if $V$ is an $\hbar$-adic quantum vertex algebra, then
for any $u,v\in V$, the $\mathcal{S}$-Jacobi identity (\ref{S-Jacobi}) holds with $\mathcal{S}(-z)(v\otimes u)$
in place of $F(u,v,-z)$.
The following is an immediate consequence:

\begin{coro}\label{lem:q-Jacobi}
Let $V$ be an $\hbar$-adic quantum vertex algebra with rational quantum Yang-Baxter operator $\mathcal{S}(x)$.
Then for any $u,v,w\in V$,
\begin{align*}
  Y(u,x)&Y(v,z)w-Y(z)\(1\ot Y(x)\)(\mathcal{S}(z-x)(v\ot u)\ot w)\\
  &=\sum_{j\ge 0}Y(u_jv,z)w\frac{1}{j!}\left(\frac{\partial}{\partial z}\right)^jx\inv\delta\(\frac{z}{x}\).
\end{align*}
\end{coro}

\begin{de}
An $\hbar$-adic nonlocal vertex algebra $V$ is said to be {\em non-degenerate} if
$V/\hbar V$ is non-degenerate.
\end{de}

This notion above is due to \cite{EK-qva}, and more importantly, the results therein imply:

\begin{prop}
Let $V$ be an $\hbar$-adic weak quantum vertex algebra.
If $V$ is non-degenerate, then there exists a  $\C[[\hbar]]$-module map
$\mathcal{S}(x): V\ptimes V\rightarrow V\ptimes V\ptimes \C((x))[[\hbar]]$,
 which is uniquely determined by the $\hbar$-adic $\mathcal{S}$-locality,
and $V$ with $\mathcal{S}(x)$ is an $\hbar$-adic quantum vertex algebra.
 \end{prop}

Let $V$ be an $\hbar$-adic nonlocal vertex algebra.
Suppose that $U$ is a $\C[[\hbar]]$-submodule of $V$ such that
\begin{align}\label{condition-hadic-subalgebra}
[U]=U, \quad  \overline{U}=U,\
\   {\bf 1}\in U,  \  \   u_mv\in U\  \mbox{for }u,v\in U,\ m\in \Z.
\end{align}
As $[U]=U$, we have $U\cap \hbar^nV=\hbar^nU$ for $n\in \N$.
Thus the $\hbar$-adic topology of $U$ coincides with the induced topology.
Furthermore, as $\overline{U}=U$, $U$ is topologically free. Then $U$ itself is an $\hbar$-adic nonlocal vertex algebra.
Motivated by this, we formulate the following notion:

\begin{de}\label{def-hadic-subalgebra}
An {\em $\hbar$-adic nonlocal vertex subalgebra} of an $\hbar$-adic nonlocal vertex algebra $V$
is defined to be a $\C[[\hbar]]$-submodule $U$ satisfying the condition (\ref{condition-hadic-subalgebra}).
\end{de}

Let $V$ be an $\hbar$-adic nonlocal vertex algebra. For any $\C[[\hbar]]$-submodule $T$, define
$$T^{(2)}={\rm span}_{\C[[\hbar]]}\{ a_mb\ |\ a,b\in T,\ m\in \Z\}.$$
Let $S$ be a subset of $V$.
Set $S^{(1)}=\C[[\hbar]]{\bf 1}+\C[[\hbar]]S$, a $\C[[\hbar]]$-submodule of $V$.
Then inductively define $S^{(n+1)}=(S^{(n)})^{(2)}$ for $n\ge 1$.
Notice that as ${\bf 1}\in S^{(1)}$ we have $S^{(1)}\subset S^{(2)}$.
In this way we obtain an ascending sequence of $\C[[\hbar]]$-submodules
$$S^{(1)}\subset S^{(2)}\subset S^{(3)}\subset \cdots$$
with $\{ {\bf 1}\}\cup S\subset S^{(1)}$.
Then we set
\begin{align}
\<S\>=\overline{ \left[ \cup_{n\ge 1}S^{(n)}\right]}.
\end{align}
Using Lemma \ref{lem:topo-op}, we immediately get:

\begin{prop}
Let $V$ be an $\hbar$-adic nonlocal vertex algebra and let $S$ be a subset. Then
$\<S\>$ is an $\hbar$-adic nonlocal vertex subalgebra containing $S$. Furthermore, any
$\hbar$-adic nonlocal vertex subalgebra that contains $S$ also contains $\<S\>$.
 \end{prop}

We also have the following result:

 \begin{lem}\label{hadic-generating-subset}
Let $V$ be an $\hbar$-adic nonlocal vertex algebra and let $S$ be a subset such that
 for every positive integer $n$, $\{a+\hbar^nV\ |\ a\in S\}$ generates $V/\hbar^nV$ as a nonlocal vertex algebra.
 Then $\<S\>=V$.
\end{lem}

 \begin{proof} Let $U$ be any $\hbar$-adic nonlocal vertex subalgebra containing $S$.
 Let $n\in \Z_{+}$. As $[U]=U$, we have
 $U\cap \hbar^nV=\hbar^nU$. Then $U/\hbar^nU\subset V/\hbar^nV$.
  Since $S\subset U$ and $\{a+\hbar^nV\ |\ a\in S\}$ generates $V/\hbar^nV$ as a nonlocal vertex algebra, we have
 $U/\hbar^nU=V/\hbar^nV$, which implies $V=U+\hbar^nV$. Now, let $v\in V$. Then there exist $u(n)\in U$ for $n\in \Z_+$ such that
 $v-u(n)\in \hbar^nV$. Thus $\lim_{n\rightarrow \infty}u(n)=v$ in $V$. As $\overline{U}=U$, we conclude $v\in U$. This proves  $U=V$.
 In particular, we have $\<S\>=V$.
 \end{proof}

For convenience, we formulate the following notion:

 \begin{de}\label{def-hadic-generating-subset}
A subset $S$ of an $\hbar$-adic nonlocal vertex algebra $V$ is called a {\em generating subset}
if for every positive integer $n$,
 $\{a+\hbar^nV\ |\ a\in S\}$ generates $V/\hbar^nV$ as a nonlocal vertex algebra.
\end{de}

\begin{rem}
{\em  Note that a subset $S$ of $V$ with $\<S\>=V$ is {\em not} necessarily a generating subset.
 For example, take $S=\hbar V$. Then we have
 $\<S\>=V$ as $[S]=V$, but $\{ a+\hbar V\ |\ a\in S\}=\{ 0+\hbar V\}$, which does not generate $V/\hbar V$
if $\dim_{\C} V/\hbar V>1$.}
\end{rem}

\subsection{$\phi$-coordinated quasi modules}

Here, we study $\phi$-coordinated quasi modules for $\hbar$-adic nonlocal vertex algebras.
Throughout this section, we fix an associate $\phi(x,z)=e^{zp(x)\partial_x}(x)$ of $F_a(x,y)$  with $p(x)\ne 0$.

\begin{de}\label{de:h-adic-phi-mod}
Let $V$ be an $\hbar$-adic nonlocal vertex algebra.
A {\em $\phi$-coordinated quasi $V$-module}  is a topologically free $\C[[\hbar]]$-module $W$
equipped with a $\C[[\hbar]]$-module map
\begin{align*}
  Y_W^\phi(\cdot,x):\  V\longrightarrow \E_\hbar(W);\ \ \
  v\mapsto Y_W^{\phi}(v,x),
\end{align*}
satisfying the conditions that $Y_W^\phi(\vac,x)=1_W$
and that for $u,v\in V$ and for every positive integer $n$,
there exists a nonzero series $q(x_1,x_2)\in\C((x_1,x_2))$ such that
\begin{align*}
  &q(x_1,x_2)\pi_n\( Y_W^\phi(u,x_1)Y_W^\phi(v,x_2)\)\in \E^{(2)}\(W/\hbar^nW\)
\end{align*}
and
\begin{align*}
  q(\phi(x,z),x)Y_W^\phi(Y(u,z)v,x)w\equiv
  \(q(x_1,x)Y_W^\phi(u,x_1)Y_W^\phi(v,x)w\)|_{x_1=\phi(x,z)}
\end{align*}
modulo $\hbar^nW[[x^{\pm 1},z^{\pm 1}]]$ for all $w\in W$.
Furthermore, $W$ is called a {\em $\phi$-coordinated $V$-module} (without the prefix ``quasi'')
if $q(x_1,x_2)$ is assumed to be a polynomial of the form $(x_1-x_2)^k$ with $k\in \N$.
\end{de}

The following is a straightforward generalization of Proposition \ref{h-adic-quotient}:

\begin{prop}\label{prop:phi-coor-ind}
Let $V$ be an $\hbar$-adic nonlocal vertex algebra and let
$Y_W^\phi(\cdot,x): V\rightarrow ({\rm End}W)[[x,x\inverse]]$ be a $\C[[\hbar]]$-module map,
where $W$ is a topologically free $\C[[\hbar]]$-module.
Then $(W,Y_W^\phi)$ is a $\phi$-coordinated (quasi) $V$-module
if and only if for every positive integer $n$,
$W/\hbar^nW$ is a $\phi$-coordinated (quasi) $V/\hbar^nV$-module.
\end{prop}

Let $G$ be a group with a linear character $\chi:G\rightarrow \C^\times$.
A {\em  $(G,\chi)$-module $\hbar$-adic nonlocal vertex algebra} is an
$\hbar$-adic nonlocal vertex algebra $V$ equipped with a group representation
$R: G\rightarrow {\rm GL}(V)$ such that $R(g){\bf 1}={\bf 1}$ and
$$R(g)Y(v,x)R(g)^{-1}=Y(R(g)v,\chi(g)x)\quad \te{ for } g\in G,\ v\in V.$$

\begin{rem}
{\em Let $V$ be an $\hbar$-adic nonlocal vertex algebra with a group representation
$R: G\rightarrow {\rm GL}(V)$.
It is straightforward to show that $V$ is a $(G,\chi)$-module $\hbar$-adic nonlocal vertex algebra
if and only if for every positive integer $n$, $V/\hbar^nV$ is a $(G,\chi)$-module nonlocal vertex algebra
over $\mathcal{R}_n$.}
\end{rem}

\begin{de}\label{def-h-G-equivariant-coordinated-module}
Let  $V$ be a $(G,\chi)$-module $\hbar$-adic nonlocal vertex algebra and
 let $\chi_{\phi}$ be a linear character of $G$, satisfying the compatibility condition (\ref{p(x)-compatibility}).
 A {\em $(G,\chi_{\phi})$-equivariant $\phi$-coordinated quasi $V$-module} is a
 $\phi$-coordinated quasi $V$-module $(W,Y_W)$ such that
for every positive integer $n$, $W/\hbar^n W$ is a
$(G,\chi_{\phi})$-equivariant $\phi$-coordinated quasi $V/\hbar^nV$-module.
\end{de}

Note that for a $(G,\chi_{\phi})$-equivariant $\phi$-coordinated quasi $V$-module $(W,Y_W^{\phi})$,
as $\cap_{n\ge 1}\hbar^nW=0$, we have
\begin{align}
Y_W(R(g)v,x)=Y_W(v,\chi_{\phi}(g)^{-1}x)\quad \te{for }g\in G,\ v\in V.
\end{align}

Next, we discuss a conceptual construction of $\hbar$-adic nonlocal vertex algebras and
their $\phi$-coordinated quasi-modules.
Let $W=W_0[[\hbar]]$ be a topologically free $\C[[\hbar]]$-module.
Recall that $\E_\hbar(W)$ consists of every $a(x)\in \te{End}(W)$ such that $\pi_n(a(x))\in \E(W/\hbar^nW)$
for $n\in \Z_{+}$, where $\pi_n: \te{End}(W)\rightarrow \te{End}(W/\hbar^nW)$ is the natural $\C[\hbar]$-module map.

\begin{de}
A finite sequence $\(a_1(x),\dots,a_r(x)\)$ in $\E_\hbar(W)$
is said to be {\em $\hbar$-adically (quasi-)compatible} if for every positive integer $n$,
$$\(\pi_n(a_1(x)),\dots,\pi_n(a_r(x))\)$$ is a
(quasi-)compatible sequence in $\E(W/\hbar^nW)$.
Furthermore, a subset $U$ of $\E_\hbar(W)$ is said to be {\em $\hbar$-adically (quasi-)compatible}
if every finite sequence in $U$ is $\hbar$-adically (quasi-)compatible.
\end{de}



Let $(\al(x),\be(x))$ be an $\hbar$-adically quasi-compatible ordered pair in $\E_\hbar(W)$.
By definition,  for every positive integer $n$,
$(\pi_n\(\al(x)\),\pi_n\(\be(x)\))$ is a quasi-compatible ordered pair in $\E(W/\hbar^nW)$.
Then  we have (see \cite{Li-phi-coor})
 $$\pi_n(\al(x))_m^\phi\pi_n(\be(x))\in \E(W/\hbar^nW)\    \    \te{ for }m\in \Z,$$
which are defined in terms of generating function
 \begin{align*}
  Y_\E^\phi\(\pi_n(\al(x)),z\)\pi_n(\be(x))
  =\sum_{m\in\Z}\pi_n(\al(x))_m^\phi\pi_n(\be(x))z^{-m-1}
\end{align*}
by
\begin{align*}
&Y_\E^\phi\(\pi_n(\al(x)),z\)\pi_n(\be(x))\\
=\ & \iota_{x,z}q(\phi(x,z),x)\inverse \(q(x_1,x)\pi_n(\al(x_1))
    \pi_n(\be(x))\)|_{x_1=\phi(x,z)},
\end{align*}
where $q(x_1,x_2)$ is any nonzero element of $\C((x_1,x_2))$ such that
$$q(x_1,x_2)\pi_n(\al(x_1))\pi_n(\be(x_2)) \in\E^{(2)}\(W/\hbar^nW\).$$
Notice that for $q(x_1,x_2)\in\C((x_1,x_2))$, if
\begin{align*}
  q(x_1,x_2)\pi_{n+1}(\al(x_1))\pi_{n+1}(\be(x_2))
    \in\E^{(2)}\(W/\hbar^{n+1}W\)
\end{align*}
then
\begin{align*}
  q(x_1,x_2)\pi_n(\al(x_1))\pi_n(\be(x_2))
    \in\E^{(2)}\(W/\hbar^nW\).
\end{align*}
From \cite{Li-phi-coor}, we have
\begin{align}
  \theta_n&  \(Y_\E^\phi\(\pi_{n+1}(\al(x)),z\)\pi_{n+1}(\be(x)) \)
  =Y_\E^\phi\(\pi_n(\al(x)),z\)\pi_n(\be(x)),
\end{align}
recalling that $\theta_n$ is the natural map from $\E(W/\hbar^{n+1}W)$ to $\E(W/\hbar^{n}W)$.

\begin{de}\label{de:Y-E-phi}
Let $\(\al(x),\be(x)\)$ be an $\hbar$-adically quasi-compatible ordered pair  in $\E_\hbar(W)$.
For $m\in\Z$, we define
\begin{align}
  \al(x)_m^\phi\be(x)=\varprojlim\limits_n \pi_n(\al(x))_m^\phi\pi_n(\be(x))\in \E_\hbar(W).
\end{align}
Furthermore, form a generating function
\begin{align}\label{vertex-operator-Y-E-phi}
  Y_\E^\phi\(\al(x),z\)\be(x)=\sum_{m\in\Z}\al(x)_m^\phi\be(x)z^{-m-1}.
\end{align}
\end{de}

From definition and Lemma \ref{lem:E-h}, $\al(x)_m^{\phi}\be(x)$ is well defined for
each $m\in\Z$, and for every positive integer $n$ we have
\begin{align}
  \pi_n\(\al(x)_m^\phi\be(x)\)=\pi_n(\al(x))_m^\phi\pi_n(\be(x)) \   \   \,  \te{for }m\in\Z.
\end{align}
Namely,
\begin{align}
  \pi_n\(Y_\E^\phi\(\al(x),z\)\be(x)\)=Y_\E^\phi\(\pi_n(\al(x)),z\)\pi_n(\be(x)).
\end{align}
Using \cite[Proposition 4.9]{Li-phi-coor} we immediately have  the following $\hbar$-adic analogue:

\begin{lem}\label{lem:quasi-compatible}
Let $\varphi_1(x),\dots,\varphi_r(x), \al(x),\be(x), \phi_1(x),\dots,\phi_s(x)\in\E_\hbar(W)$.
Assume that the ordered sequences $(\al(x),\be(x))$ and
\begin{align*}
  (\varphi_1(x),\dots,\varphi_r(x),\al(x),\be(x),\phi_1(x),\dots,\phi_s(x))
\end{align*}
are $\hbar$-adically quasi-compatible. Then for every $m\in\Z$  the ordered sequence
\begin{align*}
  (\varphi_1(x),\dots,\varphi_r(x),\al(x)_m^\phi\be(x),\phi_1(x),\dots,\phi_s(x))
\end{align*}
is $\hbar$-adically quasi-compatible. The same assertion holds without the prefix ``quasi.''
\end{lem}



\begin{de}\label{closed}
An $\hbar$-adically quasi-compatible $\C[[\hbar]]$-submodule $U$ of $\E_\hbar(W)$ is said to be
{\em  $Y_\E^\phi$-closed}  if
\begin{align*}
  \al(x)_m^\phi\be(x)\in U\  \   \,\te{for all }\al(x),\be(x)\in U, \  m\in\Z.
\end{align*}
\end{de}

We have (cf. \cite[Proposition 4.13]{Li-h-adic}):

\begin{thm}\label{prop:abs-construct}
Let $W$ be a topologically free $\C[[\hbar]]$-module, $V$ an $\hbar$-adically quasi-compatible and $Y_\E^\phi$-closed
$\C[[\hbar]]$-submodule of $\E_\hbar(W)$ such that $1_W\in V$, $[V]=V$, and $\overline{V}=V$.
Then $(V,Y_\E^\phi,1_W)$ carries the structure of an $\hbar$-adic nonlocal vertex algebra and
$W$ is a faithful $\phi$-coordinated quasi module
with $Y_W^\phi(\al(x),z)=\al(z)$ for $\al(x)\in V$.
Furthermore, if $V$ is $\hbar$-adically compatible, $W$ is a $\phi$-coordinated module.
\end{thm}

\begin{proof}
Since $\E_\hbar(W)$ is topologically free, it follows from Lemma \ref{lem:topo-op} that $V$ is topologically free.
For any $n\in \Z_{+}$, as $\pi_n(a(x))_m\pi_n(b(x))=\pi_n(a(x)_mb(x))$ for $a(x),b(x)\in V,m\in\Z$,
we see that $\pi_n(V)$ is $Y_\E^\phi$-closed quasi-compatible $\C[\hbar]$-submodule
of $\E(W/\hbar^n W)$ with $1_W\in \pi_n(V)$.
It follows from \cite[Theorem 4.8]{Li-phi-coor} that $\pi_n(V)$ is a nonlocal vertex algebra over $\C$ with $W/\hbar^nW$
as a $\phi$-coordinated quasi module.

Now we prove that the map $\pi_n$ from $V$ to $\pi_n(V)$ reduces to a $\C[[\hbar]]$-isomorphism
from $V/\hbar^n V$ onto $\pi_n(V)$, so that $V/\hbar^n V$
is a nonlocal vertex algebra over $\C$.
Let $a(x)\in V$ such that $\pi_n(a(x))=0$ in $\E(W/\hbar^nW)$.
Then $a(x)W\subset \hbar^nW[[x,x\inverse]]$.
That is, $a(x)=\hbar^nb(x)$ for some $b(x)\in (\te{End }W)[[x,x\inverse]]$.
By \cite[Lemma 4.6]{Li-h-adic}, $b(x)\in \E_\hbar(W)$.
Then we have $b(x)\in [V]=V$.
Thus $a(x)=\hbar^nb(x)\in \hbar^nV$.
This proves that $V\cap\ker\pi_n=\hbar^nV$, which implies $V/\hbar^nV\cong\pi_n(V)\subset\E(W/\hbar^nW)$.
Consequently, $V/\hbar^nV$ is a nonlocal vertex algebra over $\C$.
In view of Proposition \ref{h-adic-quotient}, $V$ is an $\hbar$-adic nonlocal vertex algebra.
By Proposition \ref{prop:phi-coor-ind}, $W$ is a $\phi$-coordinated quasi $V$-module.
The last part follows from \cite[Theorem 4.8]{Li-phi-coor}.
\end{proof}

With Theorem \ref{prop:abs-construct},  we formulate the following notion:

\begin{de}
Let $W$ be a topologically free $\C[[\hbar]]$-module.
Define a {\em $\phi$-coordinated $\hbar$-adic nonlocal vertex subalgebra}
of $\E_\hbar(W)$ to be an $\hbar$-adically quasi-compatible and $Y_\E^\phi$-closed $\C[[\hbar]]$-submodule $V$
of $\E_\hbar(W)$ such that $1_W\in V$, $[V]=V$, and $\overline{V}=V$.
\end{de}

We have the following technical result:

\begin{lem}\label{lem:bar-U-quasi-comp-Y-closed}
Let $W$ be a topologically free $\C[[\hbar]]$-module and let $U$ be an $\hbar$-adically
quasi-compatible, $Y_\E^\phi$-closed $\C[[\hbar]]$-submodule of $\E_\hbar(W)$.
Then both $[U]$ and $\overline{U}$ are $\hbar$-adically quasi-compatible and $Y_\E^\phi$-closed.
\end{lem}

\begin{proof} By using Lemma \ref{lem:E-h}, it is straightforward to show that $[U]$ is $\hbar$-adically quasi-compatible
and $Y_\E^\phi$-closed.
Let $a_1(x),\dots,a_r(x)\in \overline U$. For each $1\le j\le r$, there exist $a_{jn}(x)\in U$ for $n\in \N$ such that
$a_j(x)-a_{jn}(x)\in \hbar^n \E_\hbar(W)$.  Then
\begin{align}\label{r-product}
a_1(x_1)\cdots a_r(x_r)\in a_{1n}(x_1)\cdots a_{rn}(x_r)+\hbar^n \E_\hbar(W)
\  \   \   \mbox{ for }n\in \N.
\end{align}
As $U$ is $\hbar$-adically quasi-compatible, it follows that the sequence $(a_1(x),\dots,a_r(x))$ is
$\hbar$-adically quasi-compatible.
Therefore, $\overline{U}$ is $\hbar$-adically quasi-compatible.

Now, let $a(x),b(x)\in \overline{U}$. There exist $a_{n}(x),b_n(x)\in U$ for $n\in \N$ such that
$a(x)-a_n(x),\ b(x)-b_n(x)\in \hbar^n \E_\hbar(W)$.
For any $n\in \N,\ m\in \Z$,  we have
\begin{align*}
 \pi_n(a(x)_m^\phi b(x))=  \pi_n(a(x))_m^\phi\pi_n(b(x))=\pi_n(a_n(x))_m^\phi\pi_n(b_n(x))
  =\pi_n(a_n(x)_m^\phi b_n(x)).
 \end{align*}
As $U$ is $Y_\E^\phi$-closed, we get $ \pi_n(a(x)_m^\phi b(x))=\pi_n(a(x)_m^\phi b(x))\in \pi_n(U)$.
Then
\begin{align*}
  a(x)_m^\phi b(x)=\varprojlim_n \pi_n(a(x))_m^\phi \pi_n(b(x))\in \overline U.
\end{align*}
This proves that $\overline U$ is $Y_\E^\phi$-closed.
\end{proof}

Let $W$ be a topologically free $\C[[\hbar]]$-module
and let $U$ be an $\hbar$-adically (quasi) compatible subset of $\E_\hbar(W)$.
Set $U^{(1)}=\C[[\hbar]]U+\C[[\hbar]]1_W$, which is an $\hbar$-adically (quasi) compatible $\C[[\hbar]]$-submodule containing $1_W$.
Define $U^{(2)}$ to be the $\C[[\hbar]]$-span of the vectors $\al(x)_m^\phi\be(x)$ for $\al(x),\be(x)\in U^{(1)}$, $m\in\Z$.
Note that $U^{(1)}\subset U^{(2)}$. From Lemma \ref{lem:quasi-compatible}, $U^{(2)}$ is again an $\hbar$-adically (quasi) compatible
$\C[[\hbar]]$-submodule containing $1_W$.
Then  we inductively define $U^{(n+1)}=\(U^{(n)}\)^{(2)}$ for $n\ge 1$.
In this way, we obtain an ascending sequence of $\hbar$-adically quasi-compatible $\C[[\hbar]]$-submodules:
\begin{align*}
  U^{(1)}\subset U^{(2)}\subset U^{(3)}\subset\cdots.
\end{align*}
Then set
\begin{align}
  \<U\>_\phi=\overline{\left[\cup_{n\ge 1}U^{(n)}\right]}.
\end{align}
We have:

\begin{thm}\label{thm:abs-construct}
Let $W$ be a topologically free $\C[[\hbar]]$-module
and let $U$ be an $\hbar$-adically (quasi) compatible subset of $\E_\hbar(W)$.
Then $\<U\>_\phi$ is the smallest  $\hbar$-adically (quasi) compatible and $Y_\E^\phi$-closed $\C[[\hbar]]$-submodule
 such that
 $$\{1_W\}\cup U\subset \<U\>_\phi,\   \   [\<U\>_\phi]=\<U\>_\phi\  \te{ and }\ \overline{\<U\>_\phi}=\<U\>_\phi.$$
Furthermore, $\(\<U\>_\phi,Y_\E^\phi,1_W\)$ carries the structure of an $\hbar$-adic nonlocal vertex algebra
and $W$ is a faithful $\phi$-coordinated
(quasi) $\<U\>_\phi$-module with $Y_W^\phi(\al(x),z)=\al(z)$ for $\al(x)\in \<U\>_\phi$.
\end{thm}

\begin{proof} From the construction of the ascending sequence $\{U^{(n)}\}$,
$\cup_{n\ge 1}U^{(n)}$ is $\hbar$-adically quasi-compatible and $Y_\E^\phi$-closed.
Furthermore, from Lemma \ref{lem:bar-U-quasi-comp-Y-closed},
$\<U\>_\phi$ is $\hbar$-adically quasi-compatible and $Y_\E^\phi$-closed.
On the other hand,  from Lemma \ref{lem:topo-op} we have
$[\<U\>_\phi]=\<U\>_\phi$ and $\overline{\<U\>_\phi}=\<U\>_\phi$. It is clear that $\{1_W\}\cup U\subset \<U\>_\phi$.
Now, suppose that $K$ is an $\hbar$-adically (quasi) compatible and $Y_\E^\phi$-closed $\C[[\hbar]]$-submodule
 such that $\{1_W\}\cup U\subset K,\ [K]=K$ and $\overline{K}=K$. As $\{1_W\}\cup U\subset K$, we have
 $U^{(n)}\subset K$ for all $n\ge 1$. Then
 $$\left[\cup_{n\ge 1}U^{(n)}\right]\subset [K]=K,$$
 and hence
 $$\overline{\left[\cup_{n\ge 1}U^{(n)}\right]}\subset \overline{K}=K.$$
Thus  $\<U\>_\phi\subset K$. This proves the first assertion.
The furthermore assertion immediately follows from Theorem \ref{prop:abs-construct}.
\end{proof}

\begin{rem}\label{prop:Y-E-hom-h-adic}
{\em Let $(W,Y_W^\phi)$ be a $\phi$-coordinated quasi module for an $\hbar$-adic nonlocal vertex algebra $V$.
By definition, for $u,v\in V$, $(Y_W^{\phi}(u,x),Y_W^{\phi}(v,x))$ is $\hbar$-adically quasi compatible.
As a straightforward $\hbar$-analogue of \cite[Proposition 4.11]{Li-phi-coor}, we have
\begin{align}
  Y_W^\phi(Y(u,z)v,x)=Y_\E^\phi(Y_W^\phi(u,x),z)Y_W^\phi(v,x)\quad \text{ for }u,v\in V.
\end{align}}
\end{rem}

The following is an immediate $\hbar$-analogue of a result of \cite{JKLT-G-phi-mod}:

\begin{prop}\label{prop:tech-calculations}
Let $W$ be a topologically free $\C[[\hbar]]$-module
and let $V$ be a $\phi$-coordinated $\hbar$-adic nonlocal vertex subalgebra of $\E_\hbar(W)$.
Assume
\begin{align*}
 \al(x),\be(x),\al_i(x),\be_i(x), \mu_j(x),\nu_j(x)\in V,
 \  g_i(x_1,x_2),h_j(x_1,x_2)\in\C_\ast((x_1,x_2))[[\hbar]]
\end{align*}
for $1\le i \le r,\  1\le j\le s$. Let $\Gamma$ be a subgroup of $\C^\times$.
Suppose that
\begin{align*}
 & \sum_{i=1}^r\iota_{x_1,x_2,\hbar}\(g_i(x_1,x_2)\)\al_i(x_1)\be_i(x_2)
  -\sum_{j=1}^s\iota_{x_2,x_1,\hbar}\(h_j(x_1,x_2)\)
    \mu_j(x_2)\nu_j(x_1)\\
=&\sum_{\alpha\in \Gamma}\sum_{k=0}^\infty
  \gamma_{\alpha,k}(x_2)\frac{1}{k!}\(p(x_2)\partial_{x_2}\)^kp(x_1)x_1\inverse \delta\(\frac{\alpha^{-1}x_2}{x_1}\),
\end{align*}
where $\gamma_{\alpha,k}(x)\in V$, and for every positive integer $n$,
$\gamma_{\alpha,k}(x)\in \hbar^n V$ for all but finitely many $(\alpha,k)\in \Gamma\times \N$.
Then
\begin{align*}
 &\sum_{i=1}^r\iota_{x,x_1,x_2,\hbar}\(g_i(\phi(x,x_1),\phi(x,x_2))\)
  Y_\E^\phi(\al_i(x),x_1)Y_\E^\phi(\be_i(x),x_2)\\
 &\   \   \  \   -\sum_{j=1}^s\iota_{x,x_2,x_1,\hbar}
    \(h_j(\phi(x,x_1),\phi(x,x_2))\)
    Y_\E^\phi(\mu_j(x),x_2)Y_\E^\phi(\nu_j(x),x_1)\\
  =\  &\sum_{k=0}^\infty
 Y_\E^\phi(\gamma_{1,k}(x),x_2)\frac{1}{k!}\left(\frac{\partial}{\partial x_2}\right)^kx_1\inverse \delta\(\frac{x_2}{x_1}\).
\end{align*}
\end{prop}

\begin{de}
Let $W$ be a topologically free $\C[[\hbar]]$-module.
For $F(x_1,x_2), G(x_1,x_2)\in (\End W)[[x_1^{\pm 1},x_2^{\pm 1}]]$, we write $F(x_1,x_2)\sim G(x_1,x_2)$ if
for every positive integer $n$, there exists $k\in \N$ such that
\begin{align}
(x_1-x_2)^kF(x_1,x_2)\equiv (x_1-x_2)^kG(x_1,x_2)  \  \mod \hbar^n(\End W)[[x_1^{\pm 1},x_2^{\pm 1}]].
\end{align}
\end{de}

The following is an immediate $\hbar$-analogue of \cite[Proposition 4.27]{JKLT-Defom-va}:

\begin{prop}\label{prop:tech-reverse-calculations-h}
Let $V$ be a $(G,\chi)$-module $\hbar$-adic nonlocal vertex algebra and
let $(W,Y_W^\phi)$ be a $(G,\chi_{\phi})$-equivariant $\phi$-coordinated quasi $V$-module.
Assume $u,v\in V$,
\begin{align*}
  u^{(i)},v^{(i)}\in V,\   f_{i}(x_1,x_2)\in\C_{\phi}((x_1,x_2))[[\hbar]]\  \  (1\le i\le r)
\end{align*}
such that
\begin{align*}
  &Y(u,x_1)Y(v,x_2)\sim \sum_{i=1}^r
  \pi_{\phi}(f_{i})(x_2-x_1)Y(v^{(i)},x_2)Y(u^{(i)},x_1).
\end{align*}
Suppose that for each $n\in \Z_{+}$,
$(R(g)u)_jv\in \hbar^nV$ for all but finite many $(g,j)\in G\times \N$.
Then
\begin{align}
  &Y_W^\phi(u,x_1)Y_W^\phi(v,x_2)-
  \sum_{i=1}^r\iota_{x_2,x_1} (f_{i}(x_2,x_1))
  Y_W^\phi(v^{(i)},x_2)Y_W^\phi(u^{(i)},x_1)\\
  =\ &\sum_{\sigma\in G^0}\sum_{j\ge0}
  Y_W^\phi\((R(\sigma)u)_jv,x_2\)\frac{1}{j!}\(p(x_2)\frac{\partial}{\partial x_2}\)^j
  p(x_1)x_1\inverse\delta\(\chi_{\phi}(\sigma)^{-1}\frac{x_2}{x_1}\),\nonumber
\end{align}
where $G^0$ is any complete set of coset representatives of $\ker \chi_{\phi}$ in $G$.
\end{prop}

Combining Corollary \ref{lem:q-Jacobi} and Proposition \ref{prop:tech-reverse-calculations-h} we immediately have:

\begin{prop}\label{prop:reverse-cal}
Let $V$ be a $(G,\chi)$-module $\hbar$-adic quantum vertex algebra with rational quantum Yang-Baxter operator
$\mathcal{S}(x)$ and let $(W,Y_W)$ be a $(G,\chi_{\phi})$-equivariant $\phi$-coordinated quasi $V$-module.
Assume
$$\wh{ \mathcal{S}}(x_1,x_2):\ V\wh\ot V\rightarrow V\wh\ot V\wh\ot \C_{\phi}((x_1,x_2))[[\hbar]]$$
is a continuous $\C[[\hbar]]$-module map such that $\mathcal{S}(z)=(1\ot 1\ot \pi_{\phi})\wh {\mathcal{S}}(x_1,x_2)$.
Then
\begin{align}\label{phi-module-commuator}
  &Y_W(u,x_1)Y_W(v,x_2)w-
  Y_W(x_2)(1\ot Y_W(x_1))\(\iota_{x_2,x_1}\wh{\mathcal{S}}(x_2,x_1)(v\ot u)\ot w\)\\
  =&\sum_{\sigma\in G^{0}}\sum_{j\ge0}
  Y_W\((R(\sigma)u)_jv,x_2\)w\frac{1}{j!}  \(p(x_2)\frac{\partial}{\partial x_2}\)^j
  p(x_1)x_1\inverse\delta\(\chi_{\phi}(\sigma)^{-1}\frac{x_2}{x_1}\)\nonumber
\end{align}
for $u,v\in V,\ w\in W$, where  $G^0$ is given as in Proposition \ref{prop:tech-reverse-calculations-h}.
\end{prop}

We end up this section with the following technical result we shall use:

\begin{lem}\label{tech-extra}
Let $V$ be an $\hbar$-adic nonlocal vertex algebra and let
$$u,v, c^{(r)}\in V,\  \beta(x),\ f(x),g(x)\in \C((x))[[\hbar]]\  \  (\text{where }r\in \N)$$
such that $\lim_{r\rightarrow \infty}c^{(r)}=0$ (with respect to the $\hbar$-adic topology) and
\begin{align}\label{S-bracket-simple}
&Y(u,x_1)Y(v,x_2)-\beta(x_2-x_1)Y(v,x_2)Y(u,x_1)\\
=\ &f(x_1-x_2)+g(x_2-x_1)+\sum_{r\ge 0}Y(c^{(r)},x_2)\frac{1}{j!}
\left(\frac{\partial}{\partial x_2}\right)^rx_1^{-1}\delta\left(\frac{x_2}{x_1}\right).\nonumber
\end{align}
Then
\begin{align}
Y(u,x_1)Y(v,x_2)\sim \beta(x_2-x_1)Y(v,x_2)Y(u,x_1)  +f(-x_2+x_1)+g(x_2-x_1),
\end{align}
\begin{align}
Y(u,x)^{-}v=f(x)^{-}{\bf 1}+\sum_{r\ge 0}c^{(r)}x^{-r-1},
\end{align}
where $Y(u,x)^{-}$ and $f(x)^{-}$ denote the singular parts of $Y(u,x)$ and $f(x)$, respectively.
\end{lem}

\begin{proof} Writing $f(x)^{-}=\sum_{r\ge 0} \lambda_r x^{-r-1}$ with $\lambda_r\in \C[[\hbar]]$, we have
\begin{align*}
f(x_1-x_2) =\ & \sum_{r\ge 0}\lambda_r ((x_1-x_2)^{-r-1}-(-x_2+x_1)^{-r-1})+f(-x_2+x_1)\\
=\ & f(-x_2+x_1)+\sum_{r\ge 0}\lambda_{r}\frac{1}{r!}
\left(\frac{\partial}{\partial x_2}\right)^rx_1^{-1}\delta\left(\frac{x_2}{x_1}\right).
\end{align*}
With this, (\ref{S-bracket-simple}) can be rewritten as
\begin{align*}
&Y(u,x_1)Y(v,x_2)-\(\beta(x_2-x_1)Y(v,x_2)Y(u,x_1)+f(-x_2+x_1)+g(x_2-x_1)\)\\
=\ &\sum_{r\ge 0}Y(c^{(r)}+\lambda_r{\bf 1},x_2)\frac{1}{r!}
\left(\frac{\partial}{\partial x_2}\right)^rx_1^{-1}\delta\left(\frac{x_2}{x_1}\right),\nonumber
\end{align*}
from which we obtain the desired relations.
\end{proof}

\section{Deforming $\hbar$-adic nonlocal vertex algebras}
In this section, we present $\hbar$-adic versions of the basic results
about smash product nonlocal vertex algebras (see \cite{Li-smash}, \cite{JKLT-Defom-va}),
especially the deformation of a nonlocal vertex algebra by a right $H$-comodule nonlocal vertex algebra structure
 with a compatible $H$-module nonlocal vertex algebra structure.

\subsection{Smash product nonlocal vertex algebras}
We here recall the basic concepts and results about smash product nonlocal vertex algebras.

\begin{de}
A {\em nonlocal vertex bialgebra} is a nonlocal vertex algebra $V$ equipped with
a classical coalgebra structure such that the co-multiplication $\Delta:V\rightarrow V\ot V$
and the co-unit $\varepsilon:V\rightarrow \C$ are homomorphisms of nonlocal vertex algebras.
\end{de}

\begin{de}\label{de:mod-va-for-vertex-bialg}
Let $H$ be a nonlocal vertex bialgebra. An {\em $H$-module nonlocal vertex algebra}
is a nonlocal vertex algebra $V$ equipped with a module structure $Y_{V}^H(\cdot,x)$
on $V$ for $H$ viewed as a nonlocal vertex algebra such that
\begin{align}
  &Y_{V}^H(h,x)v\in V\ot \C((x)),\label{eq:mod-va-for-vertex-bialg1}\\
  &Y_{V}^H(h,x)\vac_V=\varepsilon(h)\vac_V,\label{eq:mod-va-for-vertex-bialg2}\\
  &Y_{V}^H(h,x_1)Y(u,x_2)v=\sum Y(Y_{V}^H(h_{(1)},x_1-x_2)u,x_2)Y_{V}^H(h_{(2)},x_1)v
  \label{eq:mod-va-for-vertex-bialg3}
\end{align}
for $h\in H$, $u,v\in V$, where $\vac_V$ denotes the vacuum vector of $V$
and $\Delta(h)=\sum h_{(1)}\ot h_{(2)}$ is the coproduct in the Sweedler notation.
\end{de}

The following is a result of \cite{Li-smash}:

\begin{thm}\label{smash-product}
Let $H$ be a nonlocal vertex bialgebra, $V$ an $H$-module nonlocal vertex algebra.
Set $V\sharp H=V\ot H$ as a vector space.
For $u,v\in V$, $h,h'\in H$, define
\begin{align}
  Y^\sharp (u\ot h,x)(v\ot h')=\sum Y(u,x)Y(h_{(1)},x)v\ot Y(h_{(2)},x)h'.
\end{align}
Then $(V\sharp H, Y^\sharp,{\bf 1}\otimes {\bf 1})$ carries the structure of a nonlocal vertex algebra.
\end{thm}




The following notions were introduced in  \cite{JKLT-Defom-va}:

\begin{de}
Let $H$ be a nonlocal vertex bialgebra.  A {\em right $H$-comodule nonlocal vertex algebra} is
a nonlocal vertex algebra $V$ equipped with a nonlocal vertex algebra homomorphism
 $\rho: V\rightarrow V\ot H$  such that
\begin{align}\label{eq:comod-cond}
(\rho\ot 1)\rho=(1\ot\Delta)\rho,\quad
  (1\ot \epsilon)\rho=\te{id}_V,
\end{align}
i.e., $\rho$ is also a right comodule structure on $V$ for $H$ viewed as a coalgebra.
\end{de}

\begin{de}\label{compatible-ma-cma}
Let $H$ be a nonlocal vertex bialgebra and let $V$ be a right $H$-comodule nonlocal vertex algebra
with comodule structure map $\rho:V\rightarrow V\ot H$.
Denote by $\mathfrak L^\rho_H(V)$ the set consisting of each linear map
\begin{align}
Y_V^H(\cdot,x):\  H\rightarrow \Hom(V,V\ot\C((x)))
\end{align}
such that $V$ with $Y_V^H(\cdot,x)$ is an $H$-module nonlocal vertex algebra and
$\rho$ is an $H$-module homomorphism with $H$ acting on the first factor of $V\otimes H$ only, i.e.,
\begin{align}\label{compatible-relation}
  \rho(Y_V^H(h,x)v)=(Y_V^H(h,x)\ot 1)\rho(v)\quad\te{for }h\in H,\  v\in V.
\end{align}
\end{de}

\begin{thm}\label{thm-deform-va}
Let $H$ be a cocommutative nonlocal vertex bialgebra,
let $(V,\rho)$ be a right $H$-comodule nonlocal vertex algebra,
and let $Y_V^H\in \mathfrak L^\rho_H(V)$.
For $a\in V$, set
\begin{align}
\mathfrak D_{Y_V^H}^\rho (Y)(a,x)=\sum Y(a_{(1)},x)Y_V^H(a_{(2)},x)\  \text{ on }V,
\end{align}
where $\rho(a)=\sum a_{(1)}\ot a_{(2)}\in V\ot H$.
Then $(V,\mathfrak D_{Y_V^H}^\rho (Y),\vac)$ carries the structure of a nonlocal vertex algebra.
Denote this nonlocal vertex algebra by $\mathfrak D_{Y_V^H}^\rho (V)$.
Furthermore, $\mathfrak D_{Y_V^H}^{\rho}(V)$ with the map $\rho$ is also a right $H$-comodule nonlocal vertex algebra.
\end{thm}


\begin{rem}
{\em Let $(V,\rho)$ be a right $H$-comodule nonlocal vertex algebra. Define
a linear map $Y_{M}^\varepsilon(\cdot,x): H\rightarrow \te{End}_{\C} (V)\subset (\te{End}V)[[x,x^{-1}]] $ by
\begin{align}
Y_{M}^{\varepsilon}(h,x)v=\varepsilon(h)v\    \   \  \mbox{ for }h\in H,\  v\in V.
\end{align}
Then $Y_{M}^\varepsilon \in\mathfrak L_H^\rho(V)$ and
$\mathfrak D_{Y_{M}^\varepsilon}^\rho (V)=V$.}
\end{rem}

From definition,  $\mathfrak L^\rho_H(V)$ is a subset of $\Hom (H,\Hom(V,V\ot\C((x))))$.
For $Y_M,Y_M'\in \Hom (H,\Hom(V,V\ot\C((x))))$,
we say that $Y_M$ and $Y_M'$ {\em commute} if
\begin{align}
Y_{M}(h,x)Y_{M}'(k,z)=Y_{M}'(k,z)Y_{M}(h,x)\   \   \mbox{ for all }h,k\in H.
\end{align}
 We have:

\begin{lem}\label{lem:L-H-rho-V-compostition}
Let $H$ be a  nonlocal vertex bialgebra and
let $(V,\rho)$ be a right $H$-comodule nonlocal vertex algebra.
For
$$Y_M(\cdot,x),\ Y_M'(\cdot,x)\in \Hom (H,\Hom(V,V\ot\C((x)))),$$
define a linear map
 $ (Y_M\ast Y_M')(\cdot,x):\  H\rightarrow  \Hom(V,V\ot\C((x)))$
by
\begin{align}
  (Y_M\ast Y_M')(h,x)= \sum Y_M(h_{(1)},x)Y'_M(h_{(2)},x)
  \end{align}
  for $h\in H$, where $\Delta(h)=\sum h_{(1)}\otimes h_{(2)}$.
Then $\Hom\left(H,\Hom(V,V\ot\C((x)))\right)$ with multiplication $*$ is an associative algebra
 with $Y_{M}^\varepsilon$  as identity.
Let $Y_M,Y_M'\in \mathfrak L_H^\rho(V)$. If $H$ is cocommutative
 and if $Y_M$ and $Y_M'$ commute, then $Y_M\ast Y_M'\in \mathfrak L_H^\rho(V)$.
\end{lem}

The following was also proved in \cite{JKLT-Defom-va}:

\begin{thm}\label{thm:S-op}
Let $H$ be a cocommutative vertex bialgebra and let $(V,\rho)$ be a right $H$-comodule vertex algebra.
Assume $Y_M(\cdot,x)$ is an invertible element of $\mathfrak L_H^\rho(V)$  with inverse $Y_M^{-1}(\cdot,x)$.
Define a linear map $\cS(x): V\otimes V\rightarrow V\otimes V\ot \C((x))$ by
\begin{align}
  \cS(x)(v\ot u)=\sum Y_M(u_{(2)},-x)v_{(1)}\ot Y_M\inverse (v_{(2)},x)u_{(1)}
\end{align}
for $u,v\in V$.
Then $\cS(x)$ is a unitary rational quantum Yang-Baxter operator and
the nonlocal vertex algebra $\mathfrak D_{Y_M}^\rho (V)$ with $\cS(x)$ is a quantum vertex algebra.
\end{thm}

For the rest of this section, we assume that $G$ is a group with linear characters $\chi$
and $\chi_\phi$, satisfying the compatibility condition (\ref{p(x)-compatibility}).

\begin{de}
A {\em $(G,\chi)$-module nonlocal vertex bialgebra} is a nonlocal vertex bialgebra $H$
equipped with a representation $R$ of $G$ on $H$ such that  $G$ acts as a coalgebra automorphism group
and $(H,R)$ is a $(G,\chi)$-module nonlocal vertex algebra.
\end{de}

\begin{prop}\label{smash-G-algebra}
Let $H$ be a $(G,\chi)$-module nonlocal vertex bialgebra.
Suppose that $V$ is a $(G,\chi)$-module nonlocal vertex algebra and an $H$-module nonlocal vertex algebra
with $H$-module structure map $Y_V^H(\cdot,x)$ such that
\begin{eqnarray}\label{RV-RH-compatibility}
R_{V}(g)Y_V^H(h,x)v=Y_V^H(R_{H}(g) h,\chi(g)x)R_{V}(g) v
\end{eqnarray}
for $g\in G,\ h\in H,\ v\in V$, where $R_{H}$ and $R_{V}$ denote the corresponding representations of $G$. Then
$V\sharp H$ is a $(G,\chi)$-module nonlocal vertex algebra with $R=R_V\ot R_H$.
\end{prop}

Furthermore, we have:

\begin{prop}\label{dual-vertex-G-algebra}
Let $H,V$ be given as in Proposition \ref{smash-G-algebra}.
In addition, assume that $(V,\rho)$ is a right $H$-comodule nonlocal vertex algebra such that
$Y_V^H\in \mathfrak L_H^\rho(V)$ and $\rho: V\rightarrow V\ot H$ is a $G$-module morphism.
Then $\mathfrak D_{Y_V^H}^\rho (V)$ is a $(G,\chi)$-module nonlocal vertex algebra
with the same map $R$.
\end{prop}

\begin{de}\label{NEW-G-phi-compatible}
Let $H$ be a nonlocal vertex bialgebra and let $(V,Y_V^H)$ be an $H$-module nonlocal vertex algebra.
The module structure $Y_V^H$ is said to be {\em $(\phi,\chi_{\phi})$-compatible}
if for $h\in H,\ v\in V$,
\begin{align}\label{G-phi-compatible0-new}
 Y_V^H(h,z)v\in V\ot \C_{\phi}((z))
\end{align}
and there exists $q(x_1,x_2)\in \C_{\chi_{\phi}(G)}[x_1,x_2]$ such that
\begin{align}
q(x_1,x_2)\wh Y_V^H(h,x_1,x_2)v\in &V\ot \C((x_1,x_2)),
\label{G-phi-compatible2-new}
\end{align}
where
\begin{align}\label{G-phi-compatible1-new}
\widehat{Y}_V^H(h,x_1,x_2)v= (1\ot \pi_{\phi}^{-1})(Y_V^H(h,z)v).
\end{align}
\end{de}

\begin{de}\label{de:G-phi-compatible-mod}
Let $H$ be a $(G,\chi)$-module nonlocal vertex bialgebra, $V$ a $(G,\chi)$-module nonlocal vertex algebra.
Assume that $V$ is also an $H$-module nonlocal vertex algebra where the module structure $Y_V^H$ is $(\phi,\chi_{\phi})$-compatible.
A {\em $(G,\chi_{\phi})$-equivariant $\phi$-coordinated quasi $(H,V)$-module}
is a $(G,\chi_{\phi})$-equivariant $\phi$-coordinated quasi module $W$ for both $H$ and $V$ such that
\begin{align}
&Y_W^H(h,x)w\in W\ot\C((x)),
    \label{eq:G-phi-compatible7}\\
&Y_W^H(h,x_1)Y_W^V(v,x_2)=\sum Y_W^V\( \iota_{x_1,x_2}\wh{Y}_V^H(h_{(1)},
    x_1,x_2)v,x_2\)Y_W^H(h_{(2)},x_1)
    \label{eq:G-phi-compatible8}
\end{align}
for  $h\in H$, $v\in V,$ $w\in W$, where $Y_W^H(\cdot,x)$ and $Y_W^V(\cdot,x)$
denote the module vertex operator maps for $H$ and $V$ on $W$.
\end{de}

We have (see \cite{JKLT-Defom-va}):

\begin{thm}\label{prop:deform-phi-quasi-mod}
Assume that $H$ is a $(G,\chi)$-module nonlocal vertex bialgebra, $V$ is a $(G,\chi)$-module nonlocal vertex algebra, and
 $V$ is also an $H$-module nonlocal vertex algebra with $(\phi,\chi_{\phi})$-compatible module map $Y_V^H$
such that for $g\in G,\ h\in H,\ v\in V$,
\begin{eqnarray*}
R_{V}(g)Y_V^H(h,x)v=Y_V^H(R_{H}(g) h,\chi(g)x)R_{V}(g) v.
\end{eqnarray*}
In addition, assume that $V$ is a right $H$-comodule nonlocal vertex algebra
 such that $Y_V^H\in \mathfrak L_H^\rho(V)$.
Let $(W,Y_W^H,Y_W^V)$ be a $(G,\chi_{\phi})$-equivariant $\phi$-coordinated quasi $(H,V)$-module.
Then $(W,Y_W^{\sharp\rho})$ is a $(G,\chi_{\phi})$-equivariant $\phi$-coordinated quasi
$\mathfrak D_{Y_V^H}^\rho (V)$-module, where for $v\in V$,
\begin{align}
Y_{W}^{\sharp \rho}(v,x)=\sum Y_{W}^V(v_{(1)},x)Y_{W}^H(v_{(2)},x) \in \E(W).
\end{align}
\end{thm}

\subsection{Smash product $\hbar$-adic nonlocal vertex algebras}
We here formulate the corresponding $\hbar$-adic versions.
First, an {\em $\hbar$-adic nonlocal vertex bialgebra} is an $\hbar$-adic nonlocal vertex algebra $H$ with an ($\hbar$-adic)
coalgebra structure $(\Delta,\epsilon)$ on $H$ over $\C[[\hbar]]$ such that both $\Delta: H\rightarrow H\wh\otimes H$ and
$\epsilon: H\rightarrow \C[[\hbar]]$ are homomorphisms of $\hbar$-adic nonlocal vertex algebras.
Note that for $a\in H$, $\Delta(a)=\sum a_{(1)}\ot a_{(2)}$ is an infinite but convergent sum
with respect to the $\hbar$-adic topology.

\begin{lem}\label{h-bialgebra-limit}
Let $H$ be an $\hbar$-adic nonlocal vertex algebra, and let $\Delta: H\rightarrow H\wh\otimes H$ and
$\epsilon: H\rightarrow \C[[\hbar]]$ be $\C[[\hbar]]$-module maps.
Then $H$ with $\Delta$ and $\epsilon$ is an $\hbar$-adic nonlocal vertex bialgebra if and only if for every positive integer $n$,
$H/\hbar^{n}H$ is a nonlocal vertex bialgebra over $\mathcal{R}_n\ (=\C[\hbar]/\hbar^n\C[\hbar])$.
\end{lem}

\begin{de}
Let $H$ be an $\hbar$-adic nonlocal vertex bialgebra.
An {\em $H$-module $\hbar$-adic nonlocal vertex algebra} is an $\hbar$-adic nonlocal vertex algebra $V$
with a module structure $Y_{V}^{H}(\cdot,x)$ on $V$ for $H$ viewed as an $\hbar$-adic nonlocal vertex algebra such that
\begin{align*}
Y_V^H(a,x)v\in V\wh\otimes \C((x))[[\hbar]]\quad \text{ for }a\in V,\ v\in V
\end{align*}
and such that the other two conditions (\ref{eq:mod-va-for-vertex-bialg2}) and (\ref{eq:mod-va-for-vertex-bialg3})
in Definition \ref{de:mod-va-for-vertex-bialg} hold.
\end{de}

\begin{lem}
Let $H$ be an $\hbar$-adic nonlocal vertex bialgebra, $V$ an $\hbar$-adic nonlocal vertex algebra,
and $Y_{V}^{H}(\cdot,x): H\rightarrow ({\rm End}_{\C[[\hbar]]} V)[[x,x^{-1}]]$ a $\C[[\hbar]]$-module map. Then
$V$ with $Y_{V}^{H}(\cdot,x)$ is an $H$-module $\hbar$-adic nonlocal vertex algebra if and only if
$V/\hbar^{n}V$ is an $H/\hbar^{n}H$-module nonlocal vertex algebra over $\mathcal{R}_n$  for every positive integer $n$.
\end{lem}




\begin{rem}\label{h-adic-smash-product}
{\em We have the $\hbar$-adic version of Theorem \ref{smash-product} with the obvious changes: $H$ and $V$ are
assumed to be an $\hbar$-adic nonlocal vertex bialgebra and an $H$-module $\hbar$-adic nonlocal vertex algebra,
$V\sharp H=V\wh{\ot} H$, a topologically free $\C[[\hbar]]$-module, and the conclusion is that
$V\sharp H$ is an $\hbar$-adic nonlocal vertex algebra.}
\end{rem}



\begin{de}
Let $H$ be an $\hbar$-adic nonlocal vertex bialgebra.
A {\em right $H$-comodule $\hbar$-adic nonlocal vertex algebra} is an $\hbar$-adic nonlocal vertex algebra
$V$ equipped with a right ($\hbar$-adic) $H$-comodule structure $\rho: V\rightarrow V\wh\ot H$, which is also a homomorphism of
$\hbar$-adic nonlocal vertex algebras.
\end{de}

Let $H$ be an $\hbar$-adic nonlocal vertex bialgebra and let $(V,\rho)$ be a
right $H$-comodule $\hbar$-adic nonlocal vertex algebra. Then define $\mathfrak L_H^\rho(V)$
in the obvious way.
We here formulate the detailed $\hbar$-adic version of Theorem \ref{thm-deform-va}.

\begin{thm}\label{prop:deform-va-h}
Let $H$ be a cocommutative $\hbar$-adic nonlocal vertex bialgebra,
let $(V,\rho)$ be a right $H$-comodule $\hbar$-adic nonlocal vertex algebra,
and let $Y_V^H\in \mathfrak L^\rho_H(V)$.
For $a\in V$, set
\begin{align}
\mathfrak D_{Y_V^H}^\rho (Y)(a,x)=\sum Y(a_{(1)},x)Y_V^H(a_{(2)},x)\  \text{ on }V,
\end{align}
where $\rho(a)=\sum a_{(1)}\ot a_{(2)}\in V\wh\ot H$.
Then $(V,\mathfrak D_{Y_V^H}^\rho (Y),\vac)$ carries the structure of an $\hbar$-adic nonlocal vertex algebra.
Denote this $\hbar$-adic nonlocal vertex algebra by $\mathfrak D_{Y_V^H}^\rho (V)$.
Furthermore, $\mathfrak D_{Y_V^H}^{\rho}(V)$ with the same map $\rho$
is also a right $H$-comodule $\hbar$-adic nonlocal vertex algebra.
\end{thm}

The following is an $\hbar$-adic version of a result of \cite[Lemma 3.8]{JKLT-Defom-va}:

\begin{lem}\label{prop:gen-set}
Let $Y_M$ be an invertible element  of $\mathfrak L_H^\rho(V)$.
Assume that $S$ and $T$ are generating $\C[[\hbar]]$-submodules of $V$ and $H$
as $\hbar$-adic nonlocal vertex algebras, respectively, such that
\begin{align}
\rho(S)\subset S\ot T,\quad\Delta(T)\subset T\ot T,\quad
  Y_M\inverse(T,x)S\subset S\wh\ot \C((x))[[\hbar]].
\end{align}
Then $S$ is also a generating $\C[[\hbar]]$-submodule of $\mathfrak D_{Y_M}^\rho (V)$.
\end{lem}


Let $H$ be an $\hbar$-adic nonlocal vertex bialgebra and let $\phi(x,z)=e^{zp(x)\partial_x}x$ with $p(x)\in \C((x))^{\times}$.
An  $H$-module $\hbar$-adic nonlocal vertex algebra $V$ is said to be {\em $(\phi,\chi_{\phi})$-compatible}
if  for any $a\in H,\ v\in V$,
\begin{align}
&Y_V^H(a,z)v\in V\wh\ot \C_{\phi}((z))[[\hbar]],\label{phi-compatible-HV-1}\\
&\wh{Y}_{V}^H(a,x_1,x_2)v\in V\wh\ot \C_{\phi}^{\chi_{\phi}(G)}((x_1,x_2))[[\hbar]],\label{phi-compatible-HV-2}
\end{align}
where
\begin{align}
\wh{Y}_{V}^H(a,x_1,x_2)v=(1\ot \pi_{\phi}^{-1})(Y_V^H(a,z)v).
\end{align}
From definition, an $H$-module $\hbar$-adic nonlocal vertex algebra $V$ is $(\phi,\chi_{\phi})$-compatible
if and only if for every positive integer $n$, $V/\hbar^nV$ is a $(\phi,\chi_{\phi})$-compatible $H/\hbar^nH$-module
nonlocal vertex algebra.



\begin{de}
Let $H$ be a $(G,\chi)$-module $\hbar$-adic nonlocal vertex bialgebra and let $V$ be a $(G,\chi)$-module  and
a $(\phi,\chi_{\phi})$-compatible $H$-module $\hbar$-adic nonlocal vertex algebra.
A {\em $(G,\chi_{\phi})$-equivariant $\phi$-coordinated quasi $(H,V)$-module}
is a $(G,\chi_{\phi})$-equivariant $\phi$-coordinated quasi module for both $H$ and $V$, such that
\begin{align}
&Y_W^H(a,x)w\in W\wh\ot\C((x))[[\hbar]],
    \label{eq:G-phi-compatible7-hadic}\\
&Y_W^H(a,x_1)Y_W^V(v,x_2)w=\sum Y_W^V\( \iota_{x_1,x_2}\wh{Y}_V^H(a_{(1)},
    x_1,x_2)v,x_2\)Y_W^H(a_{(2)},x_1)w
    \label{eq:G-phi-compatible8-hadic}
\end{align}
 for  $a\in H$, $v\in V,$ $w\in W$, where $\Delta(a)=\sum a_{(1)}\ot a_{(2)}\in H\wh\ot H$, and
where $Y_W^H(\cdot,x)$ and $Y_W^V(\cdot,x)$ denote the module vertex operator maps for $H$ and $V$ on $W$.
 \end{de}

Note that the conditions \eqref{eq:G-phi-compatible7-hadic} and \eqref{eq:G-phi-compatible8-hadic}
amount to that for every positive integer $n$, $W/\hbar^n W$ is
a $(G,\chi_{\phi})$-equivariant $\phi$-coordinated quasi $(H/\hbar^nH,V/\hbar^nV)$-module.
Then we immediately have the following $\hbar$-adic version of Theorem \ref{prop:deform-phi-quasi-mod}:

\begin{thm}\label{deform-phi-quasi-mod-h}
Let $H$ be a $(G,\chi)$-module $\hbar$-adic nonlocal vertex bialgebra and
let $(V,\rho)$ be a right $H$-comodule $\hbar$-adic nonlocal vertex algebra.
Assume that $V$ is also a $(G,\chi)$-module and an $H$-module $\hbar$-adic nonlocal vertex algebra
with $H$-module structure map $Y_V^H(\cdot,x)$ which is $(\phi,\chi_{\phi})$-compatible and lies in $\mathfrak D_{Y_V^H}^\rho (V)$,
such that
\begin{eqnarray*}
R_{V}(g)Y_V^H(h,x)v=Y_V^H(R_{H}(g) h,\chi(g)x)R_{V}(g) v\   \   \   \mbox{ for }g\in G,\ h\in H,\ v\in V.
\end{eqnarray*}
Let $(W,Y_W^H,Y_W^V)$ be a $(G,\chi_{\phi})$-equivariant $\phi$-coordinated quasi $(H,V)$-module.
For $v\in V$, set
\begin{align}
Y_{W}^{\sharp \rho}(v,x)=\sum Y_{W}^V(v_{(1)},x)Y_{W}^H(v_{(2)},x)\in \E_{\hbar}(W).
\end{align}
Then $(W,Y_W^{\sharp\rho})$ is a $(G,\chi_{\phi})$-equivariant $\phi$-coordinated quasi
$\mathfrak D_{Y_V^H}^\rho (V)$-module.
\end{thm}

%
%

%
%
%

\section{$\hbar$-adic deformations of lattice vertex algebra $V_L$}

In this section, we use a particular vertex bialgebra which was exploited in \cite{Li-smash} and \cite{JKLT-Defom-va},
to obtain a family of formal deformations of lattice vertex algebras.

\subsection{Vertex algebra $V_{L}$ and vertex bialgebra $B_L$}
We first recall the lattice vertex algebra $V_L$ (see  \cite{bor}, \cite{FLM2}).
Let $L$ be a finite rank  even lattice, i.e., a free abelian group of finite rank equipped with a
symmetric $\Z$-valued bilinear form $\<\cdot,\cdot\>$ such that $\<\alpha,\alpha\>\in 2\Z$ for $\alpha\in L$.
Assume that $L$ is non-degenerate.
Set
\begin{align}
\h=\C\ot_{\Z}L
\end{align}
 and extend $\<\cdot,\cdot\>$ to a symmetric $\C$-valued bilinear form on $\h$.
View $\h$ as an abelian Lie algebra with $\<\cdot,\cdot\>$ as a non-degenerate symmetric invariant bilinear form.
Then we have an affine Lie algebra $\wh \h$, where
$$\wh \h=\h\otimes \C[t,t^{-1}]+\C {\bf k}$$
as a vector space, and where ${\bf k}$ is central and
\begin{eqnarray}
[\alpha(m),\beta(n)]=m\delta_{m+n,0}\<\alpha,\beta\>{\bf k}
\end{eqnarray}
for $\alpha,\beta\in \h,\ m,n\in \Z$ with $\alpha(m)$ denoting $\alpha\otimes t^{m}$.  Set
\begin{eqnarray}
\wh \h^{\pm}=\h\otimes t^{\pm 1} \C[t^{\pm 1}],
\end{eqnarray}
which are abelian Lie subalgebras. Identify $\h$ with $\h\otimes t^{0}\subset \wh \h$.
Furthermore, set
\begin{eqnarray}
\wh \h'=\wh \h^{+}+\wh \h^{-}+\C {\bf k},
\end{eqnarray}
which is a Heisenberg algebra. Then $\wh \h=\wh \h'\oplus \h$, a direct sum of Lie algebras.

Let $\varepsilon:L\times L\rightarrow \C^\times$ be a 2-cocycle such that
\begin{eqnarray*}
\varepsilon(\alpha,\beta)\varepsilon(\beta,\alpha)^{-1}=(-1)^{\<\al,\be\>},\   \   \
\varepsilon(\alpha,0)=1=\varepsilon(0,\alpha)\quad \te{ for } \alpha,\beta \in L.
\end{eqnarray*}
Denote by $\C_{\varepsilon}[L]$ the $\varepsilon$-twisted group algebra of $L$,
which by definition has a designated basis $\{ e_{\alpha}\ |\ \alpha\in L\}$ with
\begin{align}
e_\al \cdot e_\be=\varepsilon(\al,\be)e_{\al+\be}\quad \te{ for }\al,\be\in L.
\end{align}
Make $\C_{\varepsilon}[L]$ an $\wh \h$-module by letting $\wh \h'$
act trivially and letting $\h$ $(=\h(0))$ act by
\begin{eqnarray}
  h(0)e_\be=\<h,\be\>e_\be \   \  \mbox{ for }h\in \h,\  \be\in L.
  \end{eqnarray}

Note that $S(\wh \h^{-})$ is naturally an $\widehat \h$-module of level $1$.
Set
\begin{eqnarray}
V_{L}=S(\wh \h^{-})\otimes \C_{\varepsilon}[L],
\end{eqnarray}
the tensor product of $\widehat \h$-modules, which is an $\widehat \h$-module of level $1$.
Set
$$\vac=1\ot e_0\in  V_{L}.$$
Identify $\h$ and $\C_{\varepsilon}[L]$ as subspaces of $V_{L}$ via the correspondence
$$a\mapsto a(-1)\otimes e_0\   (a\in\h) \te{ and } e_\al\mapsto 1\otimes e_\al\   (\al\in L).$$

For $h\in \h$, set
\begin{align}
h(z)=\sum_{n\in\Z}h(n)z^{-n-1}.
\end{align}
On the other hand, for $\al\in L$ set
\begin{align}
E^{\pm}(\alpha,z)=\exp \left(\sum_{n\in \Z_{+}}\frac{\al(\pm n)}{\pm n}z^{\mp n}\right)
\end{align}
on $V_{L}$.  For $\al\in L$,  define a linear operator
$z^{\al}:\   \C_{\varepsilon}[L]\rightarrow \C_{\varepsilon}[L][z,z^{-1}]$ by
\begin{eqnarray}
z^\al\cdot e_\be=z^{\<\al,\be\>}e_\be\   \   \   \mbox{ for }\be\in L.
\end{eqnarray}
Then  there exists a vertex algebra structure on $V_L$, which is uniquely determined by
the conditions that $\vac$ is the vacuum vector and that
\begin{align}
&Y(h,z)=h(z)\quad \te{ for }h\in \h,\\
&Y(e_\al,z)=E^{-}(-\al,z)E^{+}(-\al,z)e_\al z^\al\quad \te{ for }\al\in L.
\end{align}

As we shall need, we recall from \cite[Section 6.5]{LL} the associative algebra $A(L)$.
By definition, $A(L)$ is the associative algebra with identity $1$ over $\C$, generated by
$$\set{h[n],\  e_\al[n]}{h\in\h,\ \al\in L,\ n\in \Z},$$
where $h[n]$ is linear in $h$, subject to a set of relations written in terms of generating functions
\begin{align*}
  h[z]=\sum_{n\in\Z}h[n]z^{-n-1},\   \   \  \
  e_\al[z]=\sum_{n\in\Z}e_\al[n]z^{-n-1}.
\end{align*}
The relations are
\begin{align*}
  &\te{(AL1) }\  \   \   \   e_0[z]=1,\\
  &\te{(AL2) }\   \   \   \    \left[h[z_1],
    h'[z_2]\right]=\<h,h'\>\partial_{z_2}z_1\inverse\delta\(\frac{z_2}{z_1}\),\\
  &\te{(AL3) }\    \   \  \    \left[h[z_1],e_\al[z_2]\right]=\<\al,h\>
    e_\al[z_2]z_1\inverse\delta\(\frac{z_2}{z_1}\),\\
  &\te{(AL4) }\    \   \   \   \left[e_\al[z_1],e_\be[z_2]\right]=0\   \te{ if }\<\al,\be\>\ge0,\\
  &\te{(AL5) }\   \   \   \  (z_1-z_2)^{-\<\al,\be\>}\left[e_\al[z_1],e_\be[z_2]\right]=0\   \te{ if }\<\al,\be\><0,
\end{align*}
for $h, h'\in \h,\  \al,\be\in L$.
An $A(L)$-module $W$ is said to be {\em restricted} if for every $w\in W$, we have
$ h[z]w, \  e_\al[z]w\in W((z))$ for all $h\in\h,\  \al\in L$.

The following result was obtained in  \cite{LL}:

\begin{prop}\label{prop:VQ-mod-AQ-mod}
 For any $V_L$-module $(W,Y_W)$, $W$ is a restricted $A(L)$-module with
\begin{align*}
  h[z]=Y_W(h,z),\   \  \   \,\,e_\al[z]=Y_W(e_\al,z)
\end{align*}
for  $h\in\h,\   \al\in L$, such that the following relations hold on $W$ for $\al,\be\in L$:
\begin{align*}
&\te{(AL6) }\   \   \   \   \frac{d}{dz} e_\al[z]=\al[z]^{+} e_\al[z]+e_\al[z]\al[z]^{-},\\
  &\te{(AL7) }\   \   \    \   \Res_{x}\big(
    (x-z)^{-\<\al,\be\>-1} e_\al[x] e_\be[z]-(-z+x)^{-\<\al,\be\>-1} e_\be[z] e_\al[x]\big)\\
  &\hspace{1.75cm}
  =\varepsilon(\al,\be) e_{\al+\be}[z],
\end{align*}
where $\al[z]^{+}=\sum_{n<0}\al[n]z^{-n-1}$ and $\al[z]^{-}=\sum_{n\ge 0}\al[n]z^{-n-1}$.
On the other hand, if $W$ is a restricted $A(L)$-module satisfying (AL6) and (AL7),
then $W$ admits a  $V_L$-module structure $Y_{W}(\cdot,z)$ which is uniquely determined by
$$Y_W(h,z)= h[z],\   \   \    Y_W(e_\al,z)= e_\al[z]\   \  \   \mbox{ for }h\in\h,\   \al\in L.$$
\end{prop}

The following is a universal property of $V_{L}$ (see \cite{JKLT-G-phi-mod}, \cite{JKLT-Defom-va}):

\begin{prop}\label{prop:AQ-mod-hom-VA-hom}
Let $V$ be a nonlocal vertex algebra
and let $\psi:\h\oplus\C_{\varepsilon}[L]\rightarrow V$ be a linear map such that
$\psi(e_0)={\bf 1}$, the relations (AL1-3) and (AL6) hold with
\begin{align*}
  h[z]= Y(\psi(h),z),\quad e_\al[z]= Y(\psi(e_\al),z)\quad
  \te{for }h\in\h,\  \al\in L,
\end{align*}
and such that the following relation holds for $\al,\be\in L$:
\begin{align}\label{A(L)457}
&(x-z)^{-\<\al,\be\>-1}e_{\alpha}[x]e_{\be}[z]
-(-z+x)^{-\<\al,\be\>-1}e_{\be}[z]e_{\alpha}[x]\nonumber\\
&\quad \quad \quad =  \varepsilon(\al,\be)e_{\al+\be}[z]x^{-1}\delta\left(\frac{z}{x}\right).
\end{align}
Then $\psi$ extends uniquely to a nonlocal vertex algebra morphism from $V_L$ to $V$.
\end{prop}

View $L$ as an abelian group and let $\C[L]$ be its group algebra with basis $\set{e^\al}{\al\in L}$.
Recall that $\h=\C\otimes_{\Z}L$, a vector space, and
$\wh\h^-=\h\otimes t^{-1}\C[t^{-1}]$, an abelian Lie algebra.
Note that both $\C[L]$ and $S(\wh\h^-)$ $(=U(\wh\h^-))$ are Hopf algebras.
Set
\begin{align}
  B_L= S(\wh\h^-)\ot \C[L],
\end{align}
which is a Hopf algebra. In particular, $B_{L}$ is a bialgebra,
where the comultiplication $\Delta$ and counit $\varepsilon$
are uniquely determined by
\begin{align}
  &\Delta(h(-n))=h(-n)\ot 1+1\ot h(-n),\quad \Delta(e^\al)=e^\al\ot e^\al,\\
  &\varepsilon(h(-n))=0,\quad \varepsilon(e^\al)=1
\end{align}
for $h\in\h,\  n\in \Z_{+},\ \al\in L$.
As an algebra, $B_L$ admits a derivation $\partial$ which is uniquely determined by
\begin{align}
  \partial(h(-n))=nh(-n-1),\quad \quad\partial(u\ot e^\al)= \al(-1)u\ot e^\al+\partial u \ot e^\al
\end{align}
for $h\in\h$, $n\in \Z_{+}$, $u\in S(\wh\h^-),\ \al\in L$.
Then by the Borcherds construction, $B_L$ becomes a commutative vertex algebra with
$$Y(a,x)b=(e^{x\partial}a)b\   \   \   \mbox{ for }a,b\in B_{L}.$$
Furthermore,  we have (see \cite{Li-smash}):

\begin{lem}\label{BL-vertex-operator}
The vertex algebra $B_L$ equipped with the coalgebra structure stated above is a commutative and cocommutative vertex bialgebra.
\end{lem}

Recall that vertex algebra $V_{L}=S(\wh\h^-)\ot\C_{\epsilon}[L]$ contains $\h$ and $\C_{\epsilon}[L]$ as subspaces.
The following result was obtained in \cite{JKLT-Defom-va}:

\begin{prop}\label{lem:qva-rho-def}
 On the vertex algebra $V_L$, there exists a right $B_L$-comodule vertex algebra structure
 $\rho: V_L\rightarrow V_L\ot B_L$, which is uniquely determined by
\begin{align}
  \rho(h)=h\ot1+{\bf 1}\ot h(-1),\quad\rho(e_\al)=e_\al\ot e^\al\quad\te{for }h\in\h,\   \al\in L.
\end{align}
\end{prop}

Note that $B_L[[\hbar]]$ is naturally an $\hbar$-adic vertex bialgebra and $V_L[[\hbar]]$ is an $\hbar$-adic vertex algebra.
 By canonical extension view $\rho$ as a $\C[[\hbar]]$-module map
 \begin{align}\label{comodule-map-rho}
 \rho: \  V_L[[\hbar]]\rightarrow V_L[[\hbar]]\wh\ot B_{L}[[\hbar]]\ \ (=(V_L\ot B_L)[[\hbar]]).
 \end{align}
 Then $V_L[[\hbar]]$ with $\rho$ becomes a right $B_L[[\hbar]]$-comodule $\hbar$-adic vertex algebra.

For $a\in \h,\ f(x)\in \C((x))$, define (see \cite{EK-qva})
\begin{align}\label{eq:formula-Phi-pre}
\Phi(a,f)(x)=\sum_{n\ge 0}\frac{(-1)^{n}}{n!}f^{(n)}(x)a_{n}
\end{align}
on $V_{L}$ and on every $V_L$-module, where $f^{(n)}(x)$ denotes the $n$-th derivative.
We have
\begin{align}\label{formula-Phi}
\Phi(a,f)(x)e_{\be}=\<\be,a\>f(x)e_{\be}\   \   \   \mbox{ for }\be\in L.
\end{align}
Define $\Phi(G(x))$ for $G(x)\in \h\otimes \C((x))[[\hbar]]$ in the obvious way.
On the other hand, extend the bilinear form $\<\cdot,\cdot\>$ on $\h$ to
a $\C((x))[[\hbar]]$-valued bilinear form on $\h\otimes \C((x))[[\hbar]]$. Then
\begin{align}\label{Phi-G(x)-e-beta}
\Phi(G(x))e_{\be}=\<\be,G(x)\>e_{\be}
\end{align}
for $G(x)\in \h\otimes \C((x))[[\hbar]],\ \be\in L$.

\begin{de}
Denote by
${\rm Hom}(\h, \h\otimes \C((x))[[\hbar]])^{0}$ the space of linear maps
$\eta: \h\rightarrow \h\otimes \C((x))[[\hbar]]$ such that
$\eta_0(\al,x):=\eta(\al,x)|_{\hbar=0}\in \h\ot x\C[[x]]$ for $\al\in \h$.
\end{de}

As an immediate $\hbar$-adic version of \cite[Proposition 5.8]{JKLT-Defom-va}, we have:

 \begin{lem}\label{lem:qva-Y-M-def}
Let $\eta(\cdot,x)\in {\rm Hom}(\h, \h\otimes \C((x))[[\hbar]])^{0}$.
Then there exists a $B_L[[\hbar]]$-module structure $Y_M^\eta(\cdot,x)$ on $V_{L}[[\hbar]]$,
which is uniquely determined by
\begin{eqnarray}
 Y_M^\eta(e^\al,x) =\exp (\Phi( \eta(\alpha,x)))\   \   \   \mbox{ for } \al\in L,
\end{eqnarray}
and $V_{L}[[\hbar]]$ becomes a $B_L[[\hbar]]$-module $\hbar$-adic (nonlocal) vertex algebra.
Furthermore,  $Y_M^\eta(\cdot,x)$ is an invertible element of $\mathfrak L_{B_L[[\hbar]]}^\rho(V_L[[\hbar]])$.
\end{lem}

The following is a straightforward $\hbar$-adic version of \cite[Theorem 5.9]{JKLT-Defom-va}:

\begin{thm}\label{thm:qlva}
Let $\eta(\cdot,x)\in {\rm Hom}(\h, \h\otimes \C((x))[[\hbar]])^{0}$.
Then we have an $\hbar$-adic quantum vertex algebra $V_{L}[[\hbar]]^\eta$, where $V_{L}[[\hbar]]^\eta=V_L[[\hbar]]$
as a $\C[[\hbar]]$-module and its vertex operator map $Y_L^\eta(\cdot,x)$ is uniquely determined by
\begin{eqnarray}
&&Y_L^{\eta}(e_{\alpha},x)=Y(e_\al,x)\exp(\Phi(\eta(\alpha,x)))\quad \text{ for }\alpha\in L,  \\
&&Y_L^{\eta}(a,x)=Y(a,x)+\Phi(\eta'(a,x)) \quad \text{ for }a\in \h.
\end{eqnarray}
\end{thm}

Furthermore, we have:

\begin{prop}\label{prop-comp-sum}
Let $\eta(\cdot,x)\in {\rm Hom}(\h, \h\otimes \C((x))[[\hbar]])^{0}$.
The $\C[[\hbar]]$-module map $\rho: V_{L}[[\hbar]]\rightarrow V_L[[\hbar]]\wh\ot B_L[[\hbar]]$
makes $V_{L}[[\hbar]]^\eta$ a right $B_L[[\hbar]]$-comodule $\hbar$-adic nonlocal vertex algebra,
and for any $\psi(\cdot,x)\in {\rm Hom}(\h, \h\otimes \C((x))[[\hbar]])^{0}$ we have
\begin{align}
Y_M^{\psi}(\cdot,x)\in \mathfrak L_{B_L[[\hbar]]}^\rho(V_L[[\hbar]]^{\eta})
\end{align}
and
$$\mathfrak D_{Y_M^\psi}^\rho\(\mathfrak D_{Y_M^\eta}^\rho(V_L[[\hbar]])\)
  =\mathfrak D_{Y_M^{\psi+\eta}}^\rho(V_L[[\hbar]]).$$
\end{prop}

We continue to study the $\hbar$-adic quantum vertex algebra $V_{L}[[\hbar]]^\eta$.
For $f(x)\in \C((x))$, let $f(x)^{\pm}$ denote the regular (singular) part of $f(x)$.
Furthermore, for $\al\in \h$, let $\eta(\al,x)^{\pm}$ denote the regular (singular) part of $\eta(\al,x)$.

\begin{prop}\label{lem:commutator-V-L-h-eta}
The following relations hold on $V_L[[\hbar]]^{\eta}$  for $\al, \be\in L$:
\begin{align}
&[Y_{L}^\eta(\al,x_1),Y_{L}^\eta(\be,x_2)]
=\<\al,\be\>\partial_{x_2}x_1\inverse\delta\(\frac{x_2}{x_1}\)\label{new-h-h-relations}\\
&\quad\quad \quad \quad  -\<\eta''(\al,x_1-x_2),\be\> +\<\eta''(\be,x_2-x_1),\al\>, \nonumber\\
&[Y_{L}^\eta(\al,x_1),Y_{L}^\eta(e_\be,x_2)]
=\<\be,\al\>Y_{L}^\eta(e_\be,x_2)x_1\inverse\delta\(\frac{x_2}{x_1}\)\label{new-h-e-beta-relations}\\
&\quad\quad \quad \quad + \(\<\eta'(\al,x_1-x_2),\be\>+ \<\eta'(\beta, x_2-x_1),\al\>\)
    Y_{L}^\eta(e_\be,x_2),\nonumber\\
&e^{-\<\eta(\al,x_1-x_2),\be\>}(x_1-x_2)^{-\<\al,\be\>-1}
Y_{L}^\eta(e_\al,x_1)Y_{L}^\eta(e_\be,x_2)\label{Yeta-ealpha=ebeta}\\
&\quad - e^{-\<\eta(\be,x_2-x_1),\al\>}
  (-x_2+x_1)^{-\<\al,\be\>-1}Y_{L}^\eta(e_\be,x_2)Y_{L}^\eta(e_\al,x_1)\nonumber\\
=\ &  \epsilon(\al,\be)Y_{L}^\eta(e_{\al+\be},x_2)
    x_1\inverse\delta\(\frac{x_2}{x_1}\),\nonumber
\end{align}
 \begin{align}\label{VL-eta-d-dx}
& \frac{d}{dx}Y_{L}^{\eta}(e_{\alpha},x)\\
=\ &Y_L^{\eta}(\al,x)^{+}Y_L^{\eta}(e_{\alpha},x) +Y_L^{\eta}(e_{\alpha},x)Y_L^{\eta}(\al,x)^{-}
+\<\eta'(\al,0)^{+},\al\>Y_L^{\eta}(e_{\alpha},x),\nonumber
\end{align}
where $\eta'(\al,0)^{+}=\eta'(\al,x)^{+}|_{x=0}$.
\end{prop}

\begin{proof} Notice that for $u,v\in \h,\ g(x)\in \C((x))[[\hbar]]$, we have
\begin{align*}
[\Phi(u\ot g)(x),Y(v,z)]= \ &\sum_{n\ge 0,m\in \Z}(-1)^n\frac{1}{n!}g^{(n)}(x)[u_n,v_m]z^{-m-1}\nonumber\\
= \ &\sum_{n\ge 1}(-1)^n\frac{1}{(n-1)!}g^{(n)}(x)\<u,v\>z^{n-1}\nonumber\\
=\ &- \<u\ot g'(x-z),v\>.
\end{align*}
Then
$$[\Phi(\eta'(a,x)),Y(b,z)]=-\<\eta''(a,x-z),b\>\quad \text{ for } a,b\in \h.$$
We also have $[\Phi(\eta'(a,x)),\Phi(\eta'(b,z))]=0$.
Using these we immediately obtain (\ref{new-h-h-relations}).

For $a\in \h,\  \beta\in L$, we have
\begin{align*}
[\Phi(a\ot g)(x),Y(e_{\beta},z)]= \ &\sum_{n\ge 0}(-1)^n\frac{1}{n!}g^{(n)}(x)[a_n,Y(e_{\beta},z)]\nonumber\\
= \ &\sum_{n\ge 0}(-1)^n\frac{1}{n!}g^{(n)}(x)\<\beta,a\>z^{n}Y(e_{\beta},z)\nonumber\\
=\ & \<a\ot g(x-z),\beta\>Y(e_{\beta},z).
\end{align*}
From this we get
\begin{align*}
[\Phi(\eta'(a,x)),Y(e_{\beta},z)]=\<\eta'(a,x-z),\beta\>Y(e_{\beta},z).
\end{align*}
We also have
\begin{align*}
[Y(\alpha,x),\Phi(\eta(\beta,z))]=\<\eta'(\beta,z-x),\alpha\>,
\end{align*}
which gives
\begin{align*}
[Y(\alpha,x),\exp(\Phi(\eta(\beta,z)))]=\<\eta'(\beta,z-x),\alpha\>\exp(\Phi(\eta(\beta,z))).
\end{align*}
Using these relations we obtain (\ref{new-h-e-beta-relations}).

On the other hand, for $\al,\be\in L$, we have
\begin{align*}
&Y_L^{\eta}(e_{\alpha},x)Y_L^{\eta}(e_{\be},z)\nonumber\\
=\ &Y(e_\al,x)\exp(\Phi(\eta(\alpha,x)))Y(e_\be,z)\exp(\Phi(\eta(\be,z)))\\
=\ &Y(e_\al,x)Y\(\exp(\Phi(\eta(\alpha,x-z)))e_\be,z\)\exp(\Phi(\eta(\alpha,x)))\exp(\Phi(\eta(\be,z)))\\
=\  &e^{\<\beta,\eta(\al,x-z)\>} Y(e_\al,x)Y(e_\be,z)\exp(\Phi(\eta(\alpha,x)))\exp(\Phi(\eta(\be,z))).
\end{align*}
Symmetrically, we get
\begin{align*}
&Y_L^{\eta}(e_{\be},z)Y_L^{\eta}(e_{\alpha},x)\\
=\ &e^{\<\al,\eta(\be,z-x)\>} Y(e_\be,z)Y(e_\al,x)\exp(\Phi(\eta(\alpha,x)))\exp(\Phi(\eta(\be,z))).
\end{align*}
From \cite[Proposition 6.5.2]{LL},  the relation
\begin{align*}
&(x-z)^{-\<\al,\be\>-1}Y(e_{\alpha},x)Y(e_{\be},z)
-(-z+x)^{-\<\al,\be\>-1}Y(e_{\be},z)Y(e_{\alpha},x)\nonumber\\
&\quad \quad \quad =\varepsilon(\al,\be)Y(e_{\al+\be},z)x^{-1}\delta\left(\frac{z}{x}\right)
\end{align*}
holds on $V_{L}$.  Then we obtain (\ref{Yeta-ealpha=ebeta}) as
\begin{align}
&(x-z)^{-\<\al,\be\>-1}e^{-\<\beta,\eta(\al,x-z)\>}Y_L^{\eta}(e_{\alpha},x)Y_L^{\eta}(e_{\be},z)\\
&\   \   -(-z+x)^{-\<\al,\be\>-1}
e^{-\<\al,\eta(\be,z-x)\>}Y_L^{\eta}(e_{\be},z)Y_L^{\eta}(e_{\alpha},x)
\nonumber\\
=\  &\varepsilon(\al,\be)Y(e_{\al+\be},z)x^{-1}\delta\left(\frac{z}{x}\right)
\exp(\Phi(\eta(\alpha,x)))\exp(\Phi(\eta(\be,z)))\nonumber\\
=\  &\varepsilon(\al,\be)Y(e_{\al+\be},z)x^{-1}\delta\left(\frac{z}{x}\right)
\exp(\Phi(\eta(\alpha+\be,z)))\nonumber\\
=\  &\varepsilon(\al,\be)Y_L^\eta(e_{\al+\be},z)x^{-1}\delta\left(\frac{z}{x}\right).\nonumber
\end{align}

Note that
 $$Y_L^{\eta}(\al,x)^{\pm}=\alpha(x)^{\pm}+\Phi(\eta'(\al,x)^{\pm}),$$
$$\frac{d}{dx}Y(e_\al,x)=\al(x)^{+}Y(e_\al,x)+Y(e_\al,x)\al(x)^{-}.$$
For $a\in \h,\ f(t)\in \C((t)),\ \be\in L$, we have
$$[\Phi(a,f)(x),Y(e_\be,z)]=Y\(\Phi(a,f)(x-z)e_\be,z\)=\<a,\be\>f(x-z)Y(e_\be,z).$$
If $f(t)\in \C[[t]]$, we see that both $\Phi(a,f)(x)Y(e_\be,x)$ and $Y(e_\be,x)\Phi(a,f)(x)$
exist in $\Hom(V_L,V_L((x)))$ with
$$[\Phi(a,f)(x),Y(e_\be,x)]=f(0)\<a,\be\>Y(e_\be,x).$$
As $\eta(\al,x)^{+}\in \h\ot \C[[x,\hbar]]$, we see that
both $\Phi(\eta'(\al,x)^{+})Y(e_\be,x)$ and $Y(e_\be,x)\Phi(\eta'(\al,x)^{+})$
exist in $\Hom(V_L,V_L((x))[[\hbar]])$.  Then we have
$$[\Phi(\eta'(\al,x)^{+}),Y(e_\al,x)]=\<\eta'(\al,0)^{+},\al\>Y_L^{\eta}(e_{\alpha},x).$$
Noticing that $[\al(x)^{-},\Phi(\eta(\al,z))]=0$. Thus we obtain
 \begin{align}
& \frac{d}{dx}Y_{L}^{\eta}(e_{\alpha},x)\nonumber\\
=\ & \al(x)^{+}Y_L^{\eta}(e_{\alpha},x)+Y_L^{\eta}(e_{\alpha},x)\al(x)^{-}+Y(e_{\al},x)\Phi(\eta'(\al,x))\exp(\Phi(\eta(\al,x)))\nonumber\\
=\ &\( \al(x)^{+}+\Phi(\eta'(\al,x)^{+})\)Y_L^{\eta}(e_{\alpha},x)
+[Y(e_{\al},x),\Phi(\eta'(\al,x)^{+})]\exp(\Phi(\eta(\al,x))))\nonumber\\
&\   +Y_L^{\eta}(e_{\alpha},x)\(\al(x)^{-}+\Phi(\eta'(\al,x)^{-})\)\nonumber\\
=\ &Y_L^{\eta}(\al,x)^{+}Y_L^{\eta}(e_{\alpha},x) +Y_L^{\eta}(e_{\alpha},x)Y_L^{\eta}(\al,x)^{-}
\nonumber\\
&\  +[Y(e_{\al},x),\Phi(\eta'(\al,x)^{+})]\exp(\Phi(\eta(\al,x))))\nonumber\\
=\ &Y_L^{\eta}(\al,x)^{+}Y_L^{\eta}(e_{\alpha},x) +Y_L^{\eta}(e_{\alpha},x)Y_L^{\eta}(\al,x)^{-}
+\<\eta'(\al,0)^{+},\al\>Y_L^{\eta}(e_{\alpha},x),\nonumber
 \end{align}
proving  (\ref{VL-eta-d-dx}). Now, the proof is complete.
\end{proof}

As a consequence of Proposition \ref{lem:commutator-V-L-h-eta} we have:

\begin{coro}\label{trivial-deformation}
Let $\eta(\cdot,x):\h\to\h\ot x\C[[x,\hbar]]$ be a linear map such that
\begin{align}\label{skew-eta-def}
  \<\eta(\al,x),\be\>=\<\al,\eta(\be,-x)\>\quad \te{for }\al,\be\in\h.
\end{align}
Then there exists an $\hbar$-adic nonlocal vertex algebra isomorphism from $V_{L}[[\hbar]]$ to $V_L[[\hbar]]^{\eta}$,
which extends the identity map on $\h+\C_{\varepsilon}[L]$ uniquely.
\end{coro}

\begin{proof} From (\ref{skew-eta-def}) we have
\begin{align*}
  \<\eta^{(r)}(\al,x),\be\> -(-1)^{r} \<\al,\eta^{(r)}(\be,-x)\>=0
\end{align*}
for $r\in \N,\ \al,\be\in\h$, where $\eta^{(r)}(\al,x)$ denotes the $r$-th derivative.
Note that this implies $\<\eta'(\al,0),\al\>=0$.
From Proposition \ref{lem:commutator-V-L-h-eta}, we get
\begin{align*}
&[Y_{L}^\eta(\al,x_1),Y_{L}^\eta(\be,x_2)]
=\<\al,\be\>\partial_{x_2}x_1\inverse\delta\(\frac{x_2}{x_1}\),\\
&[Y_{L}^\eta(\al,x_1),Y_{L}^\eta(e_\be,x_2)]
=\<\be,\al\>Y_{L}^\eta(e_\be,x_2)x_1\inverse\delta\(\frac{x_2}{x_1}\),\nonumber\\
&(x_1-x_2)^{-\<\al,\be\>-1}
Y_{L}^\eta(e_\al,x_1)Y_{L}^\eta(e_\be,x_2)
- (-x_2+x_1)^{-\<\al,\be\>-1} Y_{L}^\eta(e_\be,x_2)Y_{L}^\eta(e_\al,x_1)\nonumber\\
=&\  \epsilon(\al,\be)Y_{L}^\eta(e_{\al+\be},x_2)
    x_1\inverse\delta\(\frac{x_2}{x_1}\),\nonumber \\
    & \frac{d}{dx}Y_{L}^{\eta}(e_{\alpha},x)
=Y_L^{\eta}(\al,x)^{+}Y_L^{\eta}(e_{\alpha},x) +Y_L^{\eta}(e_{\alpha},x)Y_L^{\eta}(\al,x)^{-}
\end{align*}
for $\al,\beta\in L$.
By Proposition \ref{prop:AQ-mod-hom-VA-hom}, there exists a vertex algebra morphism
$\psi_0$ from $V_{L}$ to $V_L[[\hbar]]^\eta$ with
$\psi_0(u)=u$ for $u\in\h+\C_{\epsilon}[L]$, which is canonically extended to
an $\hbar$-adic nonlocal vertex algebra morphism
$\varphi: V_L[[\hbar]]\rightarrow V_L[[\hbar]]^\eta$.
 We see that the derived $\C$-linear map from
 $V_L[[\hbar]]/\hbar V_L[[\hbar]]$ to $V_L[[\hbar]]^{\eta}/\hbar V_L[[\hbar]]^{\eta}$ is a one-to-one vertex algebra morphism
 (noticing that $V_L$ is a simple vertex algebra). Thus $\psi$ is also one-to-one.
 On the other hand, as $\h+\C_{\epsilon}[L]$ generates $V_L[[\hbar]]^{\eta}$ as an $\hbar$-adic nonlocal vertex algebra,
 it follows that $\varphi$ is onto. Therefore, $\psi$ is an isomorphism.
\end{proof}

The following is another consequence of Proposition \ref{lem:commutator-V-L-h-eta}:

\begin{coro}\label{prop:classical-limit-qva}
Let $\eta(\cdot,x)\in\Hom(\h, \h\ot \C((x))[[\hbar]])^0$ and set $\eta_0(\al,x)=\eta(\al,x)|_{\hbar=0}$.
Then $V_{L}[[\hbar]]^\eta/\hbar V_{L}[[\hbar]]^\eta\cong V_L^{\eta_0}$ as nonlocal vertex algebras (over $\C$).
Furthermore, if
\begin{align}
\<\eta_0(\al,x),\be\>=\<\al,\eta_0(\be,-x)\>\quad \te{ for }\al,\be\in \h,\label{zero-condition2}
\end{align}
then $V_{L}[[\hbar]]^\eta/\hbar V_{L}[[\hbar]]^\eta\cong V_L^{\eta_0}\cong V_L$.
\end{coro}

\begin{proof} For $\alpha\in L$, we have
\begin{align*}
&Y_{L}^{\eta}(e_\al,x)\equiv Y(e_\al,x)\exp (\Phi(\eta_0(\al,x)))\  \mod \hbar V_{L}[[\hbar]],\\
&Y_{L}^{\eta}(\al,x)\equiv Y(\al,x)+\Phi(\eta'_0(\al,x))\  \mod \hbar V_{L}[[\hbar]].
\end{align*}
It follows that  $V_L[[\hbar]]^\eta/\hbar V_L[[\hbar]]^\eta\cong V_L^{\eta_0}$ as nonlocal vertex algebras over $\C$.
For the furthermore assertion, from (the proof of) Corollary \ref{trivial-deformation}  we see that $V_L^{\eta_0}\cong V_L$.
Consequently, we have $V_L[[\hbar]]^\eta/\hbar V_L[[\hbar]]^\eta\cong V_L$ as nonlocal vertex algebras.
\end{proof}

Furthermore, we have:

\begin{prop}\label{prop:eta-iso-eta+tau}
Let $\tau:\h\to\h\ot x\C[[x,\hbar]]$ be a linear map such that
\begin{align}
  \<\tau(\al,x),\beta\>=\<\al,\tau(\beta,-x)\>\quad \te{for }\al,\beta\in\h,
\end{align}
and let $\varphi_{\tau}$ be the $\hbar$-adic nonlocal vertex algebra isomorphism from $V_L[[\hbar]]$ to $V_L[[\hbar]]^{\tau}$
obtained in Corollary \ref{trivial-deformation}.
Then $\varphi_{\tau}$ is an $\hbar$-adic nonlocal vertex algebra isomorphism from  $V_L[[\hbar]]^\eta$ to $V_L[[\hbar]]^{\eta+\tau}$
 for any $\eta(\cdot,x)\in {\rm Hom}(\h, \h\otimes \C((x))[[\hbar]])^{0}$.
 \end{prop}

\begin{proof} Recall that $\varphi_\tau|_{\h+ \C_\varepsilon[L]}=1$.
From Theorem \ref{thm:qlva}, $(V_L[[\hbar]]^\tau,\rho)$ is a right $B_L[[\hbar]]$-comodule nonlocal vertex algebra,
where $\rho$ was defined in Proposition \ref{lem:qva-rho-def}. We have
\begin{align*}
  \rho\(\varphi_\tau(u)\)=\rho(u)=(\varphi_\tau\ot 1)\rho(u)\quad \te{for }u\in \h+ \C_\varepsilon[L].
\end{align*}
Since both $V_L[[\hbar]]$ and $V_L[[\hbar]]^\tau$ are generated by $\h+\C_\varepsilon[L]$,
we conclude that
\begin{align*}
  \rho\circ\varphi_\tau=(\varphi_\tau\ot 1)\circ \rho.
\end{align*}
Noticing that $Y_M^\tau$ and $Y_M^{\eta}$ commute,
we obtain from \cite[Proposition 3.4]{JKLT-Defom-va} that $Y_M^\eta\in \mathfrak L_{B_L[[\hbar]]}^\rho(V_L[[\hbar]]^\tau)$.
From the definition of $Y_M^\eta$, we have
\begin{align*}
  \varphi_\tau(Y_M^\eta(a,x)u)=Y_M^\eta(a,x)\varphi_\tau(u)
\end{align*}
for $a\in\h+\C[L]\subset B_L,\  u\in\h+\C_\varepsilon[L]\subset V_L$.
Since $B_L[[\hbar]]$ is generated by $\h+\C[L]$ and since $V_L[[\hbar]]$ is generated by $\h+\C_\varepsilon[L]$, we get
\begin{align*}
  \varphi_\tau(Y_M^\eta(a,x)u)=Y_M^\eta(a,x)\varphi_\tau(u)\quad\te{for }a\in B_L[[\hbar]],\,u\in V_L[[\hbar]].
\end{align*}
Then it follows that $\varphi_\tau$ is an $\hbar$-adic nonlocal vertex algebra isomorphism
from $V_L[[\hbar]]^\eta=\mathfrak D_{Y_M^\eta}^\rho(V_L[[\hbar]])$ to
$\mathfrak D_{Y_M^\eta}^\rho(V_L[[\hbar]]^\tau)$, where by Proposition \ref{prop-comp-sum} we have
\begin{align*}
  \mathfrak D_{Y_M^\eta}^\rho(V_L[[\hbar]]^\tau)
  =\mathfrak D_{Y_M^\eta}^\rho\(\mathfrak D_{Y_M^\tau}^\rho(V_L[[\hbar]])\)
   =\mathfrak D_{Y_M^{\eta+\tau}}^\rho(V_L[[\hbar]])=V_L[[\hbar]]^{\eta+\tau}.
\end{align*}
Consequently, $\varphi_{\tau}$ is also an $\hbar$-adic nonlocal vertex algebra isomorphism from
$V_L[[\hbar]]^\eta$ to $V_L[[\hbar]]^{\eta+\tau}$.
\end{proof}

The following are some structural results on the $\hbar$-adic quantum vertex algebras $V_L[[\hbar]]^\eta$:

\begin{prop}\label{VL-eta-S(X)-general}
Let $\eta(\cdot,x)\in {\rm Hom}(\h, \h\otimes \C((x))[[\hbar]])^{0}$.
Then $V_L[[\hbar]]^{\eta}$ is non-degenerate.
Furthermore, the following relations hold for $\al,\be\in L$:
\begin{eqnarray*}
&&\mathcal{S}(x)(\be\otimes \al)=\be\ot \al\ot 1+{\bf 1}\ot {\bf 1}\ot (\<\eta''(\be,x),\al\>-\<\eta''(\al,-x),\be\>),\\
&&\mathcal{S}(x)(e_{\be}\ot \al)=e_\be\ot \al\ot 1+e_\be \ot \vac \ot (\<\eta'(\be,x),\al\>+\<\eta'(\al,-x),\be\>),\\
&&\mathcal{S}(x)(e_{\beta}\otimes e_{\al})=(e_{\be}\otimes e_{\al})\ot e^{\<\beta,\eta(\al,-x)\>-\<\al,\eta(\be,x)\>},
\end{eqnarray*}
and
\begin{align}
&Y_L^{\eta}(\al,x)^{-}\beta=\(\<\al,\be\>x^{-2}-\<\eta''(\al,x),\be\>^{-}\)\vac,\label{eta-al-be-sing}\\
&Y_L^{\eta}(\al,x)^{-}e_{\beta}=\(\<\al,\be\>x^{-1}+\<\eta'(\al,x),\be\>^{-}\)e_{\be},\label{eta-al-ebe-sing}\\
&Y_L^{\eta}(e_\al,x)e_\be =\varepsilon(\al,\be)x^{\<\al,\be\>}e^{\<\be,\eta(\al,x)\>}E^{-}(-\al,x)e_{\al+\be}.\label{Y-eta-e-alha-e-be}
\end{align}
\end{prop}

\begin{proof} From Corollary \ref{prop:classical-limit-qva},
$V_L[[\hbar]]^{\eta}/\hbar V_L[[\hbar]]^{\eta}$ is isomorphic to $V_L^{\eta_0}$.
It was proved in \cite{JKLT-Defom-va} (Theorem 5.9) that $V_{L}^{\eta_0}$ as a module for itself
 is irreducible.  It is clear that $V_L^{\eta_0}$ is of countable dimension over $\C$.
 Then $V_{L}^{\eta_0}$ is non-degenerate by \cite{Li-qva2}.
Thus $V_L[[\hbar]]^{\eta}$ is non-degenerate,
 and hence the quantum Yang-Baxter operator $\mathcal{S}(x)$ is uniquely determined.
Using Proposition \ref{lem:commutator-V-L-h-eta} and Lemma \ref{tech-extra}, we immediately obtain all the relations.
\end{proof}

Let $G$ be an automorphism group of $V_{L}$  such that $G(\h)=\h$.
This implies that $G$ preserves the bilinear form on $\h$ and hence
preserves the standard conformal vector $\omega$ of $V_L$. Then
 $gL(0)=L(0)g$ on $V_L$ for $g\in G$. For any linear character $\chi$ of $G$,
$V_L$ is a $(G,\chi)$-module vertex algebra with $R(g)=\chi(g)^{L(0)}g$ for $g\in G$.
The following is a straightforward $\hbar$-adic analogue of \cite[Proposition 5.10]{JKLT-Defom-va}:

\begin{prop}\label{VL-eta-G-algebra}
Assume that $G$ is an automorphism group of $V_{L}$  such that $G(\h)=\h$ and $\chi$ is a linear character of $G$.
Let $\eta(\cdot,x)\in\Hom(\h,\h\ot \C((x))[[\hbar]])^0$ be such that
\begin{eqnarray}\label{eta-G-equivarance}
(\mu\ot 1) \eta(a,x)=\eta(\mu(a),\chi(\mu)x)\quad \text{ for }\mu\in G, \ a\in\h.
\end{eqnarray}
Define $R$ as above. Then $V_L[[\hbar]]^\eta$ is a $(G,\chi)$-module $\hbar$-adic quantum vertex algebra.
\end{prop}

\subsection{Equivariant $\phi$-coordinated quasi $V_L[[\hbar]]^{\eta}$-modules}\label{sec:phi-lattice-mod}

We here use equivariant $\phi$-coordinated quasi $V_L$-modules
to construct equivariant $\phi$-coordinated quasi modules
for the $\hbar$-adic quantum vertex algebra $V_L[[\hbar]]^{\eta}$.

We assume that $G:=\<\mu\>$ is an order $N$ cyclic automorphism group of $V_L$ with $\mu(\h)=\h$.
Set $\chi=1$ (the trivial character), and  assume
$$\chi_{\phi}(\mu^k)=\xi^k\quad \text{ for }k\in \Z_N,$$
$$\phi(x,z)=xe^z; \  \ p(x)=x.$$
Then $V_L$ is a $G$-module vertex algebra and
$V_{L}[[\hbar]]$ is a $G$-module $\hbar$-adic vertex algebra.
Let $(W,Y_W)$ be a $(G,\chi_{\phi})$-equivariant $\phi$-coordinated quasi $V_{L}$-module.

\begin{de}
For $a\in \h,\ g(x)\in \C(x)[[\hbar]]$, define
\begin{align}
&\wh{\Phi}_{V}^{\phi}(a,g)(x_1,x_2)=\te{Res}_{z}Y(a,z)e^{zx_2\partial_{x_2}}g(x_1/x_2)\  \  \  \text{on }V_L[[\hbar]],\\
&\Phi_{W}(a,g)(x)=\te{Res}_{x_2}x_2^{-1}Y_W(a,x_2)\iota_{x,x_2}g(x/x_2)\  \  \  \text{on }W[[\hbar]].
\end{align}
Then define $\wh{\Phi}_{V}^{\phi}(A)(x_1,x_2)$ and $\Phi_{W}(A)(x)$ for $A\in  \h\ot \C(x)[[\hbar]]$ by linearity.
\end{de}

As the main result of this section, we have:

\begin{thm}\label{thm:phi-mod}
Let $f(\cdot,x):\  \h\to \h\ot \hbar\C(x)[[\hbar]]$ be a linear map such that
\begin{align}
  &(\mu \ot 1) f(a,x)=f(\mu (a),x)\quad \text{ for }a\in \h.\label{eq:cond-phi-mod-1}
\end{align}
Define a linear map $\eta^f(\cdot,z):\  \h\to\h\ot  \hbar\C((z))[[\hbar]]$ by
\begin{align}
  \eta^f(a,z)=\sum_{k\in\Z_N}(\mu^{k}\ot 1)f(a,\xi^{k}e^z)=\sum_{k\in\Z_N}f(\mu^{k}a,\xi^{k}e^z)
  \quad \text{ for }a\in\h.
\end{align}
Let $\eta_{(0)}(\cdot,x):\ \h \to  \h\ot x\C[[x,\hbar]]$ be a linear map such that
\begin{align}\label{eta0-general}
  \<\eta_{(0)}(\al,x),\be\>=\<\al,\eta_{(0)}(\be,-x)\>\quad \text{  for }\al,\beta\in\h.
\end{align}
Set $\eta=\eta^f+\eta_{(0)}:\  \h\to\h\ot \C((x))[[\hbar]]$.
Then $\eta(\cdot,x)\in{\rm Hom}(\h, \h\otimes \C((x))[[\hbar]])^{0}$ and
we have an $\hbar$-adic quantum vertex algebra $V_L[[\hbar]]^\eta$
on which $G$ acts as an automorphism group.
On the other hand, let $(W,Y_W)$ be any $(G,\chi_\phi)$-equivariant $\phi$-coordinated quasi
$V_L$-module. Then we have a $(G,\chi_\phi)$-equivariant $\phi$-coordinated quasi
$V_L[[\hbar]]^\eta$-module $W[[\hbar]]^f$, where $W[[\hbar]]^f=W[[\hbar]]$ as a $\C[[\hbar]]$-module and
the vertex operator map $Y_W^f(\cdot,x)$ is uniquely determined by
\begin{align}
  &Y_W^f(a,x)=Y_W(a,x)+\Phi_{W}(f'(a,x))\quad\text{ for }a\in \h,\label{YTf-h-general}\\
  &Y_W^f(e_\al,x)=Y_W(e_\al,x)\exp\(\Phi_{W}(f(\al,x))\)\quad\text{ for }\al\in L.\label{YTf-e-alpha-general}
\end{align}
\end{thm}

\begin{proof} The first assertion is clear. From the definition of $\eta^f(\cdot,x)$, it follows that
\begin{align}
(\mu \otimes 1)\eta^f(a,x)=\eta^{f}(\mu a,x)\quad \te{ for }a\in \h,
\end{align}
and we have $\eta^f(a,x)\in \h\ot  \hbar\C((x))[[\hbar]]$.   By Proposition \ref{VL-eta-G-algebra},
$V_L[[\hbar]]^{\eta^f}$ is a $G$-module $\hbar$-adic quantum vertex algebra
 with the same action of $G$ on $V_L[[\hbar]]$, i.e.,
 an  $\hbar$-adic quantum vertex algebra on which $G$ acts as an automorphism group.
 On the other hand,  given $\eta_{(0)}$ with the condition (\ref{eta0-general}), by Corollary \ref{trivial-deformation}
we have an $\hbar$-adic quantum vertex algebra isomorphism
$$\Psi_{\eta_{(0)}}:\  V_L[[\hbar]]\rightarrow V_L[[\hbar]]^{\eta_{(0)}}$$
with $\Psi_{\eta_{(0)}}=1$ on  $\h+\C_{\varepsilon}[L]$.
Furthermore, by Proposition \ref{prop:eta-iso-eta+tau} $\Psi_{\eta_{(0)}}$ is
also an $\hbar$-adic quantum vertex algebra isomorphism
$$\Psi_{\eta_{(0)}}:\  V_L[[\hbar]]^{\eta^f}\rightarrow V_L[[\hbar]]^{\eta}.$$
Consequently, $G$ acts on $V_L[[\hbar]]^{\eta}$ as an automorphism group
whose action is uniquely determined by the action of $G$ on $\h+\C_{\varepsilon}[L]$ as a subspace of $V_L[[\hbar]]^{\eta}$.

For the second part on the $(G,\chi_{\phi})$-equivariant $\phi$-coordinated quasi module structure,
we first show that there is a $(G,\chi_{\phi})$-equivariant $\phi$-coordinated quasi $V_L[[\hbar]]^{\eta^f}$-module
structure on $W[[\hbar]]$ such that (\ref{YTf-h-general}) and (\ref{YTf-e-alpha-general}) hold.
Note that for $\alpha\in L\subset \h$, as $L(0)\alpha=\alpha$ and $\chi(\mu)=1$, we have
\begin{align*}
&\Phi_{W}(f(\mu \alpha,x))=\Res_z z^{-1}Y_W(f(\mu \alpha,x/z),z)
=\Res_z z^{-1}Y_W((\mu\ot 1)f(\alpha,x/z),z)\\
=&\ \Res_z z^{-1}Y_W(f(\alpha,x/z),\xi^{-1}z)=\Res_z z^{-1}Y_W(f(\alpha,\xi^{-1}x/z),z)
=\Phi_{W}(f(\alpha,\xi^{-1}x)).
\end{align*}
As $f(\alpha,x)\in \h\ot \hbar\C(x)[[\hbar]]$, $\exp \Phi_{W}(f(\alpha,x))$ is well defined.
Furthermore, we have
$$\exp \Phi_{W}(f(0,x))=1,$$
$$\exp \Phi_{W}(f(\alpha,x))\exp \Phi_{W}(f(\beta,x))=\exp \Phi_{W}(f(\alpha+\beta,x))
\quad \text{ for }\alpha,\beta\in L. $$
Let $K$ be the subalgebra of $\Hom (W[[\hbar]],W[[\hbar]]\wh\ot \C((x))[[\hbar]])$, generated by
$$\(x\frac{d}{dx}\)^n \exp \Phi_{W}(f(\alpha,x))\quad \text{  for }n\in \N,\ \alpha\in L.$$
It follows that there exists a differential algebra morphism $\psi: B_L\rightarrow K$ such that
$$\psi(e^\alpha)=\exp \Phi_{W}(f(\alpha,x))\quad \text{ for }\alpha\in L.$$
Consequently, there is a $(G,\chi_{\phi})$-equivariant $\phi$-coordinated quasi $B_L[[\hbar]]$-module structure
$Y_{W}^{B_L}(\cdot,x)$ on $W[[\hbar]]$ with
\begin{align}
Y_{W}^{B_L}(e^\alpha,x)=\exp \Phi_{W}(f(\alpha,x))\quad \text{ for }\alpha\in L.
\end{align}

We next show that the $\phi$-coordinated quasi module structures on $W[[\hbar]]$
for $V_L[[\hbar]]$ and $B_L[[\hbar]]$ are compatible.
For $u,v\in V_L[[\hbar]]$, from Proposition \ref{prop:tech-reverse-calculations-h}  with $G^0=G$,  we get
\begin{align*}
  [Y_W(u,x_1),Y_W(v,x_2)]
  =\sum_{k\in \Z_N}\te{Res}_zY_W(Y(\mu^k(u),z)v,x_2)
  e^{zx_2\partial_{x_2}}\delta\(\xi^{-k}\frac{x_2}{x_1}\).
\end{align*}
Then, for $a\in \h,\ v\in V_L[[\hbar]]$  we have
\begin{align*}\label{prop:fk-to-eta}
  [\Phi_{W}(f(a,x_1)),Y_W(v,x_2)]
  =\sum_{k\in \Z_N}\Res_ze^{zx_2\partial_{x_2}}Y_W( Y(f(\mu^ka,\xi^kx_1/x_2),z)v,x_2).
\end{align*}
Noticing that
\begin{align}
\wh{\eta^f}(a,x_1,x_2)=(1\ot \pi_{\phi}^{-1})\eta^f(a,x)=\sum_{k\in\Z_N}f(\mu^{k}a,\xi^{k}x_1/x_2),
\end{align}
we have
\begin{align*}
\wh \Phi_V^{\phi}\(\wh{\eta^f}(a,x_1,x_2)\)=\ &\Res_z e^{zx_2\partial_{x_2}}Y\(\wh{\eta^f}(a,x_1,x_2),z\)\nonumber\\
=\ &\Res_z \sum_{k\in\Z_N}e^{zx_2\partial_{x_2}}Y(f(\mu^{k}a,\xi^{k}x_1/x_2),z).
\end{align*}
Thus
\begin{align}\label{prop:fk-to-eta}
  [\Phi_{W}(f(a,x_1)),Y_W(v,x_2)]=Y_W\(\wh \Phi_V^{\phi}(\wh\eta^f(a,x_1,x_2))v,x_2\)
\end{align}
for $a\in\h,\  v\in V_L$.  From this, the compatibility follows as $\C[L]$ generates $B_L$ as a vertex algebra.
Then by Theorem \ref{prop:deform-phi-quasi-mod},  there exists a $(G,\chi_\phi)$-equivariant $\phi$-coordinated quasi
$V_L[[\hbar]]^{\eta^f}$-module structure $Y_W^f(\cdot,x)$ on $W[[\hbar]]$, which is uniquely determined by
(\ref{YTf-h-general}) and (\ref{YTf-e-alpha-general}).

Finally, using the isomorphism $\Psi_{\eta_{(0)}}$ we obtain a $(G,\chi_\phi)$-equivariant $\phi$-coordinated quasi
$V_L[[\hbar]]^{\eta}$-module structure on $W[[\hbar]]$, where the vertex operator map is also uniquely determined by
(\ref{YTf-h-general}) and (\ref{YTf-e-alpha-general}) as $\Psi_{\eta_{(0)}}=1$ on  $\h+\C_{\varepsilon}[L]$.
\end{proof}

\begin{rem}
{\em Suppose that $\mu$ is an automorphism of $V_L$, which is lifted from
an isometry $\mu_0$ of lattice $L$ with a period $N$ such that
\begin{align}\label{mu-condition}
  \sum_{k\in\Z_N}\<\mu_0^k(\al),\al\>\in 2\Z\quad\te{for }\al\in L.
\end{align}
It was proved in \cite{JKLT-G-phi-mod} that on Lepowsky's twisted $V_L$-module $V_T$ (see \cite{Lep}),
there exists a canonical $(G,\chi_\phi)$-equivariant $\phi$-coordinated quasi
$V_L$-module structure. }
\end{rem}

\section{Twisted quantum affine algebras and $\phi$-coordinated $V_L[[\hbar]]^{\eta}$-modules}

In this section, as one of the main results of this paper, we give a canonical connection between twisted quantum affine algebras
and $\hbar$-adic quantum vertex algebras $V_L[[\hbar]]^{\eta}$ with particularly defined linear maps $\eta$.

\subsection{Specialized $\hbar$-adic quantum vertex algebras $V_{L}[[\hbar]]^{\eta}$}

Let $A=(a_{i,j})_{i,j\in I}$ be a Cartan matrix of type $A$, $D$, or $E$. In particular,
$A$ is a symmetric matrix with
\begin{align*}
  a_{i,i}=2\  \  \te{for }i\in I\quad\te{and}\quad a_{i,j}\in \{0, -1\}\ \te{for }i,j\in I\ \te{with } i\ne j.
\end{align*}
Let $\mu\in\Aut (A)$, i.e., $\mu$ is a permutation on $I$ such that
\begin{align}
  a_{\mu(i),\mu(j)}=a_{i,j}\quad\te{for }i,j\in I.
\end{align}
Set
\begin{align}
N=o(\mu)\   \ \text{ and }\   \xi=\exp\(2\pi i/N\).
\end{align}

Let $\g:=\g(A)$ be the corresponding finite-dimensional simple Lie algebra with a set of simple roots $\set{\al_i}{i\in I}$
and let $\<\cdot,\cdot\>$ be the Killing form on $\g$, normalized so that the squared length of all roots equals $2$.
Set
\begin{align}
L=\oplus_{i\in I}\Z\al_i,
\end{align}
 the root lattice of $\g$ with
 \begin{align}
  \<\al_i,\al_j\>=a_{i,j}\quad\te{for }i,j\in I.
\end{align}
Extend $\mu$ to an isometry on $L$, which is still denoted by $\mu$,  such that
\begin{align}
  \mu(\al_i)=\al_{\mu(i)}\quad\te{for }i\in I.
\end{align}

Let  $\varepsilon(\cdot,\cdot)$ be a {\em bilinear} character of $L$ with
\begin{align}\label{2cocycle-condition}
\varepsilon(\al_i,\al_i)=1\quad \text{ for }i\in I.
\end{align}
In particular,  $\varepsilon(\cdot,\cdot)$ is a $2$-cocycle satisfying the normalizing conditions
that are assumed in the construction of $V_L$.
From Section \ref{sec:phi-lattice-mod}, $\mu$ gives rise to a vertex operator algebra automorphism of $V_L$,
which is also denoted by $\mu$.
Set
\begin{align}
G=\<\mu\>\subset \te{Aut}(V_L).
\end{align}

As in Section 5, we assume
$$\phi(x,z)=e^{z x\partial_x}x=xe^z,$$
$\chi=1$, the trivial character of $G$, and let $\chi_\phi:G\to\C^\times$ be defined by
\begin{align*}
  \chi_\phi(\mu^k)=\xi^k\quad\te{for }k\in\Z.
\end{align*}

Recall  the field embedding
$$\iota_{x=e^z}:\   \C(x)\to\C((z)),$$
 which is the natural extension of the ring embedding of $\C[x]$ into $\C((z))$ with $x\mapsto e^{z}$.
The following is a straightforward fact:

\begin{lem}\label{xdx-dx-connection}
Let $f(z)\in \C[[z]],\ Q(x)\in \C(x)$. Then
\begin{align}
\left(f(\hbar x\partial_x)Q(x)\right)|_{x=e^z}=f(\hbar \partial_z)Q(e^z),
\end{align}
where it is understood that $Q(e^z)=\iota_{x=e^z}(Q(x))$,  an element of $\C((z))$.
\end{lem}

 For any integer $m$, set
 \begin{align}
 [m]_z=\frac{z^m-z^{-m}}{z-z^{-1}}\in \C[z,z^{-1}],
 \end{align}
with $z$ a formal variable as usual.
On the other hand, for any integer $n$ and for any invertible element $v$ of $\C[[\hbar]]$ with $v^2\ne1$, set
\begin{align}
  [n]_v=\frac{v^n-v^{-n}}{v-v\inverse}\in \Z[v,v^{-1}].
\end{align}
The following are some simple facts:
\begin{align}
[n]_v=n\ \  \text{ for }n\in \{-1,0,1\},\quad
[-n]_v=-[n]_v,\quad  [n]_{v^{-1}}=[n]_v.
\end{align}
Furthermore,  for $n\in \Z,\  s\in \N,\ 0\le k\le s$, define
\begin{align*}
  [s]_v!=[s]_v[s-1]_v\cdots [1]_v,\quad
  \binom{s}{k}_v=\frac{[s]_v!}{[s-k]_v![k]_v!}.
\end{align*}

As a convention, we write
\begin{align}
[m]_{e^x}=\iota_{z=e^x}([m]_z)\in\C[[x^2]].
\end{align}
Define
\begin{align}
  A(x)=([a_{i,j}]_{e^x})_{i,j\in I},
  \end{align}
 a square matrix with entries in $\C[[x^2]]$.
And we define the square matrix $C(x)$ by
\begin{align}
e^xA(x)A^{-1}-{\rm I}=xC(x);\quad  xC(x)A=e^xA(x)-A.
\end{align}
Write $C(x)=(c_{i,j}(x))_{i,j\in I}$.
Notice that
\begin{align}\label{C-property}
c_{i,j}(x)\in \C[[x]], \quad c_{i,j}(0)=\delta_{i,j},\quad c_{\mu(i),j}(x)=c_{i,\mu^{-1}(j)}(x)  \quad \text{ for }i,j\in I.
\end{align}

\begin{de}\label{def-special-f}
Let $f(\cdot,x):\  \h\to\h\ot\C(x)[[\hbar]]$ be the  linear map defined  by
\begin{align}
  f(\al_i,x)=\frac{1}{2}\hbar \sum_{j\in I}\al_j\ot c_{i,j}(\hbar x\partial_x)\(\frac{x+1}{x-1}\)
\end{align}
for $i\in I$. On the other hand, define a linear map $\eta^f: \ \h\rightarrow \h\ot \C((x))[[\hbar]]$ by
\begin{align}\label{def-eta-f}
\eta^f(\al_i,x)
= \frac{1}{2}\hbar \sum_{k\in \Z_N}\sum_{j\in I} \al_{\mu^k(j)}\ot  c_{i,j}(\hbar \partial_x)\(\frac{\xi^{k}e^x+1}{\xi^{k}e^x-1}\).
\end{align}
\end{de}

Notice that  $\eta^f(\alpha,x)|_{\hbar=0}=0$ and by (\ref{C-property}) we have
\begin{align}\label{eta-f-invariance}
(\mu\ot 1)\eta^f(\al,x)=\eta^f(\mu \alpha,x)\  \(=\eta^f(\mu \alpha,\chi(\mu)x)\) \quad \text{ for }\al\in \h.
\end{align}

Recall that for any $Q(x)\in x\C[[x]]$, by definition we have
$$\log (1+Q(x))=\sum_{n\ge 1}\frac{1}{n}(-1)^{n-1}Q(x)^{n}\in x\C[[x]].$$

\begin{de}\label{vartheta-def}
Introduce the following nonnegative integer power series
\begin{align}\label{eq:vartheta}
  \vartheta_0(x)&=\log\frac{e^{\half x}-e^{-\half x}}{x},\\
  \vartheta_k(x)&=\log\frac{\xi^{k}e^{\half x}-e^{-\half x}} {\xi^{k}-1}
      \quad \te{for }0\ne k\in\Z_N.
\end{align}
\end{de}

It is straightforward to show that the following relations hold for $k\in \Z_N$:
\begin{align}
 & \vartheta_k(x)\in x\C[[x]]\quad\te{and}\quad
  \vartheta_k(x)=\vartheta_{-k}(-x),\label{vartheta-property}\\
&\frac{d}{dx}\vartheta_k(x)= \frac{1}{2}\(\frac{\xi^{k}e^{x}+1}{\xi^{k}e^{x}-1}\)-\delta_{k,0}x^{-1}.
\label{der-vartheta}
\end{align}
Note that $\vartheta_k(x)$ as a power series is uniquely determined by (\ref{der-vartheta}) with $\vartheta_k(0)=0$.

\begin{de}\label{def-eta}
Define a linear map $\eta(\cdot,x):\  \h\to\h\ot\C((x))[[\hbar]]$  by
\begin{align}\label{def-eta-app}
  \eta(\al_i,x)
  =\frac{1}{2}\hbar \sum_{k\in\Z_N}\sum_{j\in I}\al_{\mu^k(j)}\ot  c_{i,j}(\hbar\partial_{x})\(\frac{\xi^{k}e^x+1}{\xi^{k}e^x-1}\)
 +\sum_{k\in\Z_N}\al_{\mu^k(i)}\ot \vartheta_k(x)
\end{align}
for $i\in I$.
\end{de}

Writing $\eta(\cdot,x)=\sum_{n\ge 0}\eta_n(\cdot,x)\hbar^n$ with $\eta_n(\cdot,x): \  \h\rightarrow \h\ot \C((x))$,
we have
\begin{align}\label{eta0-1}
\eta_0(\al_i,x)=\sum_{k\in\Z_N}\al_{\mu^k(i)}\ot \vartheta_k(x)\in \h\ot x\C[[x]]
\end{align}
for $i\in I$. Then
\begin{align}
\eta(\al,x)=  \eta^f(\al,x)+ \eta_0(\al,x)
\end{align}
for $\alpha\in \h$, where by (\ref{vartheta-property}),
\begin{align}\label{eta0-2}
\<\eta_0(\al,x),\beta\>=\<\al,\eta_0(\beta,-x)\>\quad \text{ for }\alpha,\beta\in \h.
\end{align}

By Theorem \ref{thm:phi-mod}, we immediately have:

\begin{prop}\label{prop:qlva-special}
With $\eta$ and $f$ defined as above,
we have an $\hbar$-adic quantum vertex algebra $V_L[[\hbar]]^{\eta}$ on which $G$ acts as an automorphism group,
where $V_L[[\hbar]]^{\eta}=V_L[[\hbar]]$ as a $\C[[\hbar]]$-module and
the vertex operator map $Y_L^{\eta}(\cdot,x)$ is uniquely determined by
\begin{align}
  &Y_L^{\eta}(a,x)=Y(a,x)+\Phi(\eta'(a,x))\quad \text{ for }a\in\h,\\
  &Y_L^{\eta}(e_{\al},x)=Y(e_{\al},x)\exp\(\Phi(\eta(\al,x))\)\quad \text{ for }\al\in L.
\end{align}
On the other hand, for any $(G,\chi_\phi)$-equivariant $\phi$-coordinated quasi
$V_L$-module $(W,Y_W)$, we have a $(G,\chi_\phi)$-equivariant $\phi$-coordinated quasi
$V_L[[\hbar]]^{\eta}$-module  $W[[\hbar]]^f$, where $W[[\hbar]]^f=W[[\hbar]]$ as a $\C[[\hbar]]$-module
and the vertex operator map $Y_W^f(\cdot,x)$ is uniquely determined by
\begin{align}
  &Y_W^f(a,x)=Y_W(a,x)+\Phi_{W}(f'(a,x))\quad \text{ for }a\in\h,\label{YTf-h}\\
  &Y_W^f(e_\al,x)=Y_W(e_\al,x)\exp\(\Phi_{W}(f(\al,x))\)\quad \text{ for }\al\in L.\label{YTf-e-alpha}
\end{align}
\end{prop}

In the following, we give additional structural results about $V_L[[\hbar]]^{\eta}$.
Though we normally use integer power series of formal variables,
it is very convenient to have logarithm $\log x$.
Consider the commutative and associative algebra $\C((x))[\log x]$ with $\log x$ viewed as a formal variable,
equipped with a derivation $\partial_x$ defined by
$$\partial_x \log x=x^{-1}\  \text{ and } \  \partial_x g(x)=\frac{d}{dx}g(x)\quad \text{ for }g(x)\in \C((x)).$$
Note that $\Z \log x+x\C[[x]]$ is an additive subgroup of $\C((x))[\log x]$.
We have a natural group homomorphism
\begin{align}
\exp:\  \Z \log x+x\C[[x]]\rightarrow \C((x))^{\times}
\end{align}
with $\exp (m\log x)=x^m$ for $m\in \Z$. Then extend $\exp$ to a group homomorphism
 $$\exp:\  \Z \log x+x\C[[x]]+\hbar\C((x))[[\hbar]]\rightarrow (\C((x))[[\hbar]])^{\times}.$$

Recall that $[a_{i,j}]_{e^x}\in \C[[x^2]]$ for $i,j\in I$. Note that
$$[a_{i,j}]_{e^{-x}}=[a_{i,j}]_{e^x}=[a_{j,i}]_{e^x}\ \text{ and }\  (\partial_x^{m}\log x)|_{x=-z}=(-\partial_z)^{m}\log z\ \text{ for }m\ge 1.$$

 \begin{lem}\label{eta-basic-facts}
 For $i,j\in I$, we have
\begin{align}\label{eta-key}
&\<\eta(\al_i,x),\al_j\>=\sum_{k\in \Z_N}[a_{\mu^k(i),j}]_{q^{\partial_x}}q^{\partial_x}\vartheta_k(x)
+\([a_{i,j}]_{q^{\partial_x}}q^{\partial_x}-a_{i,j}\)\log x.
\end{align}
\end{lem}

\begin{proof} Using (\ref{der-vartheta}) and the relation $xC(x)A=A(x)e^x-A$,  we get
\begin{align*}
&\<\eta(\al_i,x),\al_j\>\\
=\ & \sum_{k\in\Z_N}\sum_{r\in I}a_{\mu^k(r),j} c_{i,r}(\hbar\partial_{x}) (\hbar\partial_x)\(\vartheta_k(x) +\delta_{k,0}\log x\)
 +\sum_{k\in\Z_N}a_{\mu^k(i),j} \vartheta_k(x)      \\
 =\ & \sum_{k\in \Z_N}\([a_{i,\mu^{-k}(j)}]_{q^{\partial_x}}q^{\partial_x}-a_{i,\mu^{-k}(j)}\)(\vartheta_k(x)+\delta_{k,0}\log x)
  +\sum_{k\in\Z_N}a_{\mu^k(i),j} \vartheta_k(x)  \\
 =\ & \sum_{k\in \Z_N}[a_{\mu^k(i),j}]_{q^{\partial_x}}q^{\partial_x}\vartheta_k(x)
+\([a_{i,j}]_{q^{\partial_x}}q^{\partial_x}-a_{i,j}\)\log x,
\end{align*}
as claimed.
\end{proof}

For $i,j\in I$, set
\begin{align}\label{kappa-ij-def}
  \kappa_{i,j}(x)=\exp\(\sum_{k\in\Z_N}
  [a_{\mu^k(i),j}]_{q^{\partial_{x}}}q^{\partial_{x}}\vartheta_k(x) \).
\end{align}
 As $\vartheta_k(x)\in x\C[[x]]$, for any $m\in \Z$ we have
$$q^{m\partial_{x}}\vartheta_k(x)=e^{m\hbar \partial_{x}}\vartheta_k(x)\in \vartheta_k(x)+\hbar \C[[x,\hbar]]
\subset x\C[[x]]+\hbar \C[[x,\hbar]].$$
Then
\begin{align}\label{6.35}
[a_{\mu^k(i),j}]_{q^{\partial_{x}}}q^{\partial_{x}}\vartheta_k(x)\in  x\C[[x]]+\hbar\C[[x,\hbar]].
\end{align}
Thus $\kappa_{i,j}(x)$ exist in $\C[[x,\hbar]]$ and are invertible.
We immediately have:

\begin{lem}\label{eta-kappa-relation}
For $i,j\in I$,
\begin{align*}
e^{\<\eta(\al_i,x),\al_j\>}=\kappa_{i,j}(x)\exp \([a_{i,j}]_{q^{\partial_x}} q^{\partial_x}-a_{i,j})\log x \).
\end{align*}
\end{lem}

Using (\ref{der-vartheta}) and (\ref{eta-key}) we get
\begin{align}
&\< \eta'(\al_i,x),\al_j\>=\frac{1}{2}\sum_{k\in\Z_N}[a_{\mu^k(i),j}]_{q^{\partial_{x}}} q^{\partial_x}
 \(\frac{\xi^{k}e^x+1}{\xi^{k}e^x-1}\)-a_{i,j}x^{-1},\\
&\< \eta''(\al_i,x),\al_j\>
=-\sum_{k\in\Z_N}[a_{\mu^k(i),j}]_{q^{\partial_{x}}} q^{\partial_x}
\( \frac{\xi^{k}e^x}{(\xi^{k}e^x-1)^2}\)+a_{i,j}x^{-2}.
\end{align}
Also, from (\ref{eta-key}) we have
\begin{align}
&\< \eta'(\al_i,x),\al_j\>^{-}=\partial_x\< \eta(\al_i,x),\al_j\>^{-}  =[a_{i,j}]_{q^{\partial_{x}}} q^{\partial_x}x^{-1}-a_{i,j}x^{-1},
\label{eta-first-der-sing}\\
&\< \eta''(\al_i,x),\al_j\>^{-}=-[a_{i,j}]_{q^{\partial_{x}}} q^{\partial_x}x^{-2}+a_{i,j}x^{-2}.\label{eta-second-der-sing}
\end{align}

We shall often use the following straightforward fact:
\begin{align}\label{singular-fact}
{\rm Sing}_x (x-b\hbar)^{-1}F(x,\hbar)=(x-b\hbar)^{-1}F(b\hbar, \hbar)
 \end{align}
for any $b\in \C,\ F(x,\hbar)\in \C[[x,\hbar]]$.
Recall that $Y_L^{\eta}(v,x)^{-}$ denotes the singular part of $Y_L^{\eta}(v,x)$ for $v\in V_L[[\hbar]]^{\eta}$.
Specializing Proposition \ref{VL-eta-S(X)-general} we have:

\begin{lem}\label{lem:Y-eta-Sing}
Let $i,j\in I$. Then
\begin{align}
  &Y_L^\eta(\al_i,x)^-\al_j=[a_{i,j}]_{q^{\partial_{x}}}q^{\partial_{x}}x^{-2}\vac,
    \label{eq:Y-eta-Sing-h-h}\\
  &Y_L^\eta(\al_i,x)^-e_{\pm\al_j}=\pm e_{\pm\al_j}[a_{i,j}]_{q^{\partial_{x}}}q^{\partial_{x}}x^{-1},
   \label{eq:Y-eta-Sing-h-e}\\
  &Y_L^\eta(e_{\delta \al_i},x)^-e_{\delta' \al_j}=0\quad\te{for }\delta,\delta'\in \{\pm 1\} \
  \text{ with }\delta\delta' a_{i,j}\ge 0,\label{eq:Y-eta-Sing-e-e>=0}\\
  &Y_L^\eta(e_{\pm\al_i},x)^-e_{\pm\al_j}
  = \epsilon(\al_i,\al_j)(x+\hbar)\inv\kappa_{i,j}(-\hbar)
  A_i(-\hbar)^{\pm 1}
  e_{\pm(\al_i+\al_j)}\label{eq:Y-eta-Sing-e-e<0}
  \end{align}
  for $i,j\in I$ with $a_{i,j}=-1$, and
  \begin{align}
  &Y_L^\eta(e_{\pm\al_i},x)^-e_{\mp\al_i}=
  \frac{1}{2\hbar}\(\kappa_{i,i}(0)\inv x\inv
  -\kappa_{i,i}(-2\hbar)\inv
  A_i(-2\hbar)^{\pm 1}
  (x+2\hbar)\inv\)\vac,\label{eq:Y-eta-Sing-e-e-to-h>0}
\end{align}
where
\begin{align}
  A_i(x)=\exp \left(\sum_{n\in \Z_{+}}\frac{ \al_i(-n)}{n}x^n\right)=E^-(-\al_i,x).
\end{align}
\end{lem}

\begin{proof} Using (\ref{eta-al-be-sing}), (\ref{eta-al-ebe-sing}), (\ref{eta-second-der-sing}), (\ref{eta-first-der-sing})
we get
\begin{align*}
 & Y_L^\eta(\al_i,x)^-\al_j =  \(\<\al_i,\al_j\>x^{-2}-\<\eta''(\al_i,x),\al_j\>^{-}\){\bf 1}
  =  [a_{i,j}]_{q^{\partial_{x}}}q^{\partial_{x}}x^{-2}\vac,\\
  &Y_L^\eta(\al_i,x)^-e_{\pm\al_j}=\( \pm \<\al_i,\al_j\>x^{-1}-\<\eta'(\al_i,x),\pm \al_j\>^{-}\)e_{\pm \al_j}
= \pm e_{\pm\al_j}[a_{i,j}]_{q^{\partial_{x}}}q^{\partial_{x}}x^{-1}.
\end{align*}
Recall (\ref{Y-eta-e-alha-e-be}):
\begin{align*}
Y_L^{\eta}(e_\al,x)e_\be=\varepsilon(\al,\be)x^{\<\al,\be\>}e^{\<\be,\eta(\al,x)\>}E^{-}(-\al,x)e_{\al+\be}.\nonumber
\end{align*}
Taking $\alpha=\delta\al_i,\ \be=\delta'\al_j$ with $\delta,\delta'\in \{\pm 1\}$,
using Lemma \ref{eta-kappa-relation}, we get
\begin{align}\label{general-delta-delta'}
&Y_L^\eta(e_{\delta\al_i},x)e_{\delta'\al_j}\\
=\ & \epsilon(\al_i,\al_j)^{\delta\delta'}x^{\delta\delta'a_{i,j}}
(e^{\<\al_j,\eta(\al_i,x)\>})^{\delta\delta'}E^{-}(-\delta\al_i,x)e_{\delta \al_i+\delta' \al_j}\nonumber   \\
=\ &\epsilon(\al_i,\al_j)^{ \delta\delta'}A_i(x)^{\delta}e_{\delta\al_i+\delta'\al_j}
  \exp\(\delta\delta' a_{i,j}(\hbar\partial_{x})q^{\partial_{x}}\log x\)\kappa_{i,j}(x)^{\delta\delta'}.\nonumber
\end{align}
Noticing that  $\exp (mq^{n\partial_{x}}\log x)=(x+n\hbar)^m$ for $m,n\in \Z$,
we see that for $k\in \N$, $\exp\([k]_{q^{\partial_{x}}}q^{\partial_{x}}\log x\)$ is regular in $x$.
Then (\ref{eq:Y-eta-Sing-e-e>=0}) follows from (\ref{general-delta-delta'}).

In case $a_{i,j}=-1$, we have
$$\exp\([a_{i,j}]_{q^{\partial_{x}}}q^{\partial_{x}}\log x\)=\exp\(-q^{\partial_{x}}\log x\)=(x+\hbar)^{-1}.$$
Then using (\ref{singular-fact}) we obtain (\ref{eq:Y-eta-Sing-e-e<0}).
As for the last assertion, noticing that
$$\exp\(-[a_{i,i}]_{q^{\partial_{x}}}q^{\partial_{x}}\log x\)
=\exp\(-(q^{2\partial_x}+1)\log x\)=\frac{1}{2\hbar}(x^{-1}-(x+2\hbar)^{-1}),$$
using (\ref{singular-fact}) we get
\begin{align*}
  Y_L^\eta(e_{\pm\al_i},x)^-e_{\mp\al_i}
  =\  &\epsilon(\al_i,\al_i)\inv
  \frac{1}{2\hbar}\Sing_x\(A_i(x)^{\pm 1}\vac
  \kappa_{i,i}(x)\inv (x\inv- (x+2\hbar)\inv\)\\
  =\  &\frac{1}{2\hbar}\(\kappa_{i,i}(0)\inv x\inv-
  \kappa_{i,i}(-2\hbar)\inv A_i(-2\hbar)^{\pm 1}
  (x+2\hbar)\inv \)\vac.
\end{align*}
Note that we use the assumption $\epsilon(\al_i,\al_i)=1$ and $\epsilon(\delta\al_i,\delta'\al_j)=\delta\delta'\epsilon(\al_i,\al_j)$.
\end{proof}

Using (\ref{eta-key}), we also get
\begin{align}\label{theta-difference}
&\<\eta(\al_i,-x),\al_j\>-\<\eta(\al_j,x),\al_i\>\\
=\ &  \sum_{k\in \Z_N}\(q^{-a_{\mu^k(i),j}\partial_x}-q^{a_{\mu^k(i),j}\partial_x}\)\vartheta_{-k}(x)
+\(q^{-a_{i,j}\partial_x}-q^{a_{i,j}\partial_x}\)\log x.\nonumber
\end{align}
Note that for $k\in \Z_N,\ m\in \Z$, we have
\begin{align}\label{vartheta-diff}
\vartheta_k(x-m\hbar)-\vartheta_k(x+m\hbar)=\log \(\frac{q^m-\xi^{k}e^{x}}{1-q^m\xi^{k}e^{x}}\)
-\delta_{k,0}\log \(\frac{1-m\hbar/x}{1+m\hbar/x}\).
\end{align}
Then
\begin{align}
\<\eta(\al_i,-x),\al_j\>-\<\eta(\al_j,x),\al_i\>
=\log \prod_{k\in \Z_N}\(\frac{q^{a_{\mu^k(i),j}}-\xi^{-k}e^{x}}{1-q^{a_{\mu^k(i),j}}\xi^{-k}e^{x}}\).
\end{align}
On the other hand, (\ref{theta-difference}) together with (\ref{der-vartheta}) directly gives
\begin{align*}
&\<\eta'(\al_j,x),\al_i\>+\<\eta'(\al_i,-x),\al_j\>
= \frac{1}{2} \sum_{k\in \Z_N}\(q^{a_{\mu^k(i),j}\partial_x}-q^{-a_{\mu^k(i),j}\partial_x}\)\(\frac{\xi^{-k}e^{x}+1}{\xi^{-k}e^{x}-1}\),\nonumber\\
&\<\eta''(\al_j,x),\al_i\>-\<\eta''(\al_i,-x),\al_j\>
=\sum_{k\in \Z_N}\(q^{-a_{\mu^k(i),j}\partial_x}-q^{a_{\mu^k(i)j}\partial_x}\) \frac{\xi^{-k}e^{x}}{(\xi^{-k}e^{x}-1)^2}.
\nonumber
\end{align*}

\begin{de}\label{def-tilde-g-ij}
For $i,j\in I$, define a rational function (lying in $\C[q,q^{-1}](x)$)
\begin{align}
\tilde{g}_{i,j}(x)=\prod_{k\in \Z_N}\(\frac{q^{a_{\mu^k(i),j}}-\xi^{-k}x}{1-q^{a_{\mu^k(i),j}}\xi^{-k}x}\).
\end{align}
\end{de}

The following is an immediate consequence of Proposition \ref{VL-eta-S(X)-general}:

\begin{lem}\label{VL-eta-S(X)-special}
For the $\hbar$-adic quantum vertex algebra $V_L[[\hbar]]^{\eta}$,  for $i,j\in I$ we have
\begin{align*}
  &\mathcal{S}(x)(\al_j\ot \al_i)
  =\al_j\ot\al_i+\vac\ot\vac\ot \sum_{k\in\Z_N}
  \(q^{-a_{\mu^k(i),j}\pd{x}}-q^{a_{\mu^k(i),j}\pd{x}}\)
  \frac{\xi^{-k}e^{x}}{(\xi^{-k}e^x-1)^2},\\
    &\mathcal{S}(x)(e_{\pm\al_j}\ot\al_i)
  =e_{\pm\al_j}\ot\al_i\pm\half e_{\pm\al_j}\ot\vac\ot \sum_{k\in\Z_N}
  \(q^{a_{\mu^k(i),j}\pd{x}}-q^{-a_{\mu^k(i),j}\pd{x}}\)
    \frac{\xi^{-k}e^{x}+1}{\xi^{-k}e^{x}-1},\\
   &\mathcal{S}(x)(e_{\delta \al_j}\ot e_{\delta' \al_i})  =e_{\delta \al_j}\ot e_{\delta'\al_i}\ot \tilde{g}_{i,j}(e^{x})^{\delta\delta'},
\end{align*}
where $\delta,\delta'\in \{\pm 1\}$ and $\tilde{g}_{i,j}(e^{x})^{\delta\delta'}$ are understood as elements of $\C((x))[[\hbar]]$.
\end{lem}

The following is a key result closely related to the Serre relations:

\begin{lem}\label{lem:Y-eta-rel-5}
Let $i,j\in I$ such that $a_{i,j}=\<\al_i,\al_j\>=-1$. Then
\begin{align}
(e_{\pm\al_i})_0^{\eta}(e_{\pm\al_i})_0^{\eta} e_{\pm\al_j}=0
\end{align}
in $V_L[[\hbar]]^{\eta}$, where $Y_L^{\eta}(v,x)=\sum_{n\in \Z}v_n^{\eta}x^{-n-1}$ for $v\in V_L[[\hbar]]^{\eta}$.
\end{lem}

\begin{proof} From Lemma \ref{lem:Y-eta-Sing} we have
\begin{align}\label{ei-0-ej}
(e_{\pm\al_i})_0^{\eta}e_{\pm\al_j}=\varepsilon(\al_i,\al_j)\kappa_{i,j}(-\hbar)A_i(-\hbar)^{\pm 1}e_{\pm(\al_i+\al_j)}.
\end{align}
Let $\al,\be\in L$. As $\be(j)e_{\al}=\delta_{j,0}\<\al,\be\>e_{\al}$ for $j\in \N$, we have
$$[\beta(m),Y(e_\al,x)]=\<\al,\be\>x^mY(e_\al,x)\quad \text{ for }m\in \Z.$$
By using this, it is straightforward to show that
\begin{align*}
E^{-}(\beta,z)Y(e_{\al},x)=\(1-\frac{z}{x}\)^{\<\al,\be\>}Y(e_{\al},x)E^{-}(\be,z).
\end{align*}
 It is also straightforward to show that
\begin{align*}
\left[\Phi(\eta(\al,x)),\sum_{n\ge 1}\frac{\be(-n)}{-n}z^n\right]=\<\eta(\al,x)-\eta(\al,x-z),\be\>.
\end{align*}
Then
\begin{align}
&E^{-}(\be,z)Y_L^{\eta}(e_{\al},x)\\
=\ & Y_L^{\eta}(e_{\al},x)E^{-}(\be,z)\(1-\frac{z}{x}\)^{\<\al,\be\>}\exp (\<\eta(\al,x-z)-\eta(\al,x),\be\>).\nonumber
\end{align}
With $A_i(z)^{\pm 1}=E^{-}(\mp \al_i,z)$, using Lemma \ref{eta-kappa-relation} we have
\begin{align*}
 Y_L^\eta(e_{\pm\al_i},x)A_i(z)^{\pm 1}
 = A_i(z)^{\pm 1}Y_L^\eta(e_{\pm\al_i},x)
  \frac{\kappa_{i,i}(x-z)}{\kappa_{i,i}(x)}\frac{(x-z)(x+2\hbar-z)}{x(x+2\hbar)},
\end{align*}
which gives
\begin{align}\label{eq:Y-eta-A-i-comm}
  &Y_L^\eta(e_{\pm\al_i},x)A_i(-\hbar)^{\pm 1}\\
=\ & \frac{(x+\hbar)(x+3\hbar)}{x(x+2\hbar)} \kappa_{i,i}(x+\hbar)\kappa_{i,i}(x)^{-1}
A_i(-\hbar)^{\pm 1}Y_L^\eta(e_{\pm\al_i},x).\nonumber
\end{align}
Recall (\ref{Y-eta-e-alha-e-be}):
\begin{align*}
Y_L^{\eta}(e_\al,x)e_\be=\varepsilon(\al,\be)x^{\<\al,\be\>}e^{\<\be,\eta(\al,x)\>}E^{-}(-\al,x)e_{\al+\be}.\nonumber
\end{align*}
Using Lemma \ref{eta-kappa-relation} we get
$$e^{\<\al_i+\al_j,\eta(\al_i,x)\>}=e^{\<\al_i,\eta(\al_i,x)\>}e^{\<\al_j,\eta(\al_i,x)\>}
=\kappa_{i,i}(x)\kappa_{i,j}(x)(x+2\hbar)(x+\hbar)^{-1}.$$
As $\varepsilon$ is a bilinear character with $\varepsilon(\al_i,\al_i)=1$, we have
\begin{align*}
  Y_L^\eta(e_{\pm\al_i},x)e_{\pm(\al_i+\al_j)}
  =\epsilon(\al_i,\al_j) A_i(x)^{\pm 1}e_{\pm(2\al_i+\al_j)}\kappa_{i,i}(x)\kappa_{i,j}(x)
  (x+\hbar)\inv x(x+2\hbar).
\end{align*}
Combining this with \eqref{eq:Y-eta-A-i-comm} and (\ref{ei-0-ej}), we obtain
\begin{align*}
&Y_L^\eta(e_{\pm\al_i},x)(e_{\pm\al_i})_0^{\eta}e_{\pm\al_j}\\
=\ &\epsilon(\al_i,\al_j)^2(x+3\hbar)
\kappa_{i,j}(-\hbar)\kappa_{i,j}(x)\kappa_{i,i}(x+\hbar)A_i(x)^{\pm 1}A_i(-\hbar)^{\pm 1}e_{\pm (2\al_i+\al_j)},
\end{align*}
which lies in $V_L[[\hbar,x]]$. Then  $(e_{\pm\al_i})_0^{\eta}(e_{\pm\al_i})_0^{\eta} e_{\pm\al_j}=0$.
\end{proof}

\subsection{Twisted quantum affine algebras and the main result}
Here, we recall the Drinfeld realization of twisted quantum affine algebras from \cite{Dr-new}
and give a different presentation on restricted modules.

For $i\in I$, set
\begin{align}
\mathcal O(i)=\{ \mu^k(i)\ |\ k\in \Z_N\}\subset I\ \text{and }N_i=|\mathcal{O}(i)|.
\end{align}
Furthermore, set
\begin{align}
d_i=|\{ k\in \Z_N\ |\ \mu^k(i)=i\}|=N/N_i.
\end{align}

The following result can be found in \cite[Section 2.2]{FSS}:

\begin{lem}\label{lem:linking}
For every $i\in I$, the Dynkin diagram of the orbit $\mathcal O(i)$ is either

\te{(1)} a direct sum of some copies of $A_1$, in which case set $s_i=1$,  or

\te{(2)} a direct sum of some copies of $A_2$ with $N_i$ even and $a_{\mu^{N_i/2}(i),i}=-1$,
in which case set $s_i=2$.
\end{lem}


For $i,j\in I$, set
\begin{align}\label{gammaij}
d_{ij}=|\Z_N^{(i,j)}|,\quad \text{where }\Z_N^{(i,j)}=\{k\in \Z_N\mid a_{\mu^k(i),j}\ne 0\}\subset \Z_N.
\end{align}
Note that $d_{j,i}=d_{i,j}$, $d_{\mu(i),j}=d_{i,j}$, and $d_i|d_{i,j}$. From Lemma \ref{lem:linking} we have
\begin{align}
d_{i,i}=s_id_i.
\end{align}

Follow \cite{Dr-new} (cf. \cite{J-inv}, \cite{Dam}, \cite{CJKT-qeala-II-twisted-qaffinization}) to introduce polynomials
\begin{align}\label{pij}
p_{i,j}^\pm(x_1,x_2)=
    \left(
        x_1^{d_i}+q^{\mp d_i}x_2^{d_i}
    \right)^{s_i-1}
    \frac{
        q^{\pm 2d_{ij}}x_1^{d_{ij}}-x_2^{d_{ij}}
    }{
        q^{\pm 2d_i}x_1^{d_{i}}-x_2^{d_i}
    }.
\end{align}


\begin{de}
For $i,j\in I$, set
\begin{align}\label{eq:def-F-G}
&F_{i,j}^\pm(x_1,x_2)=\prod_{k\in\Z_N;\ a_{\mu^k(i),j}\ne 0}(x_1-\xi^{-k} q^{\pm a_{\mu^k(i),j}}x_2),\\
&G_{i,j}^\pm(x_1,x_2)=\prod_{k\in\Z_N;\ a_{\mu^k(i),j}\ne 0}(q^{\pm a_{\mu^k(i),j}}x_1-\xi^{-k}x_2),
\end{align}
which are elements of $\C[q, q^{-1}][x_1,x_2]$.
\end{de}

The following are some simple facts:
\begin{align}
&F_{i,\mu(j)}^\pm(x_1,x_2)=F_{i,j}^{\pm}(x_1,\xi^{-1}x_2),\ \  F_{\mu(i),j}^\pm(x_1,x_2)=F_{i,j}^{\pm}(x_1,\xi x_2), \\
&G_{i,\mu(j)}^{\pm}(x_1,x_2)=G_{i,j}^{\pm}(x_1,\xi^{-1}x_2),\ \  G_{\mu(i),j}^{\pm}(x_1,x_2)=G_{i,j}^{\pm}(x_1,\xi x_2),\\
&G_{i,j}^{\pm}(x_1,x_2)=q^{\pm \bar{a}_{ij}}F_{i,j}^{\mp}(x_1,x_2),\label{G-F-relation}\\
&G_{i,j}^{\pm}(x_1,x_2)=\left(\prod_{k\in \Z_N;\ a_{\mu^k(i),j}\ne 0}(-\xi^{-k})\right) F_{j,i}^{\pm}(x_2,x_1),
\end{align}
where
\begin{align}
\bar{a}_{i,j}=\sum_{k\in \Z_N} a_{\mu^k(i),j}\in \Z.
\end{align}

Recall from Definition \ref{def-tilde-g-ij} that for $i,j\in I$,
\begin{align*}
\tilde{g}_{i,j}(x)=\prod_{k\in \Z_N} \frac{ q^{a_{\mu^k(i),j}}-\xi^{-k} x}
{1-\xi^{-k}q^{a_{\mu^k(i),j}}x}=\prod_{k\in \Z_N;\ a_{\mu^k(i),j}\ne 0} \frac{ q^{a_{\mu^k(i),j}}-\xi^{-k} x}
{1-\xi^{-k}q^{a_{\mu^k(i),j}}x},
\end{align*}
 a {\em rational function} lying in $\C[q,q^{-1}](x)$.

\begin{de}
For $i,j\in I$,  denote by
$g_{i,j}(x)$ the Taylor series of $\tilde{g}_{i,j}(x)$:
\begin{align}
g_{i,j}(x)=\iota_{x;0}(\tilde{g}_{i,j}(x)) \in \C[q,q^{-1}][[x]].
\end{align}
\end{de}

Note that as $\tilde{g}_{i,j}(0)=q^{\bar{a}_{i,j}}$, $\tilde{g}_{i,j}(x)$ is analytic at $0$ and
$g_{i,j}(x)$ is an invertible element of $\C[q,q^{-1}][[x]]$.
We have
\begin{align}
&\quad\quad  \tilde{g}_{i,j}(x)\tilde{g}_{j,i}(1/x)=1,\\
&\tilde{g}_{i,j}(x_2/x_1)=\frac{G^+_{i,j}(x_1,x_2)}{F^+_{i,j}(x_1,x_2)}=\frac{F^-_{i,j}(x_1,x_2)}{G^{-}_{i,j}(x_1,x_2)}
\quad \text{ for }i,j\in I.\label{g-G-F}
\end{align}

{\bf Convention:} We view $g_{i,j}(x_2/x_1)^{\pm 1}$ as elements of $\C[q,q^{-1}]((x_1))((x_2))$:
\begin{equation}
g_{i,j}(x_2/x_1)^{\pm 1}=g_{i,j}(x)^{\pm 1}|_{x=x_2/x_1}
=\iota_{x_1,x_2}\left(\frac{G^+_{i,j}(x_1,x_2)}{F^+_{i,j}(x_1,x_2)}\right)^{\pm 1}.
\end{equation}

For $i,j\in I$, we have
\begin{align}
\log g_{i,j}(x)&=\sum_{k\in \Z_N}\log \left( q^{a_{\mu^k(i),j}}\cdot
\frac{1-\xi^{-k}q^{-a_{\mu^k(i),j}}x}{1-\xi^{-k}q^{a_{\mu^k(i),j}}x}\right)\\
&=\bar{a}_{i,j}\hbar+\sum_{k\in \Z_N}
\left(q^{-a_{\mu^k(i),j}x\partial_{x}}-q^{a_{\mu^k(i),j}x\partial_{x}}\right)\log (1-\xi^{-k}x),\nonumber
\end{align}
(where $\bar{a}_{i,j}=\sum_{k\in \Z_N}a_{\mu^k(i),j}$), and
\begin{align}
x\partial_{x}\log g_{i,j}(x)
=\sum_{k\in \Z_N}\left(q^{a_{\mu^k(i),j}x\partial{x}}-q^{-a_{\mu^k(i),j}x\partial_{x}}\right)\frac{\xi^{-k}x}{1-\xi^{-k}x}.
\end{align}

Note that  as  $A$ is simply-laced by assumption, for $k\in \Z_N, i, j\in I$,
we have
$$a_{\mu^k(i),j}=\<\alpha_{\mu^k(i)},\alpha_j\>=-1, 0, \text{ or }2,$$
where $a_{\mu^k(i),j}=2$ if and only if $\mu^k(i)=j$.
Then
\begin{align}
 F_{i, j}^{\pm}(x_1,x_2)=\prod_{k\in\Z_N;\ a_{\mu^k(i),j}=-1}(x_1-\xi^{-k} q^{\mp 1}x_2)
\prod_{k\in\Z_N;\ \mu^k(i)=j}(x_1-\xi^{-k} q^{\pm 2}x_2).
 \end{align}


The following notion was introduced in \cite{Dr-new} (cf. \cite{J-inv}):

\begin{de}\label{de:tqaffine}
{\em Twisted quantum affine algebra} $\qtar$ is defined to be a topological
 unital associative  $\C[[\hbar]]$-algebra generated by vectors
\begin{eqnarray}\label{eq:tqagenerators}
c,\ H_{i}(n),\  E^\pm_{i}(n)\  \   (\text{where }i\in I,\  n\in\Z)
\end{eqnarray}
with $c$ central, subject to a set of relations written in terms of generating functions:
\begin{eqnarray}
& \Psi_i^\pm(x)=q^{\pm H_{i}(0)}\, \te{exp}
    \left(
        \pm (q-q\inverse)\sum\limits_{\pm n> 0}H_{i}(n)x^{-n}
    \right),\\
&E^\pm_i(x)=\sum\limits_{n\in \Z} E^\pm_{i}(n)x^{-n}.
\end{eqnarray}
The following are the relations for $i,j\in I$:
\begin{align*}
&\te{(Q0)  }&&E^\pm_{\mu(i)}(x)=E^\pm_i(\xi\inverse x),
    \quad \Psi^\pm_{\mu(i)}(x)=\Psi^\pm_i(\xi\inverse x),\\
&\te{(Q1)  }&& [\Psi_i^\pm(x_1),\Psi_j^\pm(x_2)]=0,\\
&\te{(Q2) }&& \Psi^+_i(x_1)\Psi^-_j(x_2)=\Psi^-_j(x_2)\Psi^+_i(x_1)
    g_{i,j}(q^c x_2/x_1)\inverse g_{i,j}(q^{-c} x_2/x_1),\\
&\te{(Q3) }&&\Psi^+_i(x_1)E^\pm_j(x_2)=E^\pm_j(x_2)\Psi^+_i(x_1)
    g_{i,j}(q^{\mp\half c}x_2/x_1)^{\pm 1}
    ,\\
&\te{(Q4) }&& \Psi^-_i(x_1)E^\pm_j(x_2)=E^\pm_j(x_2)\Psi^-_i(x_1)
    g_{j,i}(q^{\mp\half c}x_1/x_2)^{\mp 1},\\
&\te{(Q5)  }&&[E_i^+(x_1),E_j^-(x_2)]
=\frac{1}{q-q\inverse}\\
 &&&  \ \  \times
    \ksum\delta_{\mu^k(i),j} \Bigg(
        \Psi_j^+(q^{\half c}x_2)\delta\left(
            \frac{q^c\xi^{-k}x_2}{x_1}
        \right)
        -
        \Psi_j^-(q^{-\half c}x_2)\delta\left(
            \frac{q^{-c}\xi^{-k}x_2}{x_1}
        \right)
    \Bigg),\\
&\te{(Q6)  }&&F^\pm_{i,j}(x_1,x_2)E^\pm_i(x_1)E^\pm_j(x_2)=
    G^\pm_{ i,j}(x_1,x_2)E^\pm_j(x_2)E^\pm_i(x_1),\\
&\te{(Q7)  }&&\sum_{\sigma\in S_{2}}\Big\{
    p_{i,j}^\pm(x_{\sigma(1)},x_{\sigma(2)})\big(\sum_{r=0}^{2}
    (-1)^r \binom{2}{r}_{q^{d_{ij}}}E_i^\pm(x_{\sigma(1)})\cdots E_i^\pm(x_{\sigma(r)})\\
 &\qquad&&\quad   \cdot E_j^\pm(y)
       E_i^\pm(x_{\sigma(r+1)})\cdots E_i^\pm(x_{\sigma(2)})\big)
\Big\}=0\quad
    \te{if }j\notin \mathcal{O}(i),\ a_{i,j}=-1,\\
&\te{(Q8) }&&\sum_{\sigma\in S_{3}}
    \left(q^{\mp \frac{3}{2}}x_{\sigma(1)}
        -(q^{\frac{1}{2}}+q^{-\frac{1}{2}})
            x_{\sigma(2)}
        +q^{\pm \frac{3}{2}}
            x_{\sigma(3)}\right)
                E_i^\pm(x_{\sigma(1)})E_i^\pm(x_{\sigma(2)})E_i^\pm(x_{\sigma(3)})
=0\\
&&&\quad\quad\quad\quad
    \te{if }s_i=2.
\end{align*}
\end{de}

\begin{rem}
{\em For $i\in I$, form generating functions
\begin{align}
\psi_i^{\pm}(x)=\pm \hbar \(H_i(0)+\frac{e^{\hbar}-e^{-\hbar}}{\hbar}\sum_{n\in \Z_+}H_i(\pm n)x^{\mp n}\).
\end{align}
 Then $\Psi_i^{\pm}(x)=\exp \psi_i^{\pm}(x)$. We see that  for $i,j\in I$,
$[\Psi_i^{\pm}(x_1),\Psi_j^{\pm}(x_2)]=0$ if and only if  $[\psi_i^{\pm}(x_1),\psi_j^{\pm}(x_2)]=0$,
which is equivalent to
\begin{align}
[H_i(m),H_j(n)]=0=[H_i(-m),H_j(-n)]\quad \text{ for } m,n\in \N.
\end{align}}
\end{rem}

Denote by $\qtarl$ the topological unital associative algebra over $\C[[\hbar]]$ generated by
the vectors in \eqref{eq:tqagenerators}, subject to the relations (Q0-Q6).
By definition, $\qtar$ is naturally a quotient algebra of $\qtarl$.

Next, we discuss restricted  modules for $\qtarl$ and $\qtar$.
A $\qtarl$-module $W$ is said to be \emph{restricted} if
$W$ is a topologically free $\C[[\hbar]]$-module and if
\begin{align}
  \Psi_i^\pm(x),  \   E_i^\pm(x)\in\E_\hbar(W)\quad\te{for }i\in I.
\end{align}
A $\qtar$-module is said to be \emph{restricted} if it is a restricted $\qtarl$-module.


\begin{de}
For $i,j\in I$, set
\begin{align}
 & f_{i,j}^\pm(x_1,x_2)=F_{i,j}^\pm(x_1,x_2)\prod_{\substack{k\in\Z_N\\ \mu^k(i)=j}}(x_1-\xi^{-k} x_2)\inv
 \in \C[q,q^{-1}]((x_1))((x_2)),\\
&C_{i,j}=\prod_{\substack{k\in\Z_N\\ a_{\mu^k(i),j}<0}}(-\xi^{-k})
 =\prod_{\substack{k\in\Z_N\\ a_{\mu^k(i),j}=-1}}(-\xi^{-k})\in \C^{\times},
\end{align}
where it is understood that $f_{i,j}^\pm(x_1,x_2)=F_{i,j}^\pm(x_1,x_2)$ if $j\notin \mathcal{O}(i)$
and $C_{i,j}=1$ if $a_{\mu^k(i),j}\ge 0$ for all $k\in \Z_N$.
\end{de}

Note that if $j\in \mathcal{O}(i)$ with $j=\mu^k(i)$ for some $k\in \Z_N$, then
\begin{align}
f_{i,j}^\pm(x_1,x_2)=\frac{F_{i,j}^\pm(x_1,x_2)}{x_1^{d_i}-\xi^{-kd_i}x_2^{d_i}}.
\end{align}

The following is a result of \cite{CJKT-qeala-II-twisted-qaffinization} (Proposition 5.6):

\begin{lem}\label{lem:locality}
Let $W$ be a restricted $\qtarl$-module. For $i_1,\dots,i_m\in I$, set
\begin{align}\label{eq:multi-func}
  E_{i_1,\dots,i_m}^\pm(x_1,\dots,x_m)=
  \left(\prod_{1\le r<s\le m}f_{i_r,i_s}^\pm(x_r,x_s)\right)E_{i_1}^\pm(x_1)\cdots E_{i_m}^\pm(x_m).
\end{align}
Then for every positive integer $n$,
\begin{align}
  E_{i_1,\dots,i_m}^\pm(x_1,\dots,x_m)w\in W((x_1,\dots,x_m))+\hbar^n W[[x_1^{\pm 1},\dots,x_m^{\pm 1}]]
\end{align}
for all $w\in W$.
Moreover, for $\sigma\in S_m$ (the symmetric group), we have
\begin{align}
  E_{i_1,\dots,i_m}^\pm(x_1,\dots,x_m)=\(\prod_{\substack{1\le s<t\le m\\ \sigma(s)>\sigma(t)}}C_{i_s,i_t}\)
  E_{i_{\sigma(1)},\dots,i_{\sigma(m)}}^\pm(x_{\sigma(1)},\dots,x_{\sigma(m)}).
\end{align}
\end{lem}

As a special case of \cite[Theorem 5.14]{CJKT-qeala-II-twisted-qaffinization},
we have:

\begin{prop}\label{prop:qSerre}
Let $W$ be a restricted $\qtarl$-module.
Then $W$ reduces to a $\qtar$-module if and only if for $i,j\in I$ with $a_{i,j}=-1$,
\begin{align}
  E_{i,i,j}^\pm(q^{-1}x,qx,x)=0.
\end{align}
\end{prop}


For $i,j\in I$, simply write $F_{i,j}(x_1,x_2)$, $G_{i,j}(x_1,x_2)$ for $F_{i,j}^{+}(x_1,x_2)$, $G_{i,j}^{+}(x_1,x_2)$.
The following is a key result:

\begin{prop}\label{prop:locality-DY}
Let $W$ be a topologically free $\C[[\hbar]]$-module and assume
$$H_{i}(n),\  E^\pm_{i}(n),\ c\in \End W \   (\text{where }i\in I,\  n\in\Z)$$
such that
\begin{align*}
&[c,H_i(n)]=0=[c,E^{\pm}_i(n)]\quad \text{ for }i\in I,\ n\in \Z,\\
&\lim_{n\rightarrow \infty}H_i(n)w=0=\lim_{n\rightarrow \infty}E^{\pm}_i(n)w\quad \text{ for }w\in W.
\end{align*}
Define $\Psi_i^{\pm}(x), E_i^{\pm}(x)$ for $i\in I$ as in Definition \ref{de:tqaffine}.  On the other hand,  define
  \begin{align}
h_{i,q}^\pm(x)= \frac{1}{2}H_i(0)+ \sum_{\pm n\in \Z_+}\frac{n}{[n]_q}q^{\mp \frac{1}{2}nc}H_i(n)x^{-n},
   \end{align}
and furthermore set
\begin{align}
  &h_{i,q}(x)=h_{i,q}^+(x)+h_{i,q}^-(x), \label{eq:def-h-x-1}\\
  & e_{i,q}^+(x)=E_i^+(x), \quad  e_{i,q}^-(x)=E_i^-(xq^{-c})\Psi_i^+(xq^{-\half c})\inv.\label{eq:def-h-x-2}
\end{align}
 Then $W$ is a $\qtar$-module if and only if the following relations hold for $i,j\in I$:
\begin{align*}
  &\te{\em(Q0$'$) }&&h_{\mu(i),q}(x)=h_{i,q}(\xi\inv x),\quad e_{\mu(i),q}^\pm(x)=e_{i,q}^\pm(\xi\inv x),\\
  &\te{\em(Q2$'$) }&&[h_{i,q}(x_1),h_{j,q}(x_2)]=\sum_{k\in\Z_N}[a_{\mu^k(i), j}]_{q^{x_2\partial_{x_2}}}[c]_{q^{x_2\partial_{x_2}}}
  \nonumber\\
  &&&\quad\times
  \(\iota_{x_1,x_2}q^{-c x_2\partial_{x_2}}-\iota_{x_2,x_1}q^{c x_2\partial_{x_2}}\) \frac{\xi^{-k}x_2/x_1}{(1-\xi^{-k}x_2/x_1)^2},\\
  &\te{\em(Q3$'$) }&&[h_{i,q}(x_1),e_{j,q}^\pm(x_2)]=\pm e_{j,q}^\pm(x_2)\sum_{k\in\Z_N}[a_{\mu^k(i),j}]_{q^{x_2\partial_{x_2}}}
  \nonumber\\
  &&&\quad\times
  \(\iota_{x_1,x_2}q^{-c x_2\partial_{x_2}}-\iota_{x_2,x_1}q^{c x_2\partial_{x_2}}\)
  \half\(\frac{1+\xi^{-k}x_2/x_1}{1-\xi^{-k}x_2/x_1}\),\\
  &\te{\em(Q4$'$) }&&e_{i,q}^+(x_1)e_{j,q}^-(x_2)-g_{j,i}(x_1/x_2)e_{j,q}^-(x_2)e_{i,q}^+(x_1)
  =\frac{1}{q-q\inv}\nonumber\\
  &&&\times
\sum_{k\in\Z_N}\delta_{\mu^k(i),j}  \(\delta\(\frac{\xi^{-k}x_2}{x_1}\)-\Psi_j^-(x_2q^{-\frac 32 c})\Psi_j^+(x_2q^{-\half c})\inv \delta\(\frac{\xi^{-k}q^{-2c}x_2}{x_1}\)\),\\
  &\te{\em(Q5$'$) }&&F_{i,j}(x_1,x_2)e_{i,q}^\pm(x_1)e_{j,q}^\pm(x_2)
  =G_{i,j}(x_1,x_2)e_{j,q}^\pm(x_2)e_{i,q}^\pm(x_1),\\
  &\te{\em(Q6$'$) }&&e_{i,i,j;q}^\pm(qx,q\inv x,x)=0\quad\te{if }a_{i,j}=-1,
\end{align*}
where $e^{\pm}_{i,i,j;q}(x_1,x_2,x_3)$ are defined as in (\ref{eq:multi-func})
with $e^{\pm}_{i,q}(x)$ in place of $E^{\pm}_i(x)$.
\end{prop}

\begin{proof} Set $G(x)=\frac{e^x-e^{-x}}{2x}\in \C[[x]]$ and
let $F(x)$ be the Taylor series of function $\frac{2x}{e^x-e^{-x}}$ (at $0$).
Then
$F(x), G(x) \in \C[[x]]$ with $F(0)=1=G(0)$ such that
$$ F(x)G(x)=1,\  \  \  \frac{1}{2}F(x)(e^x-e^{-x})=x,\  \  \ G(x)x=e^x-e^{-x}.$$

Note that  for $n\in \Z\backslash\{0\}$,
$$F(\hbar x\partial_{x})x^{-n}=F(-n\hbar)x^{-n}=\frac{-2n\hbar}{e^{-n\hbar}-e^{n\hbar}}x^{-n}
=\frac{2\hbar}{q-q^{-1}}\frac{n}{[n]_q}x^{-n}$$
and $F(\hbar x\partial_{x})1=1$. Then
 \begin{align}\label{eq:def-h-x-0}
h_{i,q}^\pm(x)=\ & F(\hbar x\partial_{x})\(\frac{1}{2}H_i(0)+\frac{q-q^{-1}}{2\hbar}
    \sum_{\pm n\in \Z_+}q^{\mp \frac{1}{2}nc}H_i(n)x^{-n}\)\\
=\ &\pm \frac{1}{2\hbar}F(\hbar x\partial_{x})
     \log \Psi_i^\pm(xq^{\pm\half c}). \nonumber
   \end{align}
On the other hand, we have
\begin{align}
 & \Psi_{i}^\pm(x)=\exp\( \pm 2\hbar G(\hbar x\partial_{x})h_{i,q}^\pm(xq^{\mp\half c})\),\label{Psi-h}\\
&  E_i^+(x)=e_{i,q}^+(x), \quad E_i^-(x)=e_{i,q}^-(xq^{c}) \Psi_{i}^+(xq^{\half c}).\label{E-e}
\end{align}

In the following, we show that (Q2) together with (Q1) is  equivalent to (Q$2'$),
(Q3) together with (Q4) is equivalent to (Q$3'$). It is straightforward to show that  in the presence of (Q3),
(Q4$'$) is equivalent to (Q5), and (Q5$'$) is equivalent to (Q6).

Note that
\begin{align*}
&\log g_{i,j}(x)^{-1}g_{i,j}(q^{-2c}x)=\log g_{i,j}(q^{-2c}x)-\log g_{i,j}(x)\\
=\ & \sum_{k\in \Z_N}\(q^{a_{\mu^k(i),j}x\partial_{x}}-q^{-a_{\mu^k(i),j}x\partial_{x}}-
q^{(a_{\mu^k(i),j}-2c)x\partial_{x}}+q^{-(2c+a_{\mu^k(i),j})x\partial_{x}}\)  \log(1-\xi^{-k}x)\\
=\ &  \sum_{k\in \Z_N}\(q^{a_{\mu^k(i),j}x\partial_{x}}-q^{-a_{\mu^k(i),j}x\partial_{x}}\)
\(q^{c x\partial_{x}}-q^{-c x\partial_{x}}\)q^{-c x\partial_{x}} \log(1-\xi^{-k}x).
\end{align*}
With this we see that (Q2) is equivalent to
\begin{align}\label{log-Psi-comm}
  &[\log \Psi_i^+(x_1q^{\half c}),\log \Psi_j^-(x_2q^{-\half c})]\\
  =\ &(q^{x_2\partial_{x_2}}-q^{-x_2\partial_{x_2}})^2\sum_{k\in\Z_N}
  [a_{\mu^k(i),j}]_{q^{x_2\partial_{x_2}}}[c]_{q^{x_2\partial_{x_2}}}q^{-c x_2\partial_{x_2}}\log(1-\xi^{-k}x_2/x_1).\nonumber
\end{align}
Note that $\frac{1}{2\hbar}F(\hbar x\partial_{x})(q^{x\partial_{x}}-q^{-x\partial_{x}})=x\partial_{x}$,
 $2\hbar G(\hbar x\partial_{x})=q^{x\partial_{x}}-q^{-x\partial_{x}}$, and $x_1\partial_{x_1}M(x_2/x_1)=-x_2\partial_{x_2}M(x_2/x_1)$
 for $M(z)\in \C[[z,z^{-1}]]$.
 Then (\ref{log-Psi-comm}) is equivalent to
\begin{align}
 & [h_{i,q}^+(x_1),h_{j,q}^-(x_2)]\label{h-i-pm-j-mp}\\
  =\ & \sum_{k\in\Z_N}[a_{\mu^k(i),j}]_{q^{x_2\partial_{x_2}}}[c]_{q^{x_2\partial_{x_2}}}
  q^{-c x_2\partial_{x_2}}x_1\partial_{x_1}x_2\partial_{x_2}\log(1-\xi^{-k}x_2/x_1),\nonumber
  \end{align}
  where
  \begin{align*}
  \partial_{(x_1)}\partial_{(x_2)}\log(1-\xi^{-k}x_2/x_1)=\frac{\xi^{-k}x_2/x_1}{(1-\xi^{-k}x_2/x_1)^2}.
\end{align*}
It is clear that  (Q2$'$) amounts to (\ref{h-i-pm-j-mp}) and $[h_{i,q}^{\pm}(x_1),h_{j,q}^{\pm}(x_2)]=0$.

In the presence of (Q1) and (Q2),  we see that (Q3) and (Q4)  are equivalent to
\begin{align}
&\Psi_i^+(x_1)e_{j,q}^\pm(x_2)=e_{j,q}^\pm(x_2)\Psi_i^+(x_1)g_{i,j}(q^{-\half c}x_2/x_1)^{\pm 1},\label{eq:locality-DY-temp001}\\
&\Psi_i^-(x_1)e_{j,q}^\pm(x_2)=e_{j,q}^\pm(x_2)\Psi_i^-(x_1)g_{j,i}(q^{-\half c}x_1/x_2)^{\mp 1}.\label{eq:locality-DY-temp002}
\end{align}
Note that \eqref{eq:locality-DY-temp001} is equivalent to
\begin{align*}
\left[\log \Psi_i^+(x_1q^{\half c}),e_{j,q}^\pm(x_2)\right]=\pm e_{j,q}^\pm(x_2)\log g_{i,j}(q^{-c}x_2/x_1),
\end{align*}
where
\begin{align*}
&\log g_{i,j}(q^{-c}x_2/x_1)\nonumber\\
=\ &\bar{a}_{i,j}\hbar
-(q^{x_2\partial_{x_2}}-q^{-x_2\partial_{x_2}})
  \sum_{k\in\Z_N}[a_{\mu^k(i),j}]_{q^{x_2\partial_{x_2}}}
  q^{-c x_2\partial_{x_2}}   \log (1-\xi^{-k}x_2/x_1).
\end{align*}
Then we have
\begin{align}\label{eq:locality-DY-temp004}
  &[h_{i,q}^+(x_1),e_{j,q}^\pm(x_2)]=\pm e_{j,q}^\pm(x_2)\sum_{k\in\Z_N}[a_{\mu^k(i),j}]_{q^{x_2\partial_{x_2}}}
    q^{-c x_2\partial_{x_2}}\left(\frac{1}{2}+\frac{\xi^{-k}\frac{x_2}{x_1}}{1-\xi^{-k}\frac{x_2}{x_1}}\right).
\end{align}
Similarly,  \eqref{eq:locality-DY-temp002} is equivalent to
\begin{align}\label{eq:locality-DY-temp005}
  &[h_{i,q}^-(x_1),e_{j,q}^\pm(x_2)]=\pm e_{j,q}^\pm(x_2)\sum_{k\in\Z_N}[a_{\mu^k(i),j}]_{q^{x_2\partial_{x_2}}}
    q^{c x_2\partial_{x_2}}\left(\frac{1}{2}+\frac{\xi^{k}\frac{x_1}{x_2}}{1-\xi^{k}\frac{x_1}{x_2}}\right).
\end{align}
Notice that
\begin{align*}
\frac{1}{2}+\frac{\xi^{-k}x_2/x_1}{1-\xi^{-k}x_2/x_1}
=\frac{1}{2}\frac{1+\xi^{-k}x_2/x_1}{1-\xi^{-k}x_2/x_1},\ \ \ \
\frac{1}{2}+\frac{\xi^{k}x_1/x_2}{1-\xi^{k}x_1/x_2}
=-\frac{1}{2}\frac{1+\xi^{-k}x_2/x_1}{1-\xi^{-k}x_2/x_1}.
\end{align*}
Then (Q3$'$) follows immediately.

The equivalence between (Q5$'$) and (Q6) follows from the following facts (from (Q3), (\ref{g-G-F}),  (\ref{G-F-relation})):
\begin{align*}
&\Psi_i^{+}(x_1q^{-\frac{1}{2}c})^{-1}E_j^{-}(x_2q^{-c})=
E_j^{-}(x_2q^{-c})\Psi_i^{+}(x_1q^{-\frac{1}{2}c})^{-1}g_{i,j}(x_2/x_1),\\
&E_i^{-}(x_1q^{-c})\Psi_j^{+}(x_2q^{-\frac{1}{2}c})^{-1}=
\Psi_j^{+}(x_2q^{-\frac{1}{2}c})^{-1}E_i^{-}(x_1q^{-c})g_{j,i}(x_1/x_2)^{-1},\\
&F_{i,j}^{+}(x_1,x_2)g_{i,j}(x_2/x_1)=G^{+}_{i,j}(x_1,x_2)=q^{\bar{a}_{i,j}}F_{i,j}^{-}(x_1,x_2),\\
&q^{\bar{a}_{i,j}}G_{i,j}^{-}(x_1,x_2)g_{j,i}(x_1/x_2)^{-1}=F_{i,j}^{+}(x_1,x_2)g_{j,i}(x_1/x_2)^{-1}=G^{+}_{i,j}(x_1,x_2).
\end{align*}

Finally, consider relation (Q6$'$).
In the presence of (Q5$'$), from \cite{CJKT-qeala-II-twisted-qaffinization} (Lemma 5.4), we have
\begin{align}\label{braided-relation-e-i-j}
  f_{i,j}^+(x_1,x_2)e_{i,q}^\pm(x_1)e_{j,q}^\pm(x_2)
  =C_{i,j}f_{j,i}^+(x_2,x_1)e_{j,q}^\pm(x_2)e_{i,q}^\pm(x_1)\quad \text{for }i,j\in I.
\end{align}
For $i_1,i_2,\dots,i_n\in I$, set
\begin{align}
  e_{i_1,\dots,i_n;q}^\pm(x_1,\dots,x_n)
  =\prod_{1\le a<b\le n}f_{i_a,i_b}^+(x_a,x_b)
  e_{i_1,q}^\pm(x_1)\cdots e_{i_n,q}^\pm(x_n),
\end{align}
which lie in $\E_{\hbar}^{(n)}(W)$. From definition, we have $e^{+}_{i,i,j;q}(x_1,x_2,x_3)=E_{i,i,j}^{+}(x_1,x_2,x_3)$ and
\begin{align*}
&e^{-}_{i,i,j;q}(x_1,x_2,x_3)=g_{i,i}(x_2/x_1)g_{i,j}(x_3/x_1)g_{i,j}(x_3/x_2)\\
&\quad \quad \cdot E_{i,i,j}^{-}(q^{-c}x_1,q^{-c}x_2,q^{-c}x_3)
\Psi_i^+(x_1q^{-\half c})\inv\Psi_i^+(x_2q^{-\half c})\inv\Psi_j^+(x_3q^{-\half c})\inv.
\end{align*}
From this we see that $e_{i,i,j;q}^\pm(qx,q\inv x,x)=0$ if and only if $E_{i,i,j}^\pm(qx,q\inv x,x)=0$.
Then it follows from Proposition \ref{prop:qSerre}.
\end{proof}

Now, let $W$ be a restricted $\qtarl$-module.  
Using (\ref{braided-relation-e-i-j}) we get
\begin{align}\label{e-quasi-compatbilty}
  e_{i_1,\dots,i_n;q}^\pm(x_1,\dots,x_n)\in \E_\hbar^{(n)}(W).
\end{align}
Then $U^\pm:=\set{e_{i,q}^\pm(x)}{i\in I}$ are $\hbar$-adically quasi compatible subsets of $\E_\hbar(W)$.
By Theorem \ref{thm:abs-construct}, $U^\pm$ generate $\hbar$-adic nonlocal vertex algebras
$\<U^\pm\>_\phi\subset\E_\hbar(W)$ with $Y_\E^\phi(\cdot,z)$ denoting the vertex operator map,
where for $\al(x),\beta(x)\in\<U^\pm\>_\phi$,
\begin{align*}
  Y_\E^\phi(\al(x),z)\beta(x)=\sum_{n\in\Z}\al(x)_n^\phi\beta(x)z^{-n-1}.
\end{align*}

We have the following result:

\begin{prop}\label{prop:Q6'=Q7'}
Let $W$ be a restricted $\qtarl$-module. Assume $i,j\in I$ with $a_{i,j}=\<\al_i,\al_j\>=-1$. Then
\begin{align}
  e_{i,q}^\pm(x)_0^\phi e_{i,q}^\pm(x)_0^\phi e_{j,q}^\pm(x)=c(x)e_{i,i,j,q}^\pm(q\inv x,qx,x)
\end{align}
for  some invertible $c(x)\in \C((x))[[\hbar]]$. Furthermore, (Q6$'$) holds if and only if
\begin{align*}
  &\te{\em (Q7$'$)}\qquad e_{i,q}^\pm(x)_0^\phi e_{i,q}^\pm(x)_0^\phi e_{j,q}^\pm(x)=0.
\end{align*}
\end{prop}

\begin{proof} By dividing it into cases $j\notin \mathcal{O}(i)$ and $j\in \mathcal{O}(i)$,
 it is straightforward to show that $f_{i,j}^{+}(xe^z,x)(z+a_{i,j}\hbar)^{-1}$ is an invertible element of $\C((x))[[z,\hbar]]$.
Set
\begin{align*}
 p(x,z)=\frac{z+a_{i,j}\hbar}{f_{i,j}^+(xe^z,x)}=\frac{z-\hbar}{f_{i,j}^+(xe^z,x)},
\end{align*}
which is also an invertible element of $\C((x))[[z,\hbar]]$. With (\ref{e-quasi-compatbilty}), we have
$$Y_\E^\phi(e_{i,q}^\pm(x),z)e_{j,q}^\pm(x)=\iota_{x,z}(f_{i,j}^+(xe^z,x)\inv) e_{i,j;q}^\pm(xe^z,x).   $$
Then
\begin{align*}
  &\Sing_z Y_\E^\phi(e_{i,q}^\pm(x),z)e_{j,q}^\pm(x)
  =\Sing_z \iota_{x,z}f_{i,j}^+(xe^z,x)\inv e_{i,j;q}^\pm(xe^z,x)\\
  =&\Sing_z \frac{1}{z-\hbar}p(x,z)e_{i,j;q}^\pm(xe^z,x)
  =\frac{1}{z-\hbar}p(x,\hbar)e_{i,j;q}^\pm(qx,x),\nonumber
\end{align*}
which implies
\begin{align}\label{two-layers}
p(x,\hbar)e_{i,j;q}^\pm(xq,x)=e_{i,q}^\pm(x)_{0}^{\phi}e_{j,q}^\pm(x)\in \<U^\pm\>_\phi.
\end{align}

Note that $e_{i,j;q}^\pm(xq,x)\in  \E_\hbar(W)$ and
\begin{align*}
  f_{i,i}^+(x_1,x_2q)f_{i,j}^+(x_1,x_2)e_{i,q}^\pm(x_1)e_{i,j;q}^\pm(qx_2,x_2)
  =e_{i,i,j,q}^\pm(x_1,qx_2,x_2).
\end{align*}
As before, we can show that
\begin{align*}
p_2(x,z):=\frac{z+\hbar}{f_{i,i}^+(xe^z,qx) f_{i,j}^+(xe^z,x)}
\end{align*}
is an invertible element of $\C((x))[[z,\hbar]]$. Then using (\ref{two-layers}) we get
\begin{align*}
&\Sing_z Y_\E^\phi(e_{i,q}^\pm(x),z)\left(e_{i,q}^\pm(x)_{0}^{\phi}e_{j,q}^\pm(x)\right)\\
 =\ &\Sing_z Y_\E^\phi(e_{i,q}^\pm(x),z)(p(x,\hbar)e_{i,j;q}^\pm(xq,x))\\
  =\  &\Sing_z f_{i,i}^+(xe^z,xq)\inv f_{i,j}^+(xe^z,x)\inv p(x,\hbar)e_{i,i,j;q}^\pm(xe^z,qx,x)\\
  =\  &\Sing_z \frac{1}{z+\hbar}p_2(x,z)p(x,\hbar)e_{i,i,j;q}^\pm(xe^z,qx,x)\\
  =\  &\frac{1}{z+\hbar}p_2(x,-\hbar)p(x,\hbar)e_{i,i,j;q}^\pm(q\inv x,qx,x),
\end{align*}
which implies
$$e_{i,q}^\pm(x)_0^{\phi}e_{i,q}^\pm(x)_{0}^{\phi}e_{j,q}^\pm(x)
=p_2(x,-\hbar)p(x,\hbar)e_{i,i,j;q}^\pm(q\inv x,qx,x),$$
proving the first assertion. The second assertion follows immediately.
 Note that
\begin{align*}
  e_{i,i,j;q}^\pm(q\inv x,qx,x)=e_{i,i,j;q}^\pm(qx,q\inv x,x)
\end{align*}
because $e_{i,i,j;q}^\pm(x_1,x_2,x)=e_{i,i,j;q}^\pm(x_2,x_1,x)$ as $C_{i,i}=1$.
\end{proof}

Note that from (\ref{6.35}) we have $\log \kappa_{i,j}(x)\in x\C[[x]]+\hbar\C[[x,\hbar]]$ for $i,j\in I$.
Then $$\log \frac{\kappa_{i,i}(0)}{\kappa_{i,i}(2\hbar)}\in \hbar\C[[\hbar]].$$
Now, we are in a position to present the main result of this paper.

\begin{thm}\label{final-theorem}
Let $(W,Y_W)$ be an equivariant $\phi$-coordinated quasi $V_L[[\hbar]]^{\eta}$-module.
Then there is a level $1$ restricted $\qtar$-module structure on $W$ such that
\begin{align*}
  &h_{i,q}(x)=Y_W(\al_i,x)+\frac{1}{4\hbar}\log \frac{\kappa_{i,i}(0)}{\kappa_{i,i}(2\hbar)},\\
  &e_{i,q}^\pm(x)=B_i Y_W(e_{\pm\al_i},x)
\end{align*}
for $i\in I$, where
$$B_i=\(\frac{e^{\hbar}-e^{-\hbar}}{2\hbar}\)^{-1/2}\kappa_{i,i}(0)^{1/2}\in \C[[\hbar]].$$
\end{thm}

\begin{rem}
{\em Note that it can be shown that $B_i=\(\frac{s_i}{d_i}\)^{\frac{1}{2}}\(\left[\frac{d_i}{s_i}\right]_q\)^{\frac{1}{2}}$. }
\end{rem}

\section{Proof of Theorem \ref{final-theorem}}

This section is devoted to the proof of Theorem \ref{final-theorem}.
Let $(W,Y_W)$ be a $(G,\chi_{\phi})$-equivariant $\phi$-coordinated quasi $V_L[[\hbar]]^{\eta}$-module,
which is fixed throughout this section.
We first give some properties of the power series $\kappa_{i,j}(x)$.

\begin{lem}\label{lem:kappa-ij-com}
The following relations hold for $i,j\in I$:
\begin{align*}
  &\kappa_{i,j}(-x)\kappa_{j,i}(x)\inv \cdot  \frac{x-a_{i,j}\hbar}{x+a_{i,j}\hbar}=
  g_{j,i}(e^{-x})^{-1}=g_{i,j}(e^x),\\
  &\kappa_{i,j}(x-2\hbar)=\kappa_{j,i}(-x).
\end{align*}
\end{lem}

\begin{proof} Let $i,j\in I$. Using the relation $\vartheta_{-k}(x)=\vartheta_k(-x)$ and
(\ref{vartheta-diff}), we get
\begin{align*}
 \kappa_{i,j}(-x)\kappa_{j,i}(x)\inv
 =\  &\prod_{k\in\Z_N}\exp\(
   [a_{\mu^k(i),j}]_{q^{-\partial_{x}}}q^{-\partial_{x}}\vartheta_k(-x)
    -[a_{\mu^{k}(j),i}]_{q^{\partial_{x}}}q^{\partial_{x}}\vartheta_{k}(x)  \)\\
  =\  &\prod_{k\in\Z_N}\exp\(
    [a_{\mu^k(j),i}]_{q^{\partial_{x}}}\(q^{-\partial_{x}}-q^{\partial_{x}}\)\vartheta_k(x)  \)\\
  =\  &\prod_{k\in\Z_N}\exp\(
   \(q^{-a_{\mu^k(j),i}\partial_{x}}-q^{a_{\mu^k(j),i}\partial_{x}}\)\vartheta_k(x)  \)\\
  =\  &\prod_{k\in\Z_N}\exp\(\vartheta_k(x-a_{\mu^k(j),i}\hbar)- \vartheta_k(x+a_{\mu^k(j),i}\hbar)\)\\
  =\ &\(\frac{x+a_{j,i}\hbar}{x-a_{j,i}\hbar}\)\prod_{k\in\Z_N}
  \frac{1-\xi^{-k}e^{-x}q^{{a_{\mu^k(j),i}}}} {q^{{a_{\mu^k(j),i}}}-\xi^{-k}e^{-x}}\\
 =\ &\(\frac{x+a_{j,i}\hbar}{x-a_{j,i}\hbar}\)g_{j,i}(e^{-x})^{-1}.
\end{align*}
Using the fact that $q^{-2\partial_x}\ (=e^{-2\hbar \partial_x})$ is an automorphism of $\C[[x,\hbar]]$, we have
\begin{align*}
& \kappa_{i,j}(x-2\hbar)=q^{-2\partial_x} \kappa_{i,j}(x)=q^{-2\partial_x}\exp\(\sum_{k\in\Z_N}
  [a_{\mu^k(i),j}]_{q^{\partial_{x}}}q^{\partial_{x}}\vartheta_k(x) \)\\
  =\ & \exp\(\sum_{k\in\Z_N}
  [a_{\mu^k(i),j}]_{q^{\partial_{x}}}q^{-\partial_{x}}\vartheta_k(x) \)\nonumber\\
   =\ & \exp\(\sum_{k\in\Z_N}
  [a_{i,\mu^{-k}(j)}]_{q^{-\partial_{x}}}q^{-\partial_{x}}\vartheta_{-k}(-x) \)\nonumber\\
   =\ & \exp\(\sum_{k\in\Z_N}
  [a_{\mu^k(j),i}]_{q^{-\partial_{x}}}q^{-\partial_{x}}\vartheta_{k}(-x) \)\nonumber\\
  =\ &\kappa_{j,i}(-x),\nonumber
\end{align*}
as desired.
\end{proof}

Recall from Remark \ref{rem-p(x)=1} that for any $g(x)\in \C(x)$, we have
$g(x_1/x_2), \ g(x_2/x_1)\in \C_{\phi}((x_1,x_2))$
with $\pi_{\phi}g(x_1/x_2)=g(e^z)$ and  $\pi_{\phi}g(x_2/x_1)=g(e^{-z})$.
Set
\begin{align*}
&\C(x_1/x_2)=\{ g(x_1/x_2)\ |\ g(x)\in \C(x)\}\subset \C(x_1,x_2)\cap \C_{\phi}((x_1,x_2)),\\
&\ \  \C_e((z))=\{ g(e^z)\ |\ g(x)\in \C(x)\} \subset \C((z)).
\end{align*}
Then $\pi_{\phi}: \C(x_1/x_2)\rightarrow \C_e((z))$ is an algebra isomorphism.
From definition (see (\ref{def-eta-f})), we have
$\eta^f(\al,x)\in \mathfrak{h}\ot \hbar\C_e((x))[[\hbar]]$ for $\al\in \mathfrak{h}$.
Recall from Lemma \ref{lem:qva-Y-M-def} that
the $B_L[[\hbar]]$-module structure $Y_M^{\eta^f}(\cdot,x)$ on $V_L[[\hbar]]$ is uniquely determined by
$$Y_M^{\eta^f}(e^{\al},x)=\exp \Phi(\eta^f(\al,x))\quad \text{ for }\al\in L.$$
By Theorem \ref{thm:S-op}, for the quantum Yang-Baxter operator $\mathcal{S}(x)$ of $V_L[[\hbar]]^{\eta^f}$ we have
$$\mathcal{S}(x)(u\ot v)\in V_L[[\hbar]]\wh\ot V_L[[\hbar]]\wh\ot \C_{e}((x))[[\hbar]]
\quad \text{for }u,v\in V_L[[\hbar]].$$
This is also true for $V_L[[\hbar]]^{\eta}$ as $V_L[[\hbar]]^{\eta}\simeq V_L[[\hbar]]^{\eta^f}$ by Proposition \ref{prop:eta-iso-eta+tau}.
Then define $$\wh{\mathcal{S}}(x_1,x_2)=(1\ot 1\ot \pi_{\phi}^{-1})\mathcal{S}(x):\
V_L[[\hbar]]\wh\ot V_L[[\hbar]]\rightarrow V_L[[\hbar]]\wh\ot V_L[[\hbar]]\wh\ot \C(x_1,x_2)[[\hbar]].$$

Combining Lemmas \ref{VL-eta-S(X)-special} and \ref{xdx-dx-connection}, we immediately get:

\begin{lem}\label{lem:wh-S}
For $i,j\in I$,
\begin{align*}
  &\wh{\mathcal{S}}(x_2,x_1)(\al_j\ot \al_i)\\
=\ & \al_j\ot\al_i+\vac\ot\vac\sum_{k\in\Z_N}
  \(q^{-a_{\mu^k(i),j}x_2\partial_{x_2}}-q^{a_{\mu^k(i),j}x_2\partial_{x_2}}\)
  \frac{\xi^{-k}x_2/x_1}{(1-\xi^{-k}x_2/x_1)^2},\\
    &\wh{\mathcal{S}}(x_2,x_1)(e_{\pm\al_j}\ot\al_i)\\
  =\ & e_{\pm\al_j}\ot\al_i\mp\half e_{\pm\al_j}\ot\vac\sum_{k\in\Z_N}
  \(q^{a_{\mu^k(i),j}x_2\partial_{x_2}}-q^{-a_{\mu^k(i),j}x_2\partial_{x_2}}\)
    \frac{\xi^{-k}x_2/x_1+1}{\xi^{-k}x_2/x_1-1},\\
  &\wh{\mathcal{S}}(x_2,x_1)(e_{\pm\al_j}\ot e_{\pm\al_i})
  =  g_{i,j}(x_2/x_1)(e_{\pm\al_j}\ot e_{\pm\al_i}),\\
     &\wh{\mathcal{S}}(x_2,x_1)(e_{\pm\al_j}\ot e_{\mp\al_i})
  =  g_{i,j}(x_2/x_1)\inv (e_{\pm\al_j}\ot e_{\mp\al_i}).
   \end{align*}
\end{lem}

The following are analogues of relations ($Q2'$) and ($Q3'$):

\begin{prop}\label{lem:Y-eta-rel-12}
The following relations hold on $W$ for $i,j\in I$:
\begin{align}
  &[Y_W(\al_i,x_1),Y_W(\al_j,x_2)]\label{eq:Y-eta-rel1}\\
  &=\ \sum_{k\in\Z_N}[a_{\mu^k(i), j}]_{q^{x_2\partial_{x_2}}}
    \(\iota_{x_1,x_2}q^{-x_2\partial_{x_2}}-\iota_{x_2,x_1}q^{x_2\partial_{x_2}}\)
    \frac{\xi^{-k}x_2/x_1}{(1-\xi^{-k}x_2/x_1)^2},\nonumber\\
  &[Y_W(\al_i,x_1),Y_W(e_{\pm\al_j},x_2)]=\pm Y_W(e_{\pm\al_j},x_2)\label{eq:Y-eta-rel2}\\
 &\quad\times \sum_{k\in\Z_N}[a_{\mu^k(i), j}]_{q^{x_2\partial_{x_2}}}
   \(\iota_{x_1,x_2}q^{-x_2\partial_{x_2}}-\iota_{x_2,x_1}q^{x_2\partial_{x_2}}\)
    \half\(\frac{1+\xi^{-k}x_2/x_1}{1-\xi^{-k}x_2/x_1}\).\nonumber
\end{align}
\end{prop}

\begin{proof} Note that (\ref{eq:Y-eta-rel1}) amounts to
\begin{align}
&Y_W(\al_i,x_1)Y_W(\al_j,x_2)-Y_W(\al_j,x_2)Y_W(\al_i,x_1)\label{Y-f-eta-alpha-i-j-new}\\
 &\quad \quad -\sum_{k\in\Z_N}[a_{\mu^k(i), j}]_{q^{x_2\partial_{x_2}}}
  (q^{-x_2\partial_{x_2}}-q^{x_2\partial_{x_2}}) \iota_{x_2,x_1} \frac{\xi^{-k}x_2/x_1}{(1-\xi^{-k}x_2/x_1)^2}\nonumber\\
=\ &\sum_{k\in\Z_N}[a_{\mu^k(i), j}]_{q^{x_2\partial_{x_2}}}
q^{-x_2\partial_{x_2}}(\iota_{x_1,x_2}-\iota_{x_2,x_1}) \frac{\xi^{-k}x_2/x_1}{(1-\xi^{-k}x_2/x_1)^2}\nonumber\\
=\ &\sum_{k\in\Z_N}[a_{\mu^k(i), j}]_{q^{x_2\partial_{x_2}}}
q^{-x_2\partial_{x_2}}x_2\partial_{x_2}\delta\left(\xi^{-k}\frac{x_2}{x_1}\right).\nonumber
\end{align}
Using Proposition \ref{VL-eta-S(X)-special} we have
\begin{align*}
&Y_L^\eta(\al_i,x_1)Y_L^\eta(\al_j,x_2)\nonumber\\
 &\  -Y_L^\eta(\al_j,x_2)Y_L^\eta(\al_i,x_1)-\sum_{k\in\Z_N}[a_{\mu^k(i), j}]_{q^{\pd{x_2}}}
  (q^{-\pd{x_2}}-q^{\pd{x_2}}) \iota_{x_2,x_1} \frac{\xi^{-k}e^{x_2-x_1}}{(1-\xi^{-k}e^{x_2-x_1})^2}\nonumber\\
   =\ &\Res_z Y_L^{\eta}\(Y_{L}^{\eta}(\al_i,z)\al_j,x_2\)e^{z\partial_{x_2}}x_1^{-1} \delta\left(\frac{x_2}{x_1}\right),
\end{align*}
noticing that for any $P(t)\in \C[t][[\hbar]],\ g(x)\in \C((x))$,
$$(P(\partial_x)g(x))|_{x=x_2-x_1}=P(\partial_{x_2})g(x_2-x_1).$$
From Lemma \ref{lem:Y-eta-Sing}, for $k\in \Z_N$ we have
$$Y_L^\eta(\mu^k\al_i,x)^{-}\al_j=Y_L^\eta(\al_{\mu^k (i)},x)^{-}\al_j
=[a_{\mu^k(i),j}]_{q^{\partial_x}}q^{\partial_x}x^{-2}{\bf 1}.$$
Using Proposition \ref{prop:tech-reverse-calculations-h} and this relation, we obtain (\ref{Y-f-eta-alpha-i-j-new}) as
\begin{align*}
&Y_W(\al_i,x_1)Y_W(\al_j,x_2)-Y_W(\al_j,x_2)Y_W(\al_i,x_1)\nonumber\\
 &\quad \quad -\sum_{k\in\Z_N}[a_{\mu^k(i), j}]_{q^{x_2\partial_{x_2}}}
  (q^{-x_2\partial_{x_2}}-q^{x_2\partial_{x_2}}) \iota_{x_2,x_1} \frac{\xi^{-k}x_2/x_1}{(1-\xi^{-k}x_2/x_1)^2}\nonumber\\
=\ &\sum_{k\in\Z_N}\Res_z Y_W\(Y_L^{\eta}(\mu^k\al_i,z)\al_j,x_2\)e^{zx_2\partial_{x_2}}\delta\left(\xi^{-k}\frac{x_2}{x_1}\right)\nonumber\\
=\ &\sum_{k\in \Z_N} \Res_z \([a_{\mu^k(i),j}]_{q^{\partial_z}}q^{\partial_z}z^{-2}\)
e^{zx_2\partial_{x_2}}\delta\left(\xi^{-k}\frac{x_2}{x_1}\right)\\
=\ & \sum_{k\in \Z_N} [a_{\mu^k(i),j}]_{q^{-x_2\partial_{x_2}}}q^{-x_2\partial_{x_2}}x_2\partial_{x_2}
\delta\left(\xi^{-k}\frac{x_2}{x_1}\right),
\end{align*}
where we also use the fact that for any $P(t)\in \C[t][[\hbar]],\ r\in \N$,
\begin{align}\label{simple-fact-F}
\Res_z (P(\partial_z)z^{-r-1})e^{zx}=\frac{1}{r!}x^{r}P(-x).
\end{align}
 Relation (\ref{eq:Y-eta-rel2}) can be proved similarly.
\end{proof}

For  $i\in I$, write $Y_W(\al_i,x)=\sum_{n\in\Z}\al_{i}^\phi(n)x^{-n}$ on $W$, and set
\begin{align}
Y_W(\al_i,x)^{\pm}=\sum_{n\ge 1}\al_{i}^\phi(\mp n)x^{\pm n}+\half \al_{i}^\phi(0).
\end{align}

From \eqref{eq:Y-eta-rel1} we readily have:

\begin{coro}\label{lem:al-negative-stay}
The following relations hold on $W$ for $i,j\in I,\ m,n\in \N$:
\begin{align}
[\al_{i}^\phi(m),\al_{j}^\phi(n)]=0=[\al_{i}^\phi(-m),\al_{j}^\phi(-n)].
\end{align}
In particular, $[\al_{i}^\phi(m),\al_{j}^\phi(0)]=0$ for all $m\in \Z$.
\end{coro}


We need the following two technical results:

\begin{lem}\label{lem:exp-cal}
Let $W$ be a topologically free $\C[[\hbar]]$-module and let
\begin{align}\label{eq:exp-cal-def}
  \al(x)\in &\Hom (W,W\wh\ot \C[x,x\inv][[\hbar]]),\ \
  \beta(x)\in(\End W)[[x]],\ \  \gamma(x)\in \C(x)[[\hbar]],
\end{align}
satisfying the condition that
\begin{align}
  &[\al(x_1),\al(x_2)]=0=[\beta(x_1),\beta(x_2)],\label{eq:exp-cal-cond1}\\
  &  [\al(x_1),\beta(x_2)]=\iota_{x_1,x_2}\gamma(x_2/x_1).\label{eq:exp-cal-cond2}
\end{align}
Assume $\al(x),\be(x)\in \hbar (\End W)[[x,x^{-1}]]$. Set $E_{\gamma}=\Res_zz\inv \gamma(e^{-z})\in \hbar\C[[\hbar]]$.
Then
\begin{align*}
 \exp\(\(\al(x)+\beta(x)\)_{-1}^\phi\)1_W=\exp \al(x)\exp \be(x)\exp\(\frac{1}{2}E_\gamma\).
\end{align*}
\end{lem}

\begin{proof} It follows from \eqref{eq:exp-cal-cond1} and \eqref{eq:exp-cal-cond2} that
$U:=\{\al(x),\beta(x)\}$ is an $\hbar$-adically quasi-compatible subset of $\E_\hbar(W)$.
By Theorem \ref{thm:abs-construct} we have an $\hbar$-adic nonlocal vertex algebra $\<U\>_\phi$.
With \eqref{eq:exp-cal-cond2}, from Proposition \ref{prop:tech-calculations} we get
\begin{align*}
  [Y_\E^\phi(\al(x),x_1),Y_\E^\phi(\beta(x),x_2)]=\iota_{x_1,x_2}\gamma(e^{x_2-x_1})
  \  \ \(=e^{-x_2\frac{\partial}{\partial x_1}} \gamma(e^{-x_1})\in \C((x_1))[[x_2,\hbar]]\),
\end{align*}
which in particular implies
\begin{align}\label{eq:exp-cal-temp1}
  [\al(x)^{\phi}_{-1},\beta(x)^{\phi}_{-1}]=E_{\gamma}.
\end{align}
By the Baker-Campbell-Hausdorff formula we obtain
\begin{align*}
  \exp\((\al(x)+\beta(x))^{\phi}_{-1}\)
  =\exp\(\beta(x)^{\phi}_{-1}\)\exp\(\al(x)^{\phi}_{-1}\)\exp\(\frac{1}{2} E_{\gamma}\).
  \end{align*}
On the other hand, from  \eqref{eq:exp-cal-def}, we have
\begin{align*}
  \al(x_1)\al(x_2)^n,\,\beta(x_1)u(x_2)\in \E_\hbar^{(2)}(W)\quad\te{for }n\in\N,\,\,u(x)\in \<U\>_\phi.
\end{align*}
Then from the definition of $Y_\E^\phi$ (see Definition \ref{de:Y-E-phi}) we get
\begin{align}
  &\al(x)^{\phi}_{-1}\al(x)^n=\al(x)^{n+1}\quad \te{for }n\in\N,\label{eq:exp-cal-temp2}\\
  &\beta(x)^{\phi}_{-1}u(x)=\beta(x)u(x)\quad\te{for }u(x)\in \<U\>_\phi.\label{eq:exp-cal-temp3}
\end{align}
Using these relations we obtain
\begin{align*}
  &\exp\((\al(x)+\beta(x))^{\phi}_{-1}\)1_W\\
  =&\exp\(\beta(x)^{\phi}_{-1}\)\exp\(\al(x)^{\phi}_{-1}\)\exp\(\half E_\gamma\)1_W\\
  =&\exp\(\beta(x)^{\phi}_{-1}\)\exp\(\al(x)\)\exp\(\half E_\gamma\)\\
  =&\exp\(\beta(x)\)\exp\(\al(x)\)\exp\(\half E_\gamma\),
\end{align*}
as desired.
\end{proof}

\begin{prop}\label{lem:exp-al-no}
For $i\in I,\  a\in\C$, we have
\begin{align}\label{Y-T-f-Ai-expression}
  Y_W(A_i&(a\hbar)^{\pm 1}\vac,x)=
  \frac{\kappa_{i,i}(0)}{\kappa_{i,i}(a\hbar)^\half \kappa_{i,i}(-a\hbar)^\half}
  \Lambda_i^+(a,x)^{\pm 1}
  \Lambda_i^-(a,x)^{\pm 1},
\end{align}
where
\begin{align}
  \Lambda_i^\pm(a,x)=q^{\half a\al_i^\phi(0)}\exp\( \sum_{n\ge 1}
    \frac{\al_{i}^\phi(\mp n)}{ \pm n}\( q^{\pm an}-1 \)x^{\pm n}\).
\end{align}
\end{prop}

\begin{proof}  Recall that
$$A_i(a\hbar)^{\pm 1}=\exp\left(\pm \sum_{n\ge 1} \frac{\al_i(-n)}{n}(a\hbar)^n\right)$$
on $V_{L}[[\hbar]]^{\eta}\ (=V_{L}[[\hbar]])$.
For $\al\in \h$, we have $Y_L^{\eta}(\al,x)=\sum_{n\in \Z}\al^{\eta}(n)x^{-n-1}$,  where
$Y_{L}^{\eta}(\al,x)=Y(\al,x)+\Phi(\eta'(\al,x))$.
Noticing that $\Phi(\eta'(\al,x)){\bf 1}=0$, we get
\begin{align}
\al^{\eta}(m){\bf 1}=\al_{m}{\bf 1}=\al(m){\bf 1}\quad \text{ for } m\in \Z.
\end{align}
Set
\begin{align}
  \wt\al_i=\sum_{n\ge 1}\frac{1}{n}(a\hbar)^n \al_i(-n)\vac=\sum_{n\ge 1}\frac{1}{n}(a\hbar)^n \al_i^{\eta}(-n)\vac
 \in \hbar V_L[[\hbar]].
\end{align}
Note that
\begin{align}
(\wt\al_i)^{\eta}_{-1}:=\Res_x x^{-1}Y^{\eta}_{L}(\wt\al_i,x)=\sum_{n\ge 1}\frac{1}{n}(a\hbar)^n \al_i^{\eta}(-n).
\end{align}
(A general fact is that  for  $n\in \Z_+$, $(v_{-n}{\bf 1})_{-1}=v_{-n}$ for any vector $v$.)
Thus
\begin{align}
A_i(a\hbar)^{\pm 1}\vac=\exp\(\pm (\wt\al_i)^{\eta}_{-1}\)\vac.
\end{align}
Denoting by ${\mathcal{D}}_{\eta}$ the $D$-operator of $V_{L}[[\hbar]]^{\eta}$, we have
\begin{align}\label{wt-al-i-expression}
 \wt\al_i =\sum_{n\ge0}\frac{(a\hbar)^{n+1}}{(n+1)!}{\mathcal{D}}_{\eta}^n\al_i\in \hbar V_{L}[[\hbar]]^{\eta}.
 \end{align}

On the other hand, write $Y_W(\wt\al_i,x)=\sum_{n\in\Z}\wt\al_{i}^\phi(n)x^{-n}$
and set
\begin{align}
 Y_W(\wt\al_i,x)^\pm=\sum_{n\ge 1}\wt\al_{i}^\phi(\mp n)x^{\pm n}+\half \wt\al_{i}^\phi(0).
\end{align}
As $Y_W({\mathcal{D}}_{\eta}v,x)=x\partial_{x}Y_W(v,x)$ for $v\in V_L[[\hbar]]^{\eta}$,
 by (\ref{wt-al-i-expression}) and \eqref{eq:Y-eta-rel1} we have
\begin{align*}
  &[Y_W(\wt\al_i,x_1),Y_W(\wt\al_i,x_2)]\\
  =\ & \sum_{m,n\ge0}\frac{(a\hbar)^{m+n+2}}{(m+1)!(n+1)!}
  \partial_{(x_1)}^m\partial_{(x_2)}^n
  [Y_W(\al_i,x_1),Y_W(\al_i,x_2)]\\
  =\ &\sum_{m,n\ge0}\frac{(a\hbar)^{m+n+2}}{(m+1)!(n+1)!}
  \partial_{(x_1)}^{m+1}\partial_{(x_2)}^{n+1}\nonumber\\
  &\quad \times \sum_{k\in\Z_N}
  [a_{\mu^k(i),i}]_{q^{x_2\partial_{x_2}}}  \(\iota_{x_1,x_2}q^{x_2\partial_{x_2}}-\iota_{x_2,x_1}q^{-x_2\partial_{x_2}}\)\log(1-\xi^{-k}x_2/x_1)\\
  = \ &\sum_{k\in\Z_N}
  [a_{\mu^k(i),i}]_{q^{x_2\partial_{x_2}}}  \(\iota_{x_1,x_2}q^{x_2\partial_{x_2}}-\iota_{x_2,x_1}q^{-x_2\partial_{x_2}}\)\\
  &\quad \times  \(q^{-ax_2\partial_{x_2}}-1\)\(q^{ax_2\partial_{x_2}}-1\) \log(1-\xi^{-k}x_2/x_1).
\end{align*}
From this we get
\begin{align*}
  [Y_W(\wt\al_i,x_1)^\pm,Y_W(\wt\al_i,x_2)^\pm]=0
\end{align*}
and
\begin{align*}
  &[Y_W(\wt\al_i,x_1)^-,Y_W(\wt\al_i,x_2)^+]\\
 =&\sum_{k\in\Z_N}[a_{\mu^k(i),i}]_{q^{x_2\partial_{x_2}}}
    q^{-x_2\partial_{x_2}}
   \(q^{-ax_2\partial_{x_2}}-1\)\(q^{ax_2\partial_{x_2}}-1\)
    \log (1-\xi^{-k}x_2/x_1).
\end{align*}

Set
\begin{align*}
  \lambda(x)=\sum_{k\in\Z_N}[a_{\mu^k(i),i}]_{q^{x\partial_{x}}}
    q^{-x\partial_{x}}
    (q^{-ax\partial_{x}}-1)(q^{ax\partial_{x}}-1)
    \log(1-\xi^{-k}x),
\end{align*}
which lies in $\C(x)[[\hbar]]$ as $x\partial_{x} \log(1-\xi^{-k}x)=\frac{-\xi^{-k}x}{1-\xi^{-k}x}$.
Noticing that $$(x\partial_x)^2\log(1-\xi^{-k}x)=-\frac{\xi^{-k}x}{(1-\xi^{-k}x)^2},\quad
\partial_x^2 (\vartheta_k(x)+\delta_{k,0}\log x)=-\frac{\xi^{-k}e^x}{(1-\xi^{-k}e^x)^2}, $$
and $\(\partial_{x}^n\log x\)|_{x=-z}=(-\partial_z)^n\log z$ for $n\ge 1$, we have
\begin{align*}
 & \lambda(e^{x})
  =\sum_{k\in\Z_N}[a_{\mu^k(i),i}]_{q^{\partial_{x}}}
    q^{-\partial_{x}}
    \(q^{-a\partial_{x}}-1\)\(q^{a\partial_{x}}-1\)(\vartheta_k(x)+\delta_{k,0}\log x),\\
    & \lambda(e^{-x})=\sum_{k\in\Z_N}[a_{\mu^k(i),i}]_{q^{\partial_{x}}}
    q^{\partial_{x}}
    \(q^{a\partial_{x}}-1\)\(q^{-a\partial_{x}}-1\)(\vartheta_k(-x)+\delta_{k,0}\log x).
\end{align*}
Furthermore, since $ \Res_xx\inv\left(\frac{d}{dx}\right)^n\log x=0$ for $n\ge 1$, we get
\begin{align*}
  \Res_xx\inv \lambda(e^{-x})=& \Res_xx\inv\sum_{k\in\Z_N}[a_{\mu^k(i),i}]_{q^{\partial_{x}}}
    q^{\partial_{x}}
    \(q^{a\partial_{x}}-1\)\(q^{-a\partial_{x}}-1\)\vartheta_k(-x)\\
   =\ &  \lim_{x\rightarrow 0}\sum_{k\in\Z_N}[a_{\mu^{-k}(i),i}]_{q^{\partial_{x}}}
    q^{\partial_{x}}
    \(q^{a\partial_{x}}-1\)\(q^{-a\partial_{x}}-1\)\vartheta_{k}(x)\\
   =\ &  \lim_{x\rightarrow 0}\sum_{k\in\Z_N}[a_{i,\mu^k(i)}]_{q^{\partial_{x}}}
    q^{\partial_{x}}
    \(2-q^{a\partial_{x}}-q^{-a\partial_{x}}\)\vartheta_{k}(x)\\
  =\ &  \lim_{x\rightarrow 0}\sum_{k\in\Z_N}[a_{\mu^k(i),i}]_{q^{\partial_{x}}}
    q^{\partial_{x}}
    \(2\vartheta_{k}(x)-\vartheta_{k}(x+a\hbar)-\vartheta_{k}(x-a\hbar)\)\\
  =\ &  \lim_{x\rightarrow 0}\log \( \frac{\kappa_{i,i}(x)^2}{\kappa_{i,i}(x+a\hbar)\kappa_{i,i}(x-a\hbar)}\).
    \end{align*}
    Thus
  $$\exp\left(\Res_xx\inv \lambda(e^{-x})\right)
  =\frac{\kappa_{i,i}(0)^2}{\kappa_{i,i}(a\hbar)\kappa_{i,i}(-a\hbar)}.$$
Then, using Lemma \ref{lem:exp-cal} and the homomorphism property of $Y_W(\cdot,x)$, we get
\begin{align}\label{Y-T-f-A-i}
  & Y_W(A_i(a\hbar)^{\pm 1}\vac,x)
  = Y_W\(\exp((\wt\al_i)^{\eta}_{-1} )\vac,x\)
  = \exp\( Y_W(\wt\al_i,x)_{-1}^{\phi}\)1_W  \nonumber    \\
  &\quad  \quad  =  \frac{\kappa_{i,i}(0)}{\kappa_{i,i}(a\hbar)^\half \kappa_{i,i}(-a\hbar)^\half}
  \exp\(\pm Y_W(\wt\al_i,x)^+\)
  \exp\(\pm Y_W(\wt\al_i,x)^-\).
\end{align}

Noticing that
\begin{align*}
  Y_W(\wt\al_i,x) =\ &\sum_{m\ge0}\frac{(a\hbar)^{m+1}}{(m+1)!}\partial_{(x)}^mY_W(\al_i,x)
  = \al_{i}^\phi(0)a\hbar+\(q^{ax\partial_{x}}-1\)\sum_{n\ne 0}\frac{\al_{i}^\phi(n)}{-n}x^{-n}\\
 =\ &\al_{i}^\phi(0)a\hbar-\sum_{n\ne 0}\frac{\al_{i}^{\phi}(n)}{n}\(q^{-an}-1\)x^{-n},
\end{align*}
we have
\begin{align*}
  Y_W(\wt\al_i,x)^{\pm}=\al_i^\phi(0)\frac{a\hbar}{2}\pm \sum_{n\ge 1}
    \frac{\al_i^\phi(\mp n)}{n}(q^{\pm an}-1)x^{\pm n}.
\end{align*}
Finally, combining this with (\ref{Y-T-f-A-i}) we obtain (\ref{Y-T-f-Ai-expression}).
\end{proof}

We also have:

\begin{prop}\label{lem:Y-eta-rel-3}
For $i,j\in I$, the following relation holds on $W$:
\begin{align*}
  Y_W&(e_{\al_i},x_1)Y_W(e_{-\al_j},x_2)
  -g_{j,i}(x_1/x_2)Y_W(e_{-\al_j},x_2)Y_W(e_{\al_i},x_1)
  =\frac{1}{2\hbar}\kappa_{i,i}(0)^{-1}\\
  \times&\sum_{k\in\Z_N}\delta_{\mu^k(i),j}
  \( \delta\(\frac{\xi^{-k}x_2}{x_1}\)
  -\frac{\kappa_{j,j}(0)^\half}{\kappa_{j,j}(2\hbar)^\half}
  \Lambda_j^+(-2,x_2)\Lambda_j^-(-2,x_2)
  \delta\(\frac{\xi^{-k}q^{-2}x_2}{x_1}\)
  \).\nonumber
\end{align*}
In particular, the right-hand side vanishes if $j\notin \mathcal{O}(i)$.
\end{prop}

\begin{proof} Notice that for $r,s\in I$, $\<\alpha_r,-\alpha_s\><0$ if and only if $r=s$.
Using Proposition \ref{prop:tech-reverse-calculations-h} and Lemma \ref{lem:wh-S}, and
then using Proposition \ref{lem:exp-al-no} we get
\begin{align*}
  &Y_W(e_{\al_i},x_1)Y_W(e_{-\al_j},x_2)
  -g_{j,i}(x_1/x_2)Y_W(e_{-\al_j},x_2)Y_W(e_{\al_i},x_1)\nonumber\\
  =\ & \Res_z \sum_{k\in \Z_N}Y_W\(Y_L^{\eta}(\mu^k e_{\al_i},z)e_{-\al_j},x_2\)
  e^{zx_2\partial_{x_2}}\delta\left(\xi^{-k}\frac{x_2}{x_1}\right)\nonumber\\
  =\ & \Res_z \sum_{k\in \Z_N}Y_W\(Y_L^{\eta}(e_{\al_{\mu^k(i)}},z)^{-}e_{-\al_j},x_2\)
  e^{zx_2\partial_{x_2}}\delta\left(\xi^{-k}\frac{x_2}{x_1}\right)\nonumber\\
  =\ & \Res_z \sum_{k\in \Z_N}\delta_{\mu^k(i),j}
  \frac{1}{2\hbar}Y_W\(\kappa_{j,j}(0)^{-1}z^{-1}{\bf 1}-(z+2\hbar)^{-1}\kappa_{j,j}(-2\hbar)^{-1}A_{j}^{+}(-2\hbar){\bf 1},x_2\)\nonumber\\
  &\quad \times e^{zx_2\partial_{x_2}}\delta\left(\xi^{-k}\frac{x_2}{x_1}\right)\nonumber\\
  =\ &\sum_{k\in \Z_N}\delta_{\mu^k(i),j}
  \frac{1}{2\hbar}\kappa_{j,j}(0)^{-1}\nonumber\\
  &\ \times \(\delta\left(\xi^{-k}\frac{x_2}{x_1}\right)
      -\frac{\kappa_{j,j}(0)}{\kappa_{j,j}(-2\hbar)}Y_W(A_{j}(-2\hbar){\bf 1},x_2)
  e^{-2\hbar x_2\partial_{x_2}}\delta\left(\xi^{-k}\frac{x_2}{x_1}\right) \)\nonumber\\
  =\ &\sum_{k\in \Z_N}\delta_{\mu^k(i),j}
  \frac{1}{2\hbar}\kappa_{j,j}(0)^{-1}\nonumber\\
 \times&\sum_{k\in\Z_N}\delta_{\mu^k(i),j}
  \( \delta\(\frac{\xi^{-k}x_2}{x_1}\)
  -\frac{\kappa_{j,j}(0)^\half}{\kappa_{j,j}(2\hbar)^\half}
  \Lambda_j^+(-2,x_2)\Lambda_j^-(-2,x_2)
  \delta\(\frac{\xi^{-k}q^{-2}x_2}{x_1}\)
  \),\nonumber
  \end{align*}
 where for the last equality, we also use the relation
 $\kappa_{j,j}(-2\hbar)=\kappa_{j,j}(0)$ from Lemma \ref{lem:kappa-ij-com}.
\end{proof}

\begin{prop}\label{lem:Y-eta-rel-4}
The following relations hold on $W$ for $i,j\in I$:
\begin{align*}
  F_{i,j}(x_1,x_2)Y_W(e_{\pm\al_i},x_1)Y_W(e_{\pm\al_j},x_2)
  =G_{i,j}(x_1,x_2)Y_W(e_{\pm\al_j},x_2)Y_W(e_{\pm\al_i},x_1).
\end{align*}
\end{prop}

\begin{proof} Note that for $r,s\in I$, $\<\alpha_r,\alpha_s\><0$ if and only if $a_{r,s}=-1$.
Using Proposition \ref{prop:tech-reverse-calculations-h}, Lemma \ref{lem:wh-S}, and
then using Lemma \ref{lem:Y-eta-Sing} we get
\begin{align}\label{eq:e-e=-1-commutator}
  &Y_W(e_{\pm\al_i},x_1)Y_W(e_{\pm\al_j},x_2)-g_{j,i}(x_1/x_2)^{-1}Y_W(e_{\pm\al_j},x_2)Y_W(e_{\pm\al_i},x_1)\\
   =\  &\Res_z\sum_{k\in\Z_N}Y_W\(Y_L^\eta(\mu^ke_{\pm\al_i},z)e_{\pm\al_j},x_2\)
  e^{zx_2\partial_{x_2}} \delta\(\frac{\xi^{-k}x_2}{x_1}\)\nonumber\\
   =\  &\Res_z\sum_{k\in\Z_N}Y_W\(Y_L^\eta(e_{\pm\al_{\mu^k(i)}},z)^{-}e_{\pm\al_j},x_2\)
  e^{zx_2\partial_{x_2}} \delta\(\frac{\xi^{-k}x_2}{x_1}\)\nonumber\\
   =\  &\sum_{\substack{k\in\Z_N\\ a_{\mu^k(i),j}=-1}}\epsilon(\al_{\mu^k(i)},\al_j)
  \kappa_{\mu^k(i),j}(-\hbar)Y_W\(A_{\mu^k(i)}(-\hbar)^{\pm 1}e_{\pm(\al_{\mu^k(i)}+\al_j)},x_2\)\nonumber\\
   &\quad\quad  \times \Res_z (z+\hbar)^{-1}
  e^{zx_2\partial_{x_2}} \delta\(\frac{\xi^{-k}x_2}{x_1}\)\nonumber\\
  =\  &\sum_{\substack{k\in\Z_N\\ a_{\mu^k(i),j}=-1}}\epsilon(\al_{\mu^k(i)},\al_j)
  \kappa_{\mu^k(i),j}(-\hbar)\nonumber\\
   &\quad\quad  \times Y_W\(A_{\mu^k(i)}(-\hbar)^{\pm 1}e_{\pm(\al_{\mu^k(i)}+\al_j)},x_2\)
    \delta\(\frac{\xi^{-k}q\inv x_2}{x_1}\).\nonumber
\end{align}
Recall that
$g_{j,i}(x_1/x_2)^{-1}=\iota_{x_2,x_1}\tilde{g}_{i,j}(x_2/x_1)=\iota_{x_2,x_1}(G_{i,j}^+(x_1,x_2)/F_{i,j}^+(x_1,x_2))$ and
\begin{align*}
 F_{i,j}^+(x_1,x_2)=\prod_{\substack{k\in\Z_N\\ \mu^k(i)=j}}(x_1-\xi^{-k}q^{2}x_2)
  \prod_{\substack{k\in\Z_N\\ a_{\mu^k(i),j}=-1}}(x_1-\xi^{-k}q\inv x_2).
\end{align*}
Then multiplying both sides of  \eqref{eq:e-e=-1-commutator} by $F_{i,j}(x_1,x_2)\ (=F_{i,j}^+(x_1,x_2))$,
we obtain the desired relation.
\end{proof}

Now, we use Proposition \ref{prop:Q6'=Q7'} to complete the proof of Theorem \ref{final-theorem}.
First, set
$$h_{i,q}(x)=Y_W(\al_i,x)+\frac{1}{4\hbar}\log \frac{\kappa_{i,i}(0)}{\kappa_{i,i}(2\hbar)},\   \
e_{\pm\al_i,q}(x)=B_iY_W(e_{\pm\al_i},x)\ \ \text{ for }i\in I.$$
Second, use the expansion $h_{i,q}(x)=\sum_{n\in \Z}h_{i,q}(n)x^{-n}$ to define
$$h_{i,q}^{\pm}(x)=\frac{1}{2}h_{i,q}(0)+ \sum_{n\ge 1}h_{i,q}(\mp n)x^{\pm n}.$$
Notice that $h_{i,q}(0)=\al_i^{\phi}(0)+\frac{1}{4\hbar}\log \frac{\kappa_{i,i}(0)}{\kappa_{i,i}(2\hbar)}$.
Third, use (\ref{Psi-h})  to define $\Psi_i^{\pm}(x)$ and then use (\ref{E-e}) to define $E_i^{\pm}(x)$.
By Proposition \ref{lem:Y-eta-rel-12}, relations (Q2$'$) and (Q3$'$) hold.
Note that
\begin{align*}
&\Psi_i^{+}(xq^{-\frac{1}{2}c})^{-1}=\exp\(-2\hbar G(\hbar x\partial_{x})h_{i,q}^{+}(xq^{-c})\),\\
&\Psi_i^{-}(xq^{-\frac{3}{2}c})=\exp\(-2\hbar G(\hbar x\partial_{x})h_{i,q}^{-}(xq^{-c})\).
\end{align*}
On the other hand, we have
\begin{align}
\Lambda_i^{\pm}(-2,x)&=q^{-\al_i^{\phi}(0)}\exp\(\sum_{n\ge 1}\frac{\al_i^{\phi}(\mp n)}{\pm n}(q^{\mp 2n}-1)x^{\pm n}\)\\
&=\exp\( \frac{q^{-2 x\partial_{x}}-1}{x\partial_{x}}\(\frac{1}{2}\al_i^{\phi}(0)+\sum_{n\ge 1}\al_i^{\phi}(\mp n)x^{\pm n}\)\)\nonumber\\
&=\exp\(2\hbar \frac{q^{-x\partial_{x}}-q^{x\partial_{x}}}{2\hbar x\partial_{x}}q^{-x\partial_{x}}h_{i,q}^{\pm}(x) \)
\exp\(\frac{1}{4}\log \frac{\kappa_{i,i}(0)}{\kappa_{i,i}(2\hbar)}\)\nonumber\\
&=\(\frac{\kappa_{i,i}(0)}{\kappa_{i,i}(2\hbar)}\)^{\frac{1}{4}}\exp\(-2\hbar G(\hbar x\partial_{x})h_{i,q}^{\pm}(xq^{-1})\).\nonumber
\end{align}
Thus
\begin{align*}
\Lambda_i^{+}(-2,x)=\(\frac{\kappa_{i,i}(0)}{\kappa_{i,i}(2\hbar)}\)^{\frac{1}{4}}\Psi_i^{+}(xq^{-\frac{1}{2}})^{-1},
\quad \Lambda_i^{-}(-2,x)=\(\frac{\kappa_{i,i}(0)}{\kappa_{i,i}(2\hbar)}\)^{\frac{1}{4}}\Psi_i^{-}(xq^{-\frac{3}{2}})
\end{align*}
(with $c=1$).
Note that if $j\in \mathcal{O}(i)$, then $\kappa_{j,j}(x)=\kappa_{i,i}(x)$ and $B_j=B_i$. We also have
$B_i^{2}\cdot \frac{1}{2\hbar}\kappa_{i,i}(0)^{-1}=\frac{1}{q-q^{-1}}$.
Combining these facts with Proposition \ref{lem:Y-eta-rel-3} we get relation (Q4$'$).
Relation (Q5$'$) follows immediately from Proposition \ref{lem:Y-eta-rel-4}.
As $W$ is a $\phi$-coordinated quasi $V_L[[\hbar]]^{\eta}$-module,
the vertex operator map $Y_W(\cdot,x)$ is a homomorphism of $\hbar$-adic nonlocal vertex algebras.
Then using Lemma \ref{lem:Y-eta-rel-5} we get
\begin{align*}
 Y_W(e_{\pm\al_i},x)_0^\phi Y_W(e_{\pm\al_i},x)_0^\phi
 Y_W(e_{\pm\al_j},x)=Y_W(\(e_{\pm\al_i}\)_0^{\eta}\(e_{\pm\al_i}\)_0^{\eta} e_{\pm\al_j},x)=0
\end{align*}
for $i,j\in I$ with $a_{i,j}=-1$.
That is, relation (Q7$'$)  holds.
Finally, we invoke Proposition \ref{prop:Q6'=Q7'} to conclude that
$W$ is a level $1$ restricted $\qtar$-module, completing the proof of Theorem \ref{final-theorem}.

\end{document}